\documentclass[10pt,reqno]{amsart}
\usepackage{amssymb}
\usepackage{amsfonts}
\usepackage{amsthm}
\usepackage{amsmath}
\usepackage[hypertex]{hyperref} 

\newcommand{\R}{{\mathbb{R}}}
\newcommand{\C}{{\mathbb{C}}}

\newcommand{\Z}{{\mathbb{Z}}}
\let\det=\undefined\DeclareMathOperator*{\det}{det}
\let\Re=\undefined\DeclareMathOperator*{\Re}{Re}

\DeclareMathOperator*{\Arg}{Arg}
\DeclareMathOperator*{\dist}{dist}
\DeclareMathOperator*{\diam}{diam}

\DeclareMathOperator*{\wlim}{weak-lim}
\DeclareMathOperator*{\supp}{supp}

\newtheorem{theorem}{Theorem}[section]
\newtheorem{proposition}[theorem]{Proposition}

\newtheorem{prop}[theorem]{Proposition}
\newtheorem{lemma}[theorem]{Lemma}
\newtheorem{corollary}[theorem]{Corollary}
\newtheorem{conjecture}[theorem]{Conjecture}
\theoremstyle{definition}
\newtheorem{definition}[theorem]{Definition}
\theoremstyle{remark}
\newtheorem{remark}[theorem]{Remark}

\newcounter{smalllist}
\newenvironment{SL}{\begin{list}{{\rm(\roman{smalllist})\hss}}{%
\setlength{\topsep}{0mm}\setlength{\parsep}{0mm}\setlength{\itemsep}{0mm}%
\setlength{\labelwidth}{2.0em}\setlength{\itemindent}{2.5em}\setlength{\leftmargin}{0em}\usecounter{smalllist}%
}}{\end{list}}

\newenvironment{CI}{\begin{list}{{\ $\bullet$\ }}{%
\setlength{\topsep}{0mm}\setlength{\parsep}{0mm}\setlength{\itemsep}{0mm}%
\setlength{\labelwidth}{0mm}\setlength{\itemindent}{0mm}\setlength{\leftmargin}{0mm}%
\setlength{\labelsep}{0mm} }}{\end{list}}

\newcommand{\eps}{\varepsilon}

\newcommand{\jp}[1]{\langle{#1}\rangle}
\newcommand{\jpn}{\langle\nabla\rangle}
\newcommand{\qtq}[1]{\quad\text{#1}\quad}
\newcommand{\vtt}{\tilde{\tilde{v}}}
\newcommand{\ulin}{u^{\text{lin}}}
\newcommand{\uolin}{u^{0,\text{lin}}}
\newcommand{\Id}{\mathrm{Id}}

\newcommand{\W}{\text{\rm\L}}

\numberwithin{equation}{section}
\allowdisplaybreaks

\begin{document}

\title[The cubic Klein--Gordon equation in two space dimensions]{Scattering for the cubic Klein--Gordon equation in two space dimensions}
\author{Rowan Killip}
\address{University of California, Los Angeles}
\author{Betsy Stovall}
\address{University of California, Los Angeles}
\author{Monica Visan}
\address{University of California, Los Angeles}

\begin{abstract}
We consider both the defocusing and focusing cubic nonlinear Klein--Gordon equations
$$
u_{tt} - \Delta u + u \pm u^3 =0
$$
in two space dimensions for real-valued initial data $u(0)\in H^1_x$ and $u_t(0)\in L^2_x$.  We show that in the
defocusing case, solutions are global and have finite global $L^4_{t,x}$ spacetime bounds.  In the focusing case, we
characterize the dichotomy between this behaviour and blowup for initial data with energy less than that of the ground
state.

These results rely on analogous statements for the two-dimensional cubic nonlinear Schr\"odinger equation, which are known in the
defocusing case and for spherically-symmetric initial data in the focusing case. Thus, our results are mostly unconditional.

It was previously shown by Nakanishi that spacetime bounds for Klein--Gordon equations imply the same for nonlinear Schr\"odinger equations.
\end{abstract}

\maketitle

\tableofcontents

%
\section{Introduction}\label{S:intro}
%
%

We consider the cubic nonlinear Klein--Gordon equation
\begin{equation}\label{nlkg}
 u_{tt} - \Delta u + u + \mu u^3 =0,
\end{equation}
for real-valued $u:\R_t\times \R^2_x\to\R$.  Here $\mu =\pm 1$ with $\mu=+1$ known as the defocusing equation and $\mu=-1$ as the focusing case.

We consider initial data in the energy space $u(0)\in H^1_x$ and $u_t(0)\in L^2_x$.  This is precisely the set of initial data for which the energy
\begin{equation}\label{E:energy}
E(u):= E\bigl(u(t),u_t(t)\bigr) := \int_{\R^2} \tfrac12 |u_t(t,x)|^2 +\tfrac12 |\nabla u(t,x)|^2 + \tfrac12 |u(t,x)|^2 + \tfrac{\mu}{4} |u(t,x)|^{4} \,dx
\end{equation}
is finite.  The energy is conserved.

\begin{definition}[Solution]\label{D:solution}
We say that a function $u: I \times \R^2 \to \R$ on a non-empty time interval $0\in I \subset \R$ is a \emph{(strong) solution} to
\eqref{nlkg} if $(u,u_t)\in C^0_t (K ;H^1_x\times L^2_x)$, $u\in L_{t,x}^{4}(K \times \R^2)$ for all compact $K \subset I$, and $u$ obeys the Duhamel formula
\begin{equation}\label{old duhamel}
\begin{aligned}
\begin{pmatrix} u(t) \\ u_t(t) \end{pmatrix}
&=\begin{bmatrix} \cos(\jpn t) & \jpn^{-1}\sin(\jpn t) \\ -\jpn\sin(\jpn t) & \cos(\jpn t)\end{bmatrix}
    \begin{pmatrix} u(0) \\ u_t(0) \end{pmatrix} \\
&\quad + \mu \int_0^t \begin{bmatrix} \jpn^{-1}\sin(\jpn (t-s)) \\ \cos(\jpn (t-s))\end{bmatrix} u^3(s)\,ds
\end{aligned}
\end{equation}
for all $t \in I$.  We refer to the interval $I$ as the \emph{lifespan} of $u$. We say that $u$ is a \emph{maximal-lifespan
solution} if the solution cannot be extended to any strictly larger interval. We say that $u$ is a \emph{global solution} if $I=\R$.
\end{definition}

Our main goal is to prove that global strong solutions exist and have finite spacetime norms.  As the equation is energy-subcritical, global well-posedness follows easily in the defocusing case.
In the focusing case, this is known not to be true for arbitrary initial data.  Indeed, explicit counterexamples are known.  Let $Q$ denote the unique positive radial $H^1_x(\R^2)$ solution to
\begin{equation}\label{E:Q defn}
\Delta Q + Q^3 = Q,
\end{equation}
which is known as the \emph{ground state} and is an optimizer in the Gagliardo--Nirenberg inequality.  (See subsection~\ref{SS:elliptic} for more about this.)
Then $u(t,x)=Q(x)$ is a static solution to \eqref{nlkg}.  In particular, it does not have finite global spacetime bounds nor does it scatter to free waves
at future/past infinity.  Moreover, finite-time blowup occurs for a large class of initial data that are slightly larger, including
the case $u(0)=(1+\eps)Q$ and $u_t(0)=0$ for any $\eps>0$.  This is proved by the method of Payne and Sattinger \cite{PayneSattinger}; see Theorem~\ref{T:blowup}
below and Section~\ref{S:blowup} for more details.  Nevertheless, it is believed that $Q$ is the minimal counterexample to the existence of spacetime bounds, in a certain sense.

In the focusing case, it is not appropriate to measure the `size' of the initial data purely in terms of the energy because of the negative sign appearing in front of the potential energy term.
Indeed, the energy of the solution with initial data $u(0)=(1+\eps)Q$ and $u_t(0)=0$ for any $\eps>0$ is strictly less that that of the static solution
$Q$.  For this reason we introduce a second notion of size, namely, the \emph{mass}:
$$
M(u(t)):=\int_{\R^2} |u(t,x)|^2\, dx.
$$
Unlike the energy, this is not conserved.

These considerations, together with our assertions regarding the defocusing case, can be summarized as follows:

\begin{conjecture}[NLKG Conjecture]\label{Conj:NLKG}
Fix $\mu = \pm 1$.  Let $(u_0,u_1) \in H^1_x\times L^2_x$ and in the focusing case assume also that $M(u_0) < M(Q)$ and $E(u_0,u_1) < E(Q)$.  Then there exists a global solution
$u$ to \eqref{nlkg} with initial data $u(0) = u_0$ and $u_t(0) = u_1$.  Moreover, this solution obeys global spacetime bounds
\begin{equation}\label{E:C(E)}
\|u\|_{L^{\infty}_t H^1_x} + \|u_t\|_{L^{\infty}_t L^2_x} + \|u\|_{L^4_{t,x}} \leq C(E(u_0,u_1)).
\end{equation}
As a consequence, the solution scatters both forward and backwards in time, that is, there exist $(u_0^\pm, u_1^\pm)\in H^1_x\times L_x^2$ such that
\begin{equation*}
\begin{pmatrix} u(t) \\ u_t(t) \end{pmatrix}
 - \begin{bmatrix} \cos(\jpn t) & \jpn^{-1}\sin(\jpn t) \\ -\jpn\sin(\jpn t) & \cos(\jpn t)\end{bmatrix}
    \begin{pmatrix} u_0^\pm \\ u_1^\pm \end{pmatrix} \longrightarrow \begin{pmatrix} 0 \\ 0 \end{pmatrix}
\end{equation*}
in $H^1_x\times L^2_x$ as $t\to\pm\infty$.
\end{conjecture}

The main point here is the $L^4_{t,x}$ spacetime bound and concomitant proof of scattering.  Global well-posedness was known previously.  In the defocusing case,
it is a simple consequence of energy conservation --- the equation is energy-subcritical.  In the focusing case, our hypotheses are sufficient to imply control of
the $L^\infty_t (H^1_x\times L^2_x)$ norm and hence also global well-posedness; the key ingredient here is the sharp Gagliardo--Nirenberg inequality.  See the discussion
in subsection~\ref{SS:elliptic}.

Nakanishi \cite{Nakanishi2008} has shown that spacetime bounds for the complex-valued Klein--Gordon equation imply spacetime bounds for the corresponding nonlinear Schr\"o\-dinger equation
(in the mass- and energy-critical settings and all dimensions).  By employing the $X^{s,b}$ arguments we describe in Section~\ref{S:embed}, one may adapt his proof to the case
of the real-valued Klein--Gordon equation discussed in this paper.  The key observation behind these results is that the Klein--Gordon equation degenerates to the Schr\"odinger equation
in the non-relativistic limit.  Thus any solution to NLS can be used to produce a solution to Klein--Gordon by suitable rescaling and other minor modifications.  We caution the reader
that this produces a very narrow subclass of solutions to Klein--Gordon and that the long-time behaviour of the two types of solution is not identical.

To summarize, a resolution to Conjecture~\ref{Conj:NLKG} implies a positive answer to the analogue for NLS, specifically,

\begin{conjecture}[NLS Conjecture]\label{Conj:NLS}
Fix $\mu = \pm 1$.  Let $w_0 \in L^2_x(\R^2)$ and in the focusing case assume also that $M(w_0)<M(Q)$.  Then there exists a unique global solution $w$ to
\begin{align}\label{normal nls}
iw_t+\Delta w= \mu |w|^2w
\end{align}
with $w(0) = w_0$.  Furthermore, this solution satisfies
$$
\|w\|_{L^4_{t,x}(\R \times \R^2)} \leq C(M(w_0))
$$
for some continuous function $C$.  As a consequence, the solution $w$ scatters both forward and backwards in time, that is, there exist $w_\pm\in L_x^2$ such that
$$
\|w(t) -e^{it\Delta} w_\pm\|_{L_x^2} \to 0 \qtq{as} t\to \pm \infty.
$$
\end{conjecture}

This conjecture has been proved in the defocusing case \cite{Dodson:2D} and for spherically symmetric data in the focusing case \cite{KTV}.  (See also \cite{Dodson:3D} and
\cite{KVZ} for these results in higher dimensions in the defocusing and focusing cases, respectively.)  The ideas developed in this paper are not sufficient to resolve the remaining case
of either conjecture.  Nevertheless, we will prove that resolution of Conjecture~\ref{Conj:NLS} is the only obstruction to spacetime bounds for NLKG.  As NLS only represents the non-relativistic limit,
this means providing spacetime bounds for all solutions except those living at astronomic length-scales.  We contend that it is more natural to treat these excluded solutions
in the NLS setting because scale invariance is restored there.

Here at last is the precise statement of the main theorem in this paper.

\begin{theorem}[Spacetime bounds]\label{T:ST bounds}
Fix $\mu = \pm 1$ and assume that Conjecture~\ref{Conj:NLS} holds for this choice.  Then Conjecture~\ref{Conj:NLKG} holds.
\end{theorem}

As mentioned above, Conjecture~\ref{Conj:NLS} has been resolved except for the focusing case with non-radial data.  As a consequence, our results are unconditional in
all but this case:

\begin{corollary}\label{C:ST bounds}
Conjecture~\ref{Conj:NLS} holds in the defocusing case and for spherically symmetric data in the focusing case.
\end{corollary}

We have chosen to present our principal result as Theorem~\ref{T:ST bounds} since it more honestly represents what is achieved in this paper.  It
also reiterates our belief that the natural way to attack dispersive PDE with broken symmetries is first to treat independently the limit cases
where the symmetries are restored.

In Section~\ref{S:blowup}, we show that the hypotheses in the focusing case are sharp in the following sense:

\begin{theorem}[Blowup]\label{T:blowup}
Let $u$ be a maximal-lifespan solution to \eqref{nlkg} in the focusing case with initial data obeying
$$
E(u) < E(Q) \qtq{and} M(u(0)) > M(Q).
$$
Then the solution $u$ blows up in finite time in at least one time direction.
\end{theorem}

Both our well-posedness and blowup results for the focusing Klein--Gordon equation constrain the behaviour of the mass and the energy, while in the NLS setting, the natural
conjecture considers only the mass.  In connection with this, we note that there are Schwartz-space solutions to focusing NLKG with mass that is arbitrarily small uniformly
in time but do not scatter (or admit global spacetime bounds).  For instance, given $\nu\in\R^2$,
$$
  u^{\nu}(t,x) := Q(x^{\perp} + \jp{\nu}x^{\parallel} - \nu t), \qtq{where} x=x^\parallel+x^{\perp} \qtq{and} x^\perp \perp \nu,
$$
is a solution to the focusing case of \eqref{nlkg} with mass $M(u^{\nu}(t)) = \jp{\nu}^{-1}M(Q)$, which can be made arbitrarily small by sending $\nu\to\infty$.  This example arises from the
fact that Lorentz boosts do not preserve the mass.  By comparison, the corresponding Galilei symmetry of NLS does conserve the mass and so this phenomenon does not occur.

The analogue of Conjecture~\ref{Conj:NLKG} for focusing nonlinearities of the form $|u|^pu$ with $p>2$ was resolved recently in \cite{IMN}, the defocusing case having been
treated previously in \cite{Nakanishi:scattering}.  To properly explain the relation between these works, our efforts here, and other work in higher dimensions, we need
to introduce the notion of criticality; the fact that our equation does not possess a scaling symmetry makes this a little more awkward than in related problems such as
the Schr\"odinger and wave equations.

The dispersion relation $-\omega=\jp\xi:=(1+|\xi|^2)^{1/2}$ for the free Klein--Gordon equation has two natural scaling limits: the wave equation at high frequencies and
the Schr\"odinger equation at large length scales.  We deliberately did not refer to low frequencies here because the action of Lorentz boosts means that a solution
characterized by a large length scale may actually be centered around any frequency it wishes.  (As we will see, Lorentz boosts will be a source of particular vexation in
the treatment of \eqref{nlkg}; this is a phenomenon that is peculiar to our nonlinearity as compared to $|u|^pu$ with $p>2$.)

Both the Schr\"odinger and wave scaling limits restore scale invariance and hence define a notion of criticality, indeed, a \emph{common} notion of criticality:
$$
\text{The nonlinearity} \quad \pm |u|^p u \quad \text{is $\dot H^{s_c}_x(\R^d)$ critical precisely when} \quad s_c=\tfrac d2-\tfrac2p.
$$

On the other hand, the energy \eqref{E:energy} controls both $L^2_x$ and $\dot H^1_x$, at least in the defocusing case, and hence all intermediate Sobolev spaces.  In two space dimensions,
$0\leq s_c<1$ corresponds to nonlinearities with $p\geq 2$.  Note that the case $p=2$ is $L^2_x$-critical, while there is no $\dot H^1_x$-critical nonlinearity of power-type in two space dimensions.
In $d>2$ dimensions, $|u|^{\frac 4d}u$ is $L_x^2$-critical and known as the mass-critical case, while $|u|^{\frac4{d-2}}u$ is $\dot H^1_x$-critical and is known as the energy-critical case.
In two dimensions the name energy-critical is used for exponential-type nonlinearities inspired by inequalities of Trudinger--Moser type; see \cite{IMMN, IMN}.

For simplicity, let us begin with some historical remarks in the defocusing case in dimensions $d\geq 3$.  For both the Schr\"odinger and Klein--Gordon equations in the
inter-critical regime $\frac4{d}<p<\frac4{d-2}$ and data in the energy space, scattering was proved some time ago; see \cite{GV:decay} and references therein, as well as
\cite{Nakanishi2001, Matador} for simplified treatments.  The case of energy-critical nonlinearities was first treated in the NLS setting in breakthrough work of Bourgain
\cite{borg:scatter}.  This paper was for radial initial data, the non-radial case being treated in \cite{CKSTT:gwp,RV,Monica:thesis art} and (non-radial) Klein--Gordon in \cite{Nakanishi}.
The $L^2_x$-critical case was resolved recently in \cite{Dodson:3D}; see also \cite{KVZ,TVZ:cc,TVZ:sloth} for earlier work in the radial case.
For the latest on the focusing case in dimensions $d\geq 3$, see \cite{AkahoriNawa,IMN,KenigMerle:H1,Berbec,KVZ}.

For the Klein--Gordon equation in the case $d=2$ discussed in this paper, scattering has been proved for data in the energy space in the inter-critical cases
\cite{Nakanishi:scattering} (defocusing) and \cite{IMN} (focusing), as well as in the energy-critical case \cite{IMMN} (defocusing) and \cite{IMN} (focusing).
The remaining $L^2_x$-critical case is the topic of this paper.

As noted above, scattering in the inter-critical regime (for both the Klein--Gordon and the Schr\"odinger equations) was proved significantly before any progress was made in the critical cases.
This is with good reason: the fact that conservation laws control both higher and lower regularity norms gives excellent control over all scale-invariant quantities.  In the critical cases,
one of the two controls becomes redundant/unusable --- any attempt to incorporate it results in non-scale invariant (and hence patently ridiculous) statements.  Of the two critical cases,
we contend that the $L^2_x$-critical case is more difficult.  Our reasons for making this claim, which go beyond the mere fact that one has been resolved and the other not, will become
apparent when we present the outline of the proof.

\subsection{Outline of the proof}
The key to proving scattering is to show finiteness of a global spacetime norm.  The most natural choice in the case of \eqref{nlkg} is $L^4_{t,x}(\R\times\R^2)$.  This is also the
natural choice for the $L^2_x$-critical NLS and, in particular, is scale-invariant in that setting.  For a general time interval $I\subseteq \R$ we will use the notation
\begin{equation}
 S_I(u) = \bigl\| u \bigr\|_{L^4_{t,x}(I\times\R^2)}^4,
\end{equation}
which we refer to as the scattering size of $u$ on the interval $I$. We will also write $S_{\geq t}(u)$ to denote the scattering size on the the interval $[t,\infty)$ and similarly,
$S_{\leq t}(u)$ for the interval $(-\infty,t]$.  The Strichartz inequality (Lemma~\ref{L:Strichartz}) shows that $S_\R(u)$ is finite for solutions $u$ of the
linear Klein--Gordon equation.

To prove the existence of a global solution $u$ with $S_\R(u)$ finite, we will employ an induction on energy/contradiction argument in the style of Kenig and Merle \cite{KenigMerle:H1}.
(See also \cite{borg:book,keraani-l2} for some key steps in the development of this methodology, as well as the review \cite{ClayNotes}.)  The idea is as follows: If Conjecture~\ref{Conj:NLKG}
were false, there would be a sequence of solutions $u_n:I_n\times\R^2\to\R$ for which
\begin{equation}\label{enemy}
E(u_n) \to E_c < \infty \qtq{but} S_{I_n}(u_n) \to \infty.
\end{equation}
In the focusing case we would also have
$$
\limsup_{n\to\infty} M(u_n(0)) < M(Q) \qtq{and} E_c < E(Q).
$$
Without loss of generality we may choose $E_c$ to be the smallest number for which such a sequence exists, which is then called the \emph{critical energy};
it is positive by virtue of the small-data theory expounded in Section~\ref{S:LT}.  A key observation of Keraani (originally made in the mass-critical NLS setting)
was that from this sequence one can extract a minimal counterexample to the conjecture.

The proof of the existence of a minimal counterexample is not trivial.  Non-compact symmetries provide an obvious means for a minimizing sequence to fail to have a convergent subsequence.
For the Klein--Gordon equation these include space and time translations, as well as Lorentz boosts.  While not a true symmetry, dilations (specifically, to large length scales in our case)
also provide a manner in which the minimizing sequence may fail to converge.  The usual method for studying variational problems with symmetry is the concentration compactness technique and
this is what we will employ here.  Recall from \cite{Lions:IHP} that concentration compactness presents us with three scenarios: \emph{compactness} (the desired outcome)
\emph{vanishing} (the complete lack of concentration) or \emph{dichotomy} (the splitting into two or more wave packets).

The key to disproving vanishing is an inverse Strichartz inequality (see Theorem~\ref{T:InvStrich}), which shows that the scattering size cannot be large without possessing a bubble
of concentration.  When the critical regularity is positive, this type of result can be deduced from an inverse Sobolev embedding inequality.  It was in this manner that concentration compactness
techniques were first introduced in the dispersive setting; see \cite{bahouri-gerard}.  In the $L^2_x$-critical setting, there is no possibility to involve Sobolev embedding and all
known inverse inequalities rely on deep results in harmonic analysis, specifically, on progress toward the Restriction Conjecture of Stein, \cite{Stein:RC}.  (Sections~4.2 and~4.4
of the lecture notes \cite{ClayNotes} discuss inverse Sobolev and Strichartz inequalities, respectively, in a consistent manner and should aid the reader in making comparisons.)

We will use the sharp bilinear restriction theorem of Tao, \cite{taoGAFA03}.  The result is global for the paraboloid (the Schr\"odinger case), but only applies to compact subsets
of other elliptic surfaces (for example $\omega+\jp{\xi}=0$, which is the dispersion relation in our case).  This necessitates several preliminary reductions: first to a frequency
annulus (cf. Lemma~\ref{L:Annular}) and then to a narrow (unit width) sector inside this annulus (cf. Corollary~\ref{C:tube decoupling}).  A Lorentz boost is then used to bring this
tube into a fixed neighbourhood of the origin where Tao's result may then be applied.

The necessity of incorporating Lorentz boosts in the discussion of the inverse Strichartz inequality is (as mentioned
before) a peculiarity of the $L^2_x$-critical case. $L^2_x$-supercriticality of an equation prevents the characteristic
frequency scale of the functions in the minimizing sequence from shrinking to zero. This reasoning also explains why
previous work on scattering in the $L^2_x$-supercritical regime has not had to directly address (or appeal to) the
connection to NLS, namely, because NLS is only revealed when the characteristic frequency scale of a solution shrinks
to zero.  On the other hand, by playing boundedness of the energy against the $\dot H^{1/2}_x$ scaling of Lorentz
boosts, one can deduce that the boost parameters of minimizing sequences must stay bounded (that is, do not approach
the speed of light).  As the speed associated to any boosting of the minimizing sequence is then comparable to the
variation in speed already found in the solution (that is, the variation in $\xi/\jp\xi$ as $\xi$ varies over the Fourier
support), the very existence of a boost becomes moot.

In the $L^2_x$-critical case, there is no lower bound on the characteristic frequency scale of a solution and, in particular, it may be much less than the (still bounded) parameter of a boost.
This is why we must address the action of Lorentz boosts in this paper.

We turn now to a brief discussion of the second unfavourable concentration compactness scenario: dichotomy.  This will be excluded on the basis of the fact that we are dealing
with a \emph{minimizing} sequence (cf. Case~II in Section~\ref{S:min blowup}).  The key tools here are the decoupling results discussed in Section~\ref{S:lpd} and the
stability theory recorded in Section~\ref{S:LT}. Together they show that despite the fact that our equation is nonlinear, multiple wave packets act independently; hence,
if a solution consisting of multiple wave packets has infinite scattering size, then so does one of the constituent wave packets.  Passing to the single wave packet would
then give a smaller value for $E_c$ appearing in \eqref{enemy}.  Contradiction!

On the basis of all we have discussed so far, we find ourselves in the desired (concentration) compactness scenario. In the NLS (or semilinear wave) context this would mean that, after
applying symmetries of the equation to our minimizing sequence, we may exhibit a convergent subsequence and hence a minimal-energy blowup solution.  In our case, however, one wrinkle remains:
the Klein--Gordon equation does not possess scaling symmetry.  As a consequence, the limiting minimal blowup solution may be a solution to NLS instead of NLKG!  Indeed, this occurs whenever
the characteristic length scale of our minimizing sequence diverges to infinity.  (In the energy-subcritical regime, boundedness of the energy prevents divergence to zero.)

Previously, Nakanishi \cite{Nakanishi2008} used this approximate embedding of NLS inside (complex) NLKG to show that the spacetime norm of a solution to the former provides
a \emph{lower bound} on the function $C$ appearing in \eqref{E:C(E)}.  In this way, he proved that spacetime bounds for NLKG imply spacetime bounds for NLS. To use Conjecture~\ref{Conj:NLS}
to prove scattering for NLKG, we need to do what is essentially the exact opposite: transfer \emph{upper bounds} from NLS to NLKG.  This reversal introduces two new aspects.
First, for lower bounds one merely needs to control the quality of the approximation for the amount of time it takes for the NLS solution to accumulate the majority of its
scattering size; for upper bounds, we need to control the NLKG solution globally in time.  In the long-time regime, the quality of the NLS to NLKG approximation deteriorates
to an unsatisfactory degree.  We deal with this issue by noting that by the time the approximation breaks down, the NLS solution has started emulating a solution to the linear
Schr\"odinger equation and that consequently, the future linear Klein--Gordon evolution is small.  This is the topic of Proposition~\ref{P:large times approx};
the same philosophy was employed in \cite{KKSV}, which considers the embedding of NLS inside gKdV.

The second new aspect is that to prove lower bounds, one may embed the NLS solution in whichever manner is convenient.  For upper bounds, we must contend with all embeddings, including
the possible incorporation of a Lorentz boost.

A third difference between this work and \cite{Nakanishi2008} stems from our decision to consider the Klein--Gordon equation for real-valued functions.  This distinction is essentially
irrelevant for most of the arguments, with the notable exception of the NLS/NLKG correspondence.  In order to get an appreciation for the difference, we invite the reader to consider
the nonlinearities
$$
|v|^2v \qtq{versus} [\Re v]^2 \Re v
$$
for a plane wave $v = e^{i\xi x-i\jp\xi t}$.
The fact that the NLS/NLKG embedding is still valid is a testament to the fact that $\cos^2(\xi x-\jp\xi t)$ behaves sufficiently like a constant function; however, this resemblance
cannot be captured using $L^q_t L^r_x$ norms.  Instead, we need to use $X^{s,b}$-inspired methods which capture the fact that the error terms present in this approximation do not exhibit
spacetime resonance with the linear propagator.  See also \cite{CCT, KKSV, Tao:gkdv} for similar arguments in the NLS/gKdV context.

The details of how these hurdles are overcome may be found in Section~\ref{S:embed}.  Ultimately, we show that if the characteristic length scale of a minimizing sequence were to diverge
(to infinity) then the solutions would inherit spacetime bounds from NLS, which contradicts the required divergence of their scattering sizes (cf. \eqref{enemy}).  In this way we
deduce that the failure of Conjecture~\ref{Conj:NLKG} implies the existence of a minimal-energy solution to NLKG of infinite scattering size.  This solution is global (in time).
Moreover, by utilizing the action of time translations (in the now standard manner), we see that this minimal-energy solution is almost periodic (= has precompact orbit) modulo translations:

\begin{definition}[Almost periodicity modulo translations] \label{D:apmt}
Fix $\mu = \pm1$.  We say that a global solution $u$ to \eqref{nlkg} is {\it almost periodic modulo translations} (in $H^1_x \times L^2_x$) if there exist functions $x:\R \to \R^2$
and $C:\R^+ \to \R^+$ such that for every $t \in \R$ and $\eta > 0$, we have
\begin{gather} \label{E:apmt u}
\int_{|x-x(t)| > C(\eta)} |u(t,x)|^2 + |\nabla u(t,x)|^2 + |u_t(t,x)|^2\, dx < \eta \\
\label{E:apmt uhat}
\int_{|\xi| > C(\eta)}  |\jp{\xi}\hat u(t, \xi)|^2 + |\hat u_t(t, \xi)|^2\, d\xi < \eta.
\end{gather}
We refer to $x(t)$ as the {\it spatial center function} and to $C$ as the {\it compactness modulus function}.
\end{definition}

\begin{remark}\label{R:conc E}  By \eqref{E:apmt u} and the Gagliardo--Nirenberg inequality (enlarging $C(\eta)$ if necessary),
\begin{equation*}
\int_{|x-x(t)| > C(\eta)}|u(t,x)|^2 + |\nabla u(t,x)|^2 + |u_t(t,x)|^2 + |u(t,x)|^4\, dx < \eta.
\end{equation*}
Similarly, by compactness one may ensure that for each $t \in \R$,
\begin{equation} \label{E:apmt uhat2}
\int_{|\xi| < 1/C(\eta)} |\jp{\xi}\hat u(t,\xi)|^2 + |\hat u_t(t,\xi)|^2\, d\xi < \eta.
\end{equation}
\end{remark}

To recap, the arguments discussed thus far in this overview culminate in the proof (in Section~\ref{S:min blowup}) of the following:

\begin{theorem}[Reduction to almost periodic solutions] \label{T:reduct}  Fix $\mu = \pm 1$ and suppose that Conjecture~\ref{Conj:NLS} holds but Conjecture~\ref{Conj:NLKG}
fails for this value of $\mu$.  Then there exists a global solution $u$ to \eqref{nlkg} with energy $E(u) = E_c$ $($and mass $M(u(0)) < M(Q)$ in the focusing case$)$.  Moreover, $u$ is almost periodic modulo translations
and blows up $($that is, possesses infinite scattering size$)$ both forward and backward in time.
\end{theorem}

The proof of this theorem actually shows that all minimal-energy blowup solutions are almost periodic, but this fact is not needed to prove Theorem~\ref{T:ST bounds}.

We refer to solutions of the type described in Theorem~\ref{T:reduct} as being \emph{soliton-like}.  While they move around in space (potentially arbitrarily), their profile does not change
very much; specifically, it remains inside some compact set in $H^1_x$.  To complete the proof of Theorem~\ref{T:ST bounds}, we merely need to prove that such soliton-like solutions
do not exist.  The first step in doing this is to provide some control over the spatial location $x(t)$ of the soliton.  From Einstein's relation $E^2=P^2c^2 + m^2c^4$ relating the energy $E$,
momentum
\begin{gather}
P := - \int_{\R^2} u_t \nabla u \, dx,
\end{gather}
and the rest mass $m$ ($c=1$ denotes the speed of light), we see that a minimal-energy blowup solution must have zero momentum.  In Lemma~\ref{L:x(t)} we use this to show that
$x(t)=o(t)$.  This is then combined with a monotonicity formula of virial type to obtain a contradiction.

\subsection{Notations}

Our convention for the Fourier transform is as follows:
$$
\hat f(\xi) = \tfrac1{2\pi} \int_{\R^2} e^{-ix\xi} f(x)\,dx.
$$

\begin{definition}[Littlewood--Paley projections]  Fix, once and for all, a smooth function $\phi:\R^2\to[0,1]$ obeying $\phi(\xi)=1$ for $|\xi|\leq 1$
and $\phi(\xi)=0$ for $|\xi| \geq \tfrac{99}{98}$.
For $N\in 2^{\Z}$ with $N\geq 1$, we define
\begin{align}\label{E:LP defn}
\widehat{\!P_N f} (\xi) = \begin{cases}  \phi(\xi)\hat f(\xi) & \text{if } N=1\\  [\phi(\xi/N) - \phi(2\xi/N)] \hat f(\xi) & \text{otherwise}.\end{cases}
\end{align}
We will typically abbreviate $f_N := P_N f$.
\end{definition}

\subsection*{Acknowledgments}  The first author was supported by NSF grants DMS-0701085 and DMS-1001531.  The second author was supported by an NSF Postdoctoral
Fellowship.  The third author was supported by NSF grant DMS-0901166 and a Sloan Foundation
Fellowship.  We are grateful to K.~Nakanishi for bringing this problem to our attention.

%
\section{Basic tools}\label{S:tools}
%

\subsection{Strichartz inequalities}

It will be more convenient for us to recast Klein--Gordon as a first-order equation for a complex-valued function via the map
\begin{equation} \label{E:real to complex}
(u,u_t ) \mapsto v= u + i\jpn^{-1}u_t .
\end{equation}
This is easily seen to be a bijection between real-valued solutions of \eqref{nlkg}
and complex-valued solutions of
\begin{align} \label{nlkg1st}
-i v_t + \jpn v +\mu \jpn^{-1} `(\Re v)^3 = 0,
\end{align}
with $\mu$ as in \eqref{nlkg}. As such, the local/global theories for the two equations are equivalent.

Corresponding to Definition~\ref{D:solution}, a solution to \eqref{nlkg1st} with initial data $v(0)$ satisfies the Duhamel formula
\begin{equation} \label{E:integral equation 1st}
v(t) = e^{-it\jpn}v(0) - i\mu \jpn^{-1}\int_0^t e^{-i(t-s)\jpn} \bigl( \Re v(s) \bigr)^3\, ds.
\end{equation}
We will consistently use the letter $u$ to denote solutions to \eqref{nlkg} and $v$ the corresponding solution to \eqref{nlkg1st}.  Note that the energies and scattering sizes
are related in the following ways:
\begin{align*}
S_I(u) &= S_I(v) = \int_I \int_{\R^2} |\Re v(t,x)|^4\, dx\, dt \\
E(u(t)) &= E(v(t)) = \int_{\R^2} \tfrac{1}2 |\jpn v(t,x)|^2 +\tfrac{\mu}4 |\Re v(t,x)|^4\, dx.
\end{align*}
Consistent with Definition~\ref{D:solution}, strong solutions of \eqref{nlkg1st} must have finite scattering size on compact subsets of the interval of existence.

The linear propagator associated to this first-order equation is $e^{-it\jpn}$. We will need to understand how this interacts with scaling.
For this reason, we record the basic dispersive estimate in the following form:

\begin{lemma}[Dispersive estimate] \label{L:Dispersive}
\begin{equation}\label{E:Dispersive}
\bigl\| e^{-i\lambda^2t\jp{\lambda^{-1}\nabla}} P_N f \bigr\|_{L^\infty_x(\R^2)} \lesssim  |t|^{-1} \jp{\lambda^{-1} N}^2 \| f \|_{L^1_x(\R^2)}.
\end{equation}
\end{lemma}

\begin{proof}
The phase function $\Phi(\xi)=\xi\cdot x - \lambda^2 t\langle\lambda^{-1}\xi\rangle$ obeys
$$
\text{Hessian}(\Phi)(\xi) = \det( \partial_j\partial_k \Phi(\xi) ) = t^2 \langle \lambda^{-1} \xi \rangle^{-4}.
$$
With this information, the method of stationary phase yields the requisite bounds on the integral kernel of
$e^{-i\lambda^2t\jp{\lambda^{-1}\nabla}} P_N$.
\end{proof}

Combining this dispersive estimate and the conservation of $L^2_x$ in the usual manner (cf. \cite{tao:keel} and references therein) yields the following:

\begin{lemma}[Strichartz inequality]\label{L:half Strichartz} For each $2< q \leq \infty$ and $2\leq r<\infty$ obeying the scaling condition $\tfrac2q + \tfrac 2r = 1$,
\begin{equation}\label{E:half Strichartz}
\bigl\| e^{-i\lambda^2t\jp{\lambda^{-1}\nabla}} f \bigr\|_{L^q_tL^r_x(\R\times\R^2)} \lesssim  \bigl\| \jp{\lambda^{-1}\nabla}^{\frac{r-2}r} f \bigr\|_{L^2_x(\R^2)}.
\end{equation}
Note that the implicit constant is independent of $\lambda$.
\end{lemma}

Using the $\lambda=1$ case of this lemma, the form of the free propagator, and Duhamel's principle, yields the following:

\begin{lemma}[Strichartz inequality]\label{L:Strichartz}
Let $I$ be a time interval and let $u$ and $v$ be solutions to the forced Klein--Gordon equations
$$
 u_{tt} - \Delta u + u = F \qtq{and} -i v_t + \jpn v = \jpn^{-1} G.
$$
Then,
\begin{align*}
\| \jp{\nabla_{t,x}}^{\frac{2}r} u\|_{L^q_tL^r_x(I\times\R^2)} &\lesssim \bigl\| \jp{\nabla_{t,x}} u(t_0) \bigr\|_{L^2_x(\R^2)}
    + \bigl\|\jpn^{1-\frac{2}{\tilde r}} F \bigr\|_{L^{\tilde q'}_tL^{\tilde r'}_x(I\times\R^2)} \\
\|\jpn^{\frac{2}r} v\|_{L^q_tL^r_x(I\times\R^2)}
        &\lesssim \bigl\| \jpn v(t_0) \bigr\|_{L^2_x(\R^2)} + \bigl\|\jpn^{1-\frac{2}{\tilde r}} G \bigr\|_{L^{\tilde q'}_tL^{\tilde r'}_x(I\times\R^2)}
\end{align*}
for any $t_0\in I$ and each $2< q,\tilde q\leq \infty$ and $2\leq r,\tilde r <\infty$ obeying the scaling condition $\tfrac2q + \tfrac 2r = \tfrac2{\tilde q} + \tfrac2{\tilde r} = 1$.
\end{lemma}

\subsection{Symmetries}

As for the wave equation, the full Poincar\'e group acts as symmetries of our equation.  Non-compact symmetries provide a clear obstruction to proving the existence
of minimal blowup solutions --- they provide an easy means for minimizing sequences to fail to converge.  In view of this, we will need some basic information about (as
well as notation for) the action of translations and Lorentz boosts.  As noted in the introduction, even though our equation is not scale invariant, dilations play an
important role in its analysis; thus, we will need to discuss these as well.

\subsubsection{Translations}

Our notation for translations is
\begin{equation}
[T_{y} f \bigr](x) := f(x-y).
\end{equation}

\subsubsection{Lorentz Boosts}\label{SS:boosts}

We parameterize Lorentz boosts in a manner inspired by their action on the Fourier side: Given a frequency parameter $\nu\in\R^2$, we define
\begin{align}\label{E:L boost}
(\tilde t, \tilde x) = L_\nu(t,x) := \bigl( \jp{\nu} t - \nu\cdot x, x^\perp + \jp{\nu} x^{\|} - \nu t \bigr).
\end{align}
Here $x^\perp$ and $x^{\|}$ denote (respectively) the components of $x$ perpendicular and parallel  to $\nu$. This
corresponds to a boost by velocity $\nu/\jp{\nu}$, which is to say that the observer with coordinates $(t,x)$ perceives
the observer with coordinates $(\tilde t, \tilde x)$ as moving with this velocity.  An easy computation shows
\begin{align}\label{E:L inv boost}
 L_\nu^{-1}(\tilde t,\tilde x) = \bigl( \jp{\nu} \tilde t + \nu\cdot \tilde x, \tilde x^\perp + \jp{\nu} \tilde x^{\|} + \nu \tilde t \, \bigr) = L_{-\nu}(\tilde t,\tilde x).
\end{align}

Note that the linear transformation \eqref{E:L boost} has determinant one and hence preserves spacetime volume.

Lorentz invariance of the linear (or nonlinear) Klein--Gordon equation is precisely the fact that $u\circ L_\nu^{-1}$ is a
solution if and only if $u$ is a solution. In particular,
\begin{equation}\label{E:boost a character}
u(t,x)= \exp(-i\jp{\xi} t + i\xi\cdot x) \implies u\circ L_\nu^{-1}(\tilde t,\tilde x) = \exp(-i\jp{\tilde \xi}\tilde t + i\tilde \xi\cdot\tilde x),
\end{equation}
where the new frequency parameters are related to the old via
\begin{align}\label{tildexi from xi}
\bigl( \jp{\tilde \xi} , \tilde \xi\, \bigr) = L_\nu\bigl(\jp{\xi},\xi\bigr) \qtq {or equivalently,} \tilde \xi=\ell_\nu(\xi):=\xi^\perp + \jp{\nu} \xi^{\|} - \nu\jp{\xi}.
\end{align}
Note in particular that $\tilde\xi=0$ if and only if $\xi=\nu$, which matches with the fact that $u$ represents a wave
traveling with velocity $\xi/\jp{\xi}$.  Note also that $\ell_{-\nu}\circ \ell_\nu=\Id$.

Associated to the action of Lorentz boosts on solutions to the linear Klein--Gordon equation there is a corresponding
action on initial data.  We denote this by the symbol $\W_\nu$ which is defined as follows:
\begin{equation}\label{E:Wa def}
[\W_\nu f](x) := [e^{-i\,\cdot\,\jpn} f]\circ L_\nu(0,x).
\end{equation}
By \eqref{E:L inv boost}, this is equivalent to
\begin{equation}\label{E:Wa inv def}
[\W_\nu^{-1} f](x) := [e^{-i\,\cdot\,\jpn} f]\circ L_\nu^{-1} (0,x).
\end{equation}
Note that $\W_\nu$ has been specifically defined so that
\begin{equation} \label{E:Boost linear soln}
[e^{-it\jpn}\W_\nu^{-1} f](x) := [e^{-i\,\cdot\,\jpn} f]\circ L_\nu^{-1}(t,x).
\end{equation}
In view of \eqref{E:boost a character}, the action of $\W_\nu$ is easily understood on the Fourier side.  In particular, we
have the following:

\begin{lemma}[Action of boosts] \label{L:boost action}
Let $\tilde\xi=\ell_\nu(\xi)$, as in \eqref{tildexi from xi}.  Then
\begin{align}\label{E:Fourier Boost}
\bigl( \W_\nu^{-1} f\bigr)\widehat{\ } (\tilde\xi) = \jp{\xi} \jp{\tilde\xi}^{-1} \hat f(\xi).
\end{align}
As a result, boosts do not commute with space/time translations:
\begin{align}\label{E:Boost translates}
 \W_\nu^{-1} T_y \, e^{i\tau\jpn} = T_{\tilde y} \, e^{i\tilde\tau\jpn} \W_\nu^{-1} \qtq{where} (\tilde\tau,\tilde y)=L_{\nu}(\tau,y).
\end{align}
The operator $\W_\nu$ is unitary in $H^{1/2}_x$, that is,
\begin{align} \label{E:H12 unitary}
\big\langle \W_\nu^{-1} f,\ \jpn g\bigr\rangle_{L^2_x} = \big\langle f,\ \jpn\W_\nu g\bigr\rangle_{L^2_x},
\end{align}
but not in general $H^s_x$ spaces:
\begin{align*}
\big\langle \W_\nu^{-1} f,\ g\bigr\rangle_{H^s_x} = \big\langle f,\ m_s(\nabla) \W_\nu g\bigr\rangle_{H^s_x},  \qtq{with}
    m_s(\xi) = m_s(\xi;\nu) =\biggl(\frac{\jp{\tilde\xi}}{\jp{\xi}}\biggr)^{2s-1}.
\end{align*}
However, $\| m_s\|_{L^\infty_\xi} + \| m_s^{-1} \|_{L^\infty_\xi} \lesssim \jp{\nu}^{|2s-1|}$.
\end{lemma}

\begin{proof}
From \eqref{E:Wa inv def} and \eqref{E:boost a character} we have
\begin{align*}
[\W_\nu^{-1} f] (x) &= (2\pi)^{-1} \int_{\R^2} e^{i\tilde\xi x} \hat f(\xi)\,d\xi = (2\pi)^{-1} \int_{\R^2} e^{i\tilde\xi x}  \hat f(\xi)\,\jp{\tilde\xi}^{-1} \jp{\xi}\,  d\tilde\xi.
\end{align*}
In the last equality we used that
\begin{align*}
\tilde\xi^{\|} = \jp{\nu}\xi^{\|} - \nu \jp{\xi}  \qtq{and so}
\frac{\partial \tilde\xi^{\|}}{\partial \xi^{\|}} = \jp{\nu} - \frac{\nu \xi^{\|}}{\jp{\xi}} = \frac{\jp{\tilde\xi}}{\jp{\xi}};
\end{align*}
hence (by triangularity) the full Jacobian is
\begin{equation}\label{E:Boost Jacobian}
\biggl| \frac{\partial \tilde\xi}{\partial\xi}\biggr| = \frac{ \jp{\tilde\xi} }{ \jp{\xi} },
    \qtq{that is,} d\xi = \jp{\tilde\xi}^{-1} \jp{\xi}  \,d\tilde\xi.
\end{equation}

Next we turn to \eqref{E:Boost translates}.  As Lorentz boosts preserve the Minkowski metric,
$$
- \tau \jp{\xi} + y\xi = -\tilde\tau\jp{\tilde\xi} + \tilde y \tilde\xi.
$$
Hence by \eqref{E:Fourier Boost},
\begin{align*}
\bigl( \W_\nu^{-1} T_{y} e^{i\tau\jpn} f\bigr)\widehat{\mathstrut\ }(\tilde\xi) = \jp{\xi} \jp{\tilde\xi}^{-1} e^{i\tau\jp{\xi}-iy\xi} \hat f(\xi)
    = e^{i\tilde\tau\jp{\tilde\xi} - i \tilde y \tilde\xi} \bigl( \W_\nu^{-1} f\bigr)\widehat{\mathstrut\ }(\tilde\xi).
\end{align*}
The result now follows after inverting the Fourier transforms.

The interaction of $\W_\nu$ with inner products in $H^{1/2}_x$, or any $H^s_x$ space, follows easily from \eqref{E:Fourier Boost} and \eqref{E:Boost Jacobian}.
\end{proof}

\begin{remark}
The behaviour of the $H^1_x$ norm becomes less mysterious if we consider instead the physical quantities of energy and momentum defined in the introduction.
If $u$ is a solution of the linear Klein--Gordon equation and $\tilde u=u\circ L_\nu^{-1}$, then the energy-momentum vectors are related by
$$
\bigl( \tilde E , \tilde P \bigr) = L_\nu\bigl(E,P),
$$
which follows from \eqref{E:boost a character}, \eqref{E:Boost Jacobian}, and Plancherel.  This relation also holds in the nonlinear case; see Corollary~\ref{C:boostable}.
\end{remark}

\begin{remark}
When interpreted in terms of its action on solutions $u(t,x)$ of linear Klein--Gordon (as opposed to initial data), the relation \eqref{E:Boost translates} takes the form
\begin{equation}\label{E:alt commute}
u(\cdot - \tau, \cdot - y) \circ L_\nu^{-1} = u \circ L_\nu^{-1}(\cdot - \tilde\tau, \cdot - \tilde y)
    \qtq{when} (\tilde\tau,\tilde y)=L_{\nu}(\tau,y),
\end{equation}
which is, of course, nothing but the linearity of the transformation $L_\nu$.
\end{remark}

\subsubsection{Scaling}

As noted in the introduction, the (nonlinear) Klein--Gordon equation does not possess a scaling symmetry; indeed, one of the key themes of this paper is how solutions with
small Fourier support behave as solutions of  the (nonlinear) Schr\"odinger equation.  Because shrinking Fourier support is a way an optimizing sequence may fail to converge,
this is something we need to address in our concentration compactness principle.  In this subsection, we merely introduce some notation for ($L^2_x$-preserving) dilation/scaling
operators and note how these interact with Fourier multipliers, including the free evolution.

\begin{definition}
For each $\lambda\in(0,\infty)$ we define a unitary operator $D_\lambda$ on $L^2_x$ by
$$
\bigl[D_\lambda f\bigr] (x) = \lambda^{-1} f(x/\lambda)
$$
Observe that $D_\lambda$ dilates by a factor $\lambda$ in the sense that the diameter of the support of $D_\lambda f$ is $\lambda$ times larger than that of $f$.
\end{definition}

Note that
\begin{equation}\label{E:dilate m}
m(\nabla) D_\lambda f  = D_\lambda m(\lambda^{-1}\nabla) f
\end{equation}
for any Fourier multiplier $m(\nabla)$.

\subsection{Useful lemmas}

The remainder of this section contains certain manipulations of symmetries that we will need in the proof of the inverse Strichartz inequality, Theorem~\ref{T:InvStrich}.

\begin{lemma}\label{L:h_n}
Fix $h\in L^2_x$ and $B>0$. Then with $m_0$ as in Lemma~\ref{L:boost action}, the set
\begin{equation*}
\mathcal K :=\bigl\{ D_\lambda^{-1} \W_\nu^{-1} m_0(\nabla)^{-1} e^{i\nu x} D_\lambda h : |\nu|\leq B \text{ and } B^{-1} \leq \lambda < \infty \bigr\}
\end{equation*}
is precompact in $L^2_x$.  Moreover, the closure of $\mathcal K$ does not contain $0$ unless $h\equiv 0$.

If $\hat h$ is the characteristic function of $[-1,1]^2$, then
\begin{equation}\label{E:K fourier decay}
\!\! \supp(\hat g)\subseteq \{|\xi|\lesssim \jp B\}, \quad \|g\|_{L^2_x} \gtrsim \jp{B}^{-1}, \qtq{and}
    \int_{|x|\sim R} |g(x)|^2\,dx \lesssim \frac{\jp{B}}{\jp{R}},
\end{equation}
all uniformly for $g\in\mathcal K$.
\end{lemma}

\begin{proof}
Careful computation shows that
\begin{equation}\label{E:h_n mess}
\bigl[ D_\lambda^{-1} \W_\nu^{-1} m_0(\nabla)^{-1} e^{i\nu x} D_\lambda h \bigr]\widehat{\ }(\tilde\xi) = \hat h \circ G(\tilde\xi)
\end{equation}
where $G(\tilde\xi) =\lambda[\ell_\nu^{-1}\bigl(\tilde\xi/\lambda\bigr)-\nu]= \tilde\xi^\perp + \jp\nu \tilde\xi^\| + \lambda \nu [\jp{\lambda^{-1}\tilde\xi} - 1]$.
The conclusions of the lemma will follow from some basic properties of this function $G$.

First we note that $G$ is a bijection on $\R^2$; indeed, the computation \eqref{E:h_n mess} reveals it to be the composition of dilations, translations (cf. $e^{i\nu x}$),
and the bijection $\ell_\nu^{-1}$ associated to $\W_\nu^{-1}$.  Moreover, the Jacobian of $G$ is
$$
\det[ G' ] = \jp{\nu} + \lambda^{-1} \jp{\lambda^{-1}\tilde\xi}^{-1} \nu\cdot\tilde\xi,
$$
which is uniformly bounded both above and below:
\begin{equation}\label{Bound det G'}
\jp{\nu}^{-1} \lesssim \bigl| \det[ G' ] \bigr| \lesssim \jp{\nu}.
\end{equation}
(Here we used $|\lambda^{-1}\tilde\xi| \leq \jp{\lambda^{-1}\tilde\xi}$ and $[\jp\nu-|\nu|][\jp\nu+|\nu|]=1$.)

From \eqref{Bound det G'} we can conclude that $\hat h\mapsto \hat h\circ G$ is a uniformly bounded family of
operators on $L^2_x$ for $|\nu|\leq B$ and $\lambda\in[B^{-1},\infty]$. Note the inclusion of $\lambda=\infty$ here; this is possible
since $\lim_{\lambda\to\infty} G(\tilde\xi) = \tilde\xi^\perp + \jp\nu \tilde\xi^\|$, which is still a bijection with bounded Jacobian.

To finish the proof of precompactness, it suffices to show that $\hat h\circ G$ varies continuously in $L^2_x$ as $\lambda$ and $\nu$ vary over over the compactified region.
By virtue of the uniform boundedness of $\hat h\mapsto \hat h\circ G$ we may safely replace $\hat h$ by an element of $C^\infty_c(\R^2)$.  With this reduction,
the result becomes an easy consequence of the continuity of $G(\tilde\xi)$ as a function of $\lambda$ and $\nu$ and the fact that
\begin{equation}\label{E:G transport}
[\jp{\nu}-|\nu|] |\tilde\xi| \lesssim |G(\tilde\xi)| \lesssim [\jp{\nu}+|\nu|] |\tilde\xi|,
\end{equation}
which follows from $|\jp{\lambda^{-1}\tilde\xi} - 1|\leq\lambda^{-1}|\tilde\xi|$, a consequence of the subadditivity of the square-root.

That $\mathcal K$ stays away from zero follows immediately from the upper bound in \eqref{Bound det G'}.

Lastly we turn to \eqref{E:K fourier decay}.  Inclusion of the Fourier support follows immediately from the lower bound in \eqref{E:G transport},
while the second claim and the case $R\leq 1$ of the last inequality follow directly from \eqref{Bound det G'}.
To treat the case $R\geq 1$, we note that by \eqref{E:h_n mess}, the Fourier transform of a fixed $g\in\mathcal K$ is the characteristic function
of a set with piecewise smooth boundary and
$$
\text{Length}(\partial\supp \hat g) \lesssim \|(G^{-1})'\|_{L^\infty(\R^2;\R^{2\times 2})} \lesssim \jp{\nu}.
$$
Therefore, for each vector $|\eta|\leq 1$,
$$
\int_{\R^2} |x| |g(x)|^2 \frac{|e^{i\eta x} -1|^2}{|\eta| |x|} \,dx = |\eta|^{-1} \int_{\R^2} |\hat g(\tilde\xi-\eta) - \hat g(\tilde\xi)|^2\,d\tilde\xi
    \lesssim \jp\nu.
$$
The estimate then follows by adding together this estimate for vectors $\eta$ of length $R^{-1}$ pointing in a fixed collection of
directions.  (The exact number of vectors needed is dictated by the constants in $|x|\sim R$.)
\end{proof}

\begin{lemma}\label{L:pointwise}
(a) Suppose $g_n\rightharpoonup g$ weakly in $H^1_x$ and $\lambda_n\to\lambda\in(0,\infty)$.  Then there is a subsequence so that
\begin{equation*}
\bigl[e^{-i\lambda_n^2 t \jp{\lambda_n^{-1}\nabla}} g_n \bigr](x) \to [e^{-i\lambda^2 t \jp{\lambda^{-1}\nabla}} g](x)  \quad\text{for almost every $(t,x)\in\R\times\R^2$.}
\end{equation*}
(b) For $\lambda_n\to\lambda\in(0,\infty)$ and fixed $g\in H^1_x$,
\begin{equation*}
\bigl\| e^{-i\lambda_n^2 t \jp{\lambda_n^{-1}\nabla}} g - e^{-i\lambda^2 t \jp{\lambda^{-1}\nabla}} g \bigr\|_{L^4_{t,x}} \to 0.
\end{equation*}
(c) Fix $\theta\in(0,\frac12)$ and suppose $g_n\rightharpoonup g$ weakly in  $L^2_x$ and $\lambda_n\to\infty$.  Then there is a subsequence so that
\begin{equation*}
\bigl[e^{-i\lambda_n^2 t [\jp{\lambda_n^{-1}\nabla}-1]} P_{\leq \lambda_n^\theta} g_n \bigr](x) \to [e^{it\Delta/2} g](x) \quad\text{for almost every $(t,x)\in\R\times\R^2$.}
\end{equation*}
(d) For $\lambda_n\to\infty$, $\theta\in(0,\frac12)$, and fixed $g\in L^2_x$,
\begin{equation*}
\bigl\| e^{-i\lambda_n^2 t [\jp{\lambda_n^{-1}\nabla}-1]} P_{\leq \lambda_n^\theta} g - e^{it\Delta/2} g \bigr\|_{L^4_{t,x}} \to 0.
\end{equation*}
\end{lemma}

\begin{proof}
With regard to almost everywhere convergence of a subsequence, it suffices (via Cantor's diagonal argument) to work on a generic cube,
say, $(t,x)\in [-L,L]^3$. This in turn can be deduced from $L_{t,x}^2$ convergence of a larger subsequence there.

Consider part (a).  First we show the existence of an almost everywhere convergent subsequence; only after that will we identify the limit.  By the Strichartz inequality
Lemma~\ref{L:half Strichartz},
$$
\limsup_{n\to \infty}\bigl\| \jp{\partial_t}^{\frac14} \jpn^{\frac14} e^{-i\lambda_n^2 t \jp{\lambda_n^{-1}\nabla}} g_n \bigr\|_{L^4_{t,x}(\R\times\R^2)} \lesssim_\lambda \| g \|_{H^{1}_x}.
$$
Combining this with Rellich's Theorem, specifically, compactness of the embedding $W^{1/4,4}(\R^3)\hookrightarrow
L^2([-L,L]^3)$, we obtain an $L^2_{t,x}$ (and thence a.e.) convergent subsequence on this cube.

The fact that we have local convergence in $L^2_{t,x}$ also allows us to identify the limit: for all $F\in\C^\infty_c(\R\times\R^2)$,
\begin{align*}
\lim_{n\to\infty} \int_\R \int_{\R^2} \overline{ F(t,x) } & [e^{-i\lambda_n^2 t \jp{\lambda_n^{-1}\nabla}} g_n](x) \,dx\,dt \\
={}& \lim_{n\to\infty} \int_{\R^2} g_n(x) \int_\R \overline{ [e^{i\lambda_n^2 t \jp{\lambda_n^{-1}\nabla}} F(t,\cdot)](x) } \,dt\,dx \\
={}& \int_{\R^2} g(x) \int_\R \overline{ [e^{i\lambda^2 t \jp{\lambda^{-1}\nabla}} F(t,\cdot)](x) } \,dt\,dx \\
={}& \int_\R \int_{\R^2} \overline{ F(t,x) } [e^{-i\lambda^2 t \jp{\lambda^{-1}\nabla}} g](x) \,dx\,dt.
\end{align*}

The proof of (b) is easily adapted from the proof of (d) which we give below.

We now turn to the more subtle part (c) where $\lambda_n\to\infty$.  By noting that
\begin{equation}\label{phase approx}
\lambda_n^2 t [\jp{\lambda_n^{-1}\xi}-1]  = \tfrac12 t |\xi|^2 + O\bigl(t\lambda_n^{-2} |\xi|^4\bigr) \qtq{as} \lambda_n\to\infty,
\end{equation}
we deduce that for $\theta<\frac12$,
$$
\bigl\| e^{-i\lambda_n^2 t [\jp{\lambda_n^{-1}\nabla}-1]} P_{\leq \lambda_n^\theta} - e^{it\Delta/2} P_{\leq \lambda_n^\theta} \bigr\|_{L^2_x \to L^2_x}
    \underset{n\to\infty}\longrightarrow 0.
$$
Thus it suffices to prove convergence of a subsequence of $e^{it\Delta/2} P_{\leq \lambda_n^\theta} g_n$ on our generic
cube $[-L,L]^3$.  The key to doing so is the well-known local smoothing estimate
for the Schr\"odinger equation:
\begin{equation}\label{E:LocalSmoothing}
\int_\R \! \int_{[-L,L]^2} \bigl|\bigl[\jpn^{\frac12} e^{it\Delta/2} f\bigr](x)\bigr|^2 \,dx\,dt
    \lesssim L \cdot \|f\|_{L^2_x(\R^2)}^2 ;
\end{equation}
see \cite{ConsSaut,Sjolin87,Vega88}.  This estimate implies
\begin{equation}\label{E:LocalSmoothing'}
\bigl\| \jp{\partial_t}^{\frac18} \jpn^{\frac14} e^{it\Delta/2} g_n \bigr\|_{L^2_{t,x}([-L,L]^3)} \lesssim L^{1/2} \cdot \|g_n\|_{L^2_x(\R^2)}
\end{equation}
and hence by Rellich's Theorem, the existence of an $L^2_{t,x}$ convergent subsequence on $[-L,L]^3$.
The identification of the limit follows by testing against $F\in\C^\infty_c(\R\times\R^2)$, as above.

Lastly, we address part (d).  By the Strichartz inequality Lemma~\ref{L:half Strichartz} and its analogue for $e^{it\Delta/2}$, it suffices to treat the case when $g$ is a Schwartz function.
(Note the importance of uniformity in $\lambda$ in Lemma~\ref{L:half Strichartz}.)  Next we note that by the (uniform in $\lambda$) dispersive estimate \eqref{E:Dispersive}
and its analogue for the Schr\"odinger propagator,
$$
\bigl\| e^{-i\lambda_n^2 t [\jp{\lambda_n^{-1}\nabla}-1]} P_{\leq \lambda_n^\theta} g \bigr\|_{L^4_{t,x}(|t|\geq T)}
    + \bigl\| e^{it\Delta/2} g \bigr\|_{L^4_{t,x}(|t|\geq T)}
    \lesssim T^{-\frac14} \| g\|_{L^{4/3}_x}.
$$
Thus we are left to control the region $|t|\leq T$.  First we note that
\begin{align}\label{1}
\bigl\| e^{-i\lambda_n^2 t [\jp{\lambda_n^{-1}\nabla}-1]} P_{\leq \lambda_n^\theta} g - e^{it\Delta/2} g \bigr\|_{L^\infty_t L^2_x(|t|\leq T)}
    \lesssim (\lambda_n^{-2} T + \lambda_n^{-4\theta}) \|g\|_{H^4_x},
\end{align}
which follows from \eqref{phase approx} and the fact that
$$
\bigl\|P_{> \lambda_n^\theta}g \bigr\|_{L^\infty_t L^2_x(|t|\leq T)}\lesssim \lambda_n^{-4\theta} \|g\|_{H^4_x}.
$$
On the other hand, by the Strichartz estimates and Sobolev embedding,
\begin{align}\label{2}
\bigl\| e^{-i\lambda_n^2 t [\jp{\lambda_n^{-1}\nabla}-1]} P_{\leq \lambda_n^\theta} g \bigr\|_{L^3_t L^6_x}
    + \bigl\| e^{it\Delta/2} g \bigr\|_{L^3_t L^6_x} \lesssim \| g \|_{H^{2/3}_x}.
\end{align}
Interpolating between \eqref{1} and \eqref{2}, we obtain
\begin{align*}
\lim_{n\to\infty} \bigl\| e^{-i\lambda_n^2 t [\jp{\lambda_n^{-1}\nabla}-1]} P_{\leq \lambda_n^\theta} g - e^{it\Delta/2} g \bigr\|_{L^4_{t,x}(|t|\leq T)} = 0
\end{align*}
for each fixed $T$.
\end{proof}

The significance of this lemma for us is that it provides the input for the following variant of Fatou's lemma
due to Br\'ezis and Lieb  (see also \cite[Theorem~1.9]{LiebLoss}):

\begin{lemma}[Refined Fatou, \cite{BrezisLieb}]\label{L:BrezisLieb}
Let $d\geq 1$ and $1\leq p<\infty$ and suppose $\{F_n\}\subseteq L^p(\R^d)$ with $\limsup \|F_n\|_p<\infty$.  If $F_n\to F$ almost everywhere, then
\begin{align*}
\int_{\R^d} \Bigl| |F_n|^p - |F_n-F|^p - |F|^p \Bigr|\,dx \to 0.
\end{align*}
In particular, if $G_n\to F$ in $L^p$ sense, then
\begin{equation}\label{E:BrezisLieb}
\limsup_{n\to\infty} \|F_n-G_n\|_{L^p} \leq \limsup_{n\to\infty} \; \Bigl( \, \|F_n\|_{L^p}^p \! - \|F\|_{L^p}^p \Bigr)^{1/p}.
\end{equation}
\end{lemma}

\subsection{Elliptic estimates}\label{SS:elliptic}

In this subsection, we first record a special case of the sharp Gagliardo--Nirenberg inequality of Weinstein \cite{weinstein} and then discuss some consequences for our equation
in the focusing setting.

\begin{theorem}[Sharp Gagliardo--Nirenberg, \cite{weinstein}]\label{T:SharpGN}   For all $f\in H^1_x(\R^2)$,
\begin{equation}\label{E:General SharpGN}
\| f \|_{L^4_x}^{4} \leq 2\|Q\|_{L^2_x}^{-2} \, \|f\|_{L^{2}_x}^{2} \, \|\nabla f\|_{L^{2}_x}^{2}.
\end{equation}
Here $Q$ denotes the unique positive radial Schwartz solution to $\Delta Q + Q^{3} = Q$.
Moreover, equality holds in \eqref{E:General SharpGN} if and only if $f(x)=\alpha Q(\lambda (x-x_0))$ for some $\alpha\in\C$, $\lambda\in(0,\infty)$, and $x_0\in\R^2$.
\end{theorem}

It is not difficult to prove the existence of $Q$, for example by variational methods.  A proof of uniqueness (along with earlier references) can be found in Kwong \cite{kwong}.
Integrating the equation for $Q$ against $Q$ and $x\cdot \nabla Q$ yields
\begin{align}\label{E:Poh}
E(Q):= \int_{\R^2} \tfrac12 |Q|^2 +\tfrac12 |\nabla Q|^2 - \tfrac{1}{4} |Q|^{4} \,dx = \int_{\R^2} \tfrac12 |Q|^2=:\tfrac12 M(Q).
\end{align}
This is known as Pohozaev's identity.

\begin{proposition}[Energy coercivity]\label{P:coercive}
Let $u:I\times\R^2\to\R$ be a solution to \eqref{nlkg} in the focusing case with initial data $(u(0), u_t(0))=(u_0, u_1)\in H^1_x\times L_x^2$
whose energy
$$
E(u)= E(u_0,u_1) = \int_{\R^2} \tfrac12 |u_1|^2 +\tfrac12 |\nabla u_0|^2 + \tfrac12 |u_0|^2 - \tfrac{1}{4} |u_0|^{4} \,dx
$$
obeys $E(u)<E(Q)$.
\begin{SL}
\item If $M(u(0))<M(Q)$, then
\begin{align} \label{E:Mass trap}
\int_{\R^2} |u(t)|^2 + |u_t(t)|^2 \,dx &\leq 2E(u) < M(Q)  \\
\int_{\R^2} |\nabla u(t)|^2 + |u_t(t)|^2 \,dx &\leq 2E(u) < M(Q)  \label{E:Grad trap}
\end{align}
for all $t\in I$.  As a consequence,
\begin{align}\label{E:coercive E}
2E(u)\leq \|u(t)\|_{H^1_x}^2 + \|u_t(t)\|_{L_x^2}^2 \leq 4E(u).
\end{align}
\item If $M(u(0))>M(Q)$ then
\begin{align}\label{E:big mass}
M(u(t)) > M(Q) \qtq{and} \int_{\R^2} |\nabla u(t)|^2 > M(Q)
\end{align}
for all $t\in I$.  Moreover,
\begin{align}\label{E:M_tt}
\partial_{tt} M(u(t)) = 2\int_{\R^2} |u_t|^2 - |\nabla u|^2 - |u|^2 + |u|^{4} \,dx  > 6\int_{\R^2} |u_t|^2 \,dx.
\end{align}
\end{SL}
\end{proposition}

\begin{proof}
We begin with part (i).  As energy is conserved, $E(u(t))<E(Q)$ for all $t\in I$.  Combining this with \eqref{E:Poh} and the sharp Gagliardo--Nirenberg inequality shows that
$$
M(u(t)) \leq M(Q) \implies M(Q) > 2E(u(t)) \geq \int_{\R^2} |u(t)|^2 + |u_t(t)|^2 \geq M(u(t)).
$$
By definition, solutions are continuous in $H^1_x\times L^2_x$ and so we see that \eqref{E:Mass trap} follows by a simple bootstrap/continuity argument.

We now turn to the proof of \eqref{E:Grad trap}. By the sharp Gagliardo--Nirenberg inequality and \eqref{E:Poh},
$$
M(Q) > 2E(u(t)) \geq \int_{\R^2} |u_t(t)|^2 + |u(t)|^2 + \bigl[ 1 - \tfrac{M(u(t))}{M(Q)} \bigr] |\nabla u|^2 \,dx
$$
and so, neglecting the $u_t$ term and doing a little rearranging, we find that
\begin{equation}\label{E:grad vs mass}
\bigl[ M(Q) - M(u(t)) \bigr] \bigl[ M(Q) - \|\nabla u(t)\|_{L^2_x}^2 \bigr] > 0.
\end{equation}
From \eqref{E:Mass trap}, we see that $M(u(t))<M(Q)$ throughout the interval of existence.  Thus $\|\nabla u(t)\|_{L^2_x}^2 < M(Q)$ and so, by the sharp Gagliardo--Nirenberg inequality,
$$
\tfrac14 \int_{\R^2} |u(t,x)|^4 \,dx \leq \tfrac12 \int_{\R^2} |u(t,x)|^2 \,dx.
$$
The estimate \eqref{E:Grad trap} now follows directly from the definition of energy.

The proof of part (ii) closely parallels that of part (i).  The first inequality in \eqref{E:big mass} follows from a simple bootstrap argument
based on the fact that if $M(u(t))=M(Q)$ then by the sharp Gagliardo--Nirenberg inequality, $2E(u(t))\geq M(Q)=2E(Q)$.
The second inequality in \eqref{E:big mass} follows from the first via \eqref{E:grad vs mass}.

The evaluation of $\partial_{tt} M(u(t))$ is an elementary computation.  Note that the answer can be rewritten in the form
$$
\partial_{tt} M(u(t)) =  -8 E(u) + \int_{\R^2} 6|u_t|^2 + 2|\nabla u|^2 + 2|u|^2,
$$
from which \eqref{E:M_tt} follows from $2E(u)<M(Q)$ and both inequalities in \eqref{E:big mass}.
\end{proof}


\section{Local theory}\label{S:LT}


By using $L^4_{t,x}$ and $L^\infty_t H^1_x$ as the basic spaces in a contraction mapping argument, the Strichartz estimates directly yield local well-posedness,
persistence of regularity, and stability results.  The arguments leading to local well-posedness can be found in any textbook on dispersive PDE, for example, \cite{cazenave:book}.
A proof of the persistence of regularity result may be adapted from the proof for NLS given in \cite[Lemma 3.10]{Matador}.  We formulate these basic statements in the context of the first-order
equation \eqref{nlkg1st}; the reader should have no difficulty reformulating them for solutions to \eqref{nlkg}.

\begin{proposition}[Local well-posedness for $H^1$ initial data] \label{P:lwp}
\leavevmode\quad
Let $v_0 \in H^1_x(\R^2)$.  Then there exists a unique maximal-lifespan (strong) solution $v:I\times\R^2\to \C$ to \eqref{nlkg1st} with $v(0) = v_0$.  Furthermore, the following hold:
\begin{CI}
\item (Blowup alternative) If $T=\sup I$ is finite, then $\| v(t) \|_{H^1_x} \to \infty$ as $t\to T$.
\item (Conservation laws)  The energy and momentum are finite and constant in time.
\item (Scattering)  If $v$ does not blow up forward in time, that is, if $S_{[0,\infty)}(v) < \infty$, then there exists $v_+ \in H^1_x(\R^2)$ such that
\begin{equation} \label{E:scattering}
\lim_{t \to \infty} \|v(t) - e^{-it\jpn}v_+\|_{H^1_x(\R^2)} = 0.
\end{equation}
Furthermore, for each $v_+ \in H^1_x(\R^2)$, there exists a unique $v$ which solves \eqref{nlkg1st} in a neighbourhood of $+\infty$ and satisfies \eqref{E:scattering}.  In either case,
\begin{equation} \label{E:scattering nrg}
E(v) = \tfrac{1}{2}\|v_+\|_{H^1_x(\R^2)}^2.
\end{equation}
Similar statements hold backward in time.
\item (Small data result) If $\,\|v_0\|_{H^1_x}$ is sufficiently small, then $v$ is global and moreover,
$$
S_{\R}(v) \lesssim E(v)^2.
$$
\item (Small solution to LKG implies small solution to NLKG) If $I$ is an interval, $0 \in I$, and $\|\Re e^{-it\jpn}v_0\|_{L^4_{t,x}(I \times \R^2)}$ is sufficiently small,
then $I$ is contained in the lifespan of $v$ and
$$
S_I(v) \lesssim \|\Re e^{-it\jpn}v_0\|_{L^4_{t,x}(I \times \R^2)}^4.
$$
\item (Persistence of regularity) Let $I \subset \R$ and assume that $S_{I}(v) < L$.  Given $s \geq 0$,
\begin{equation} \label{E:nlkg persistence}
\|\jpn^{s+\frac2r} v\|_{L^q_t L^r_x(I \times \R^2)} \lesssim_{L,s,q,r}\| \jpn^{s+1} v_0\|_{L^2_x(\R^2)}
\end{equation}
for each $q$ and $r$ obeying $2<q\leq \infty$ and $\frac1q+\frac1r=\frac12$.
\end{CI}
\end{proposition}

\begin{remark}\label{R:small D}
Note that in the small-data setting, the Gagliardo--Nirenberg inequality shows that $\|v(t)\|_{H^1_x}^2 \sim E(v)$.  As a consequence,
$\|v\|_{L^\infty_t H^1_x} \lesssim \|v_0\|_{H^1_x}$.
\end{remark}

In the defocusing case, energy controls $\| v(t) \|_{H^1_x}^2$ and so finite-time blowup cannot occur. In the focusing case,
the energy is no longer coercive in general; nevertheless, by part~(i) of Proposition~\ref{P:coercive} we do obtain the following:

\begin{corollary}[Global well-posedness] \label{Cor:gwp}
In the defocusing case, any initial data $v_0\in H^1_x$ leads to a global solution to \eqref{nlkg1st}.  In the focusing case, any initial data $v_0\in H^1_x$
obeying
$$
\int_{\R^2} |\Re v_0|^2\,dx \leq \int_{\R^2} Q^2\,dx \qtq{and} E(v_0) < E(Q)
$$
leads to a global solution.  Here $Q$ denotes the ground state, as in Theorem~\ref{T:SharpGN}.
\end{corollary}

In view of this corollary, the main objective of this paper is to prove that these global solutions obey spacetime bounds, which, by the local theory
mentioned above, will also imply scattering.  As described in the introduction, we will argue by contradiction, showing that failure of Theorem~\ref{T:ST bounds}
implies the existence of minimal-energy counterexamples (see Section~\ref{S:min blowup}).  A key ingredient is the following stability theory.
The significance of stability theory has really only come to the fore with the investigation of scaling-critical equations.  The archetypal argument appears in \cite{CKSTT:gwp};
see also \cite{TV}. A detailed proof for the mass-critical nonlinear Schr\"odinger equation is given in \cite{ClayNotes}, and only minor modifications are required to extend that argument
to the first-order nonlinear Klein--Gordon \eqref{nlkg1st}.

\begin{prop}[Stability theory] \label{P:stability}
Let $I$ be an interval and let $\tilde{v}$ be an approximate solution to \eqref{nlkg1st} on $I$ in the sense that
$$
-i \tilde v_t + \jpn\tilde{v} +\mu {\jpn}^{-1} (\Re \tilde{v})^3 + e_1 + e_2 = 0,
$$
with small error terms $e_1$ and $e_2$.  Assume that
$$
\bigl\| \jpn^{1/2} \tilde{v} \bigr\|_{L^{\infty}_t L^2_x(I \times \R^2)} \leq M \qtq{and} \|\Re \tilde{v}\|_{L^4_{t,x}(I \times \R^2)} \leq L
$$
for some positive constants $M$ and $L$.  Let $t_0 \in I$ and let $v_0$ satisfy the condition
$$
\bigl\| \jpn^{1/2} (v_0 - \tilde{v}(t_0)) \bigr\|_{L^2_x(\R^2)} \leq M'
$$
for some positive constant $M'$.  Then if $0 < \eps < \eps_1(L,M,M')$ and if $v_0$ and the error terms satisfy
\begin{align*}
\| e^{-i(t-t_0) \jpn} (v_0 - \tilde{v}(t_0))\|_{L^4_{t,x}(I \times \R^2)} \leq \eps, \\
\|\jpn e_1\|_{L^{4/3}_{t,x}(I \times \R^2)} + \|e_2\|_{L_t^1 H_x^{1/2}(I \times \R^2)} \leq \eps,
\end{align*}
then there exists a solution $v$ to \eqref{nlkg1st} with initial data $v_0$ at time $t=t_0$.  Furthermore, the solution $v$ satisfies
\begin{align*}
\|v - \tilde{v}\|_{L^4_{t,x}(I \times \R^2)} \leq \eps C(M,M',L) \\
\|v - \tilde{v}\|_{L^{\infty}_t H^{1/2}_x (I \times \R^2)} \leq M'C(M,M',L).
\end{align*}
\end{prop}

The local theory described so far treats time as an absolute, which jars with the Lorentz invariance of our equation.  Moreover, it does not allow us to consider
boosted solutions $u\circ L_\nu$, even for small values of $\nu$.  The next lemma remedies this by proving local existence in a larger spacetime region; we then
make some basic observations about the behaviour of the boosted solutions in Corollary~\ref{C:boostable}.

\begin{lemma}[Boostable local solutions]\label{L:boostable}
Given initial data $(u_0,u_1)\in  H^1_x \times L^2_x$, there is an $\eps>0$ and a local solution $u$ to \eqref{nlkg} matching this data
$($at $t=0)$ and defined in the spacetime region $\Omega=\{ (t,x) : |t| - \eps |x| < \eps\}$.  Moreover,
\begin{gather*}
\| u \|_{L^q_t L^r_x(\Omega)} := \| \chi_\Omega u \|_{L^q_t L^r_x(\R\times\R^2)} < \infty \quad\text{for each $2<q\leq\infty$ and $\tfrac1q + \tfrac1r=\tfrac12$,} \\
\| u \|_{L^\infty_t(H^1_x\times L^2_x)(\Omega)} ^2 := \sup_t \int_{\R^2} \chi_\Omega(t,x) \bigl[ |u_t(t,x)|^2 + |\nabla u(t,x)|^2 + |u(t,x)|^2 \bigr] \,dx < \infty,
\end{gather*}
and
\begin{align}\label{E:boost tightness}
\lim_{R\to\infty} \sup_{|t|<\eps R} \int_{|x|>R} \bigl[ |u_t(t,x)|^2 + |\nabla u(t,x)|^2 + |u(t,x)|^2 \bigr] \,dx = 0.
\end{align}
The solution $u$ with these properties is unique.
\end{lemma}

\begin{remark} \label{R:big eps}
As we will see from the proof, if $(u_0,u_1)$ leads to a global solution, then we may take any $0 < \eps < 1$.
\end{remark}

\begin{proof}
Both existence and uniqueness follow directly by combining the local theory described so far with finite speed of propagation.  First we note that by Proposition~\ref{P:lwp}
there exists $T_0=T_0(u_0,u_1)>0$ so that there is a (unique) local solution $u$ to \eqref{nlkg} defined on the spacetime slab $|t| < T_0$ and having finite spacetime norms there.

Next, let $\phi$ denote a smooth cutoff function with $\phi(x)=1$ outside the unit ball and $\phi(x)=0$ when $|x|<\frac12$.  Given any $\eta>0$ (in particular, the threshold
for the small data theory), there is an $R_0$ sufficiently large so that
$$
  \int_{\R^2}  \bigl[ |\phi(x/R_0) u_1(x)|^2 + |\nabla [\phi(x/R_0) u_0(x)]|^2 + |\phi(x/R_0)u_0(x)|^2 \bigr] \,dx < \eta.
$$
Thus by Proposition~\ref{P:lwp}, there is a global solution $\tilde u$ to \eqref{nlkg} with initial data $\tilde u(0,x) = \phi(x/R_0) u_0(x)$ and $\tilde u_t(0,x) = \phi(x/R_0) u_1(x)$.
By uniqueness and finite speed of propagation, $\tilde u$ provides an extension of $u$ to the spacetime region where $|x|-|t| > R_0$.  Note that $\tilde u$ has finite (global) spacetime norms.

To recap, we have proved that there is a unique local solution with finite spacetime bounds on the region where $|t|<T_0$ or $|x|-|t|>R_0$.  This includes the region
$\Omega$ provided we choose $\eps <T_0/(1+R_0+T_0)$ and so settles the majority of the lemma; it remains only to prove the tightness statement \eqref{E:boost tightness}.

The proof of \eqref{E:boost tightness} is a simple variation on the $\tilde u$ construction above.  Indeed, if $\tilde u^{(\tilde R)}$ is the solution to \eqref{nlkg} with
initial data $\tilde u^{(\tilde R)}(0,x) = \phi(x/\tilde R) u_0(x)$ and $\tilde u_t^{(\tilde R)}(0,x) = \phi(x/\tilde R) u_1(x)$ then (cf. Remark~\ref{R:small D})
$$
\lim_{\tilde R\to\infty}   \| \tilde u^{(\tilde R)} \|_{L^\infty_t H^1_x(\R\times\R^2)}
     + \| \partial_t \tilde u^{(\tilde R)} \|_{L^\infty_t L^2_x(\R\times\R^2)} =0.
$$
Taking $\tilde R< (1-\eps)R$, this proves \eqref{E:boost tightness} because $u$ and $\tilde u^{(\tilde R)}$ agree on the region $|x| - |t| > \tilde R$.
\end{proof}

\begin{corollary}\label{C:boostable}
In view of Lemma~\ref{L:boostable}, any initial data $u(0)\in H^1_x$ and $u_t(0)\in L^2_x$ lead to a solution $u$ to \eqref{nlkg} in a spacetime region of the form
$\Omega=\{ (t,x) : |t| - \eps |x| < \eps\}$ for some $\eps>0$.  For $\tfrac{|\nu|}{\jp{\nu}}<\eps$, we have
\begin{SL}
\item $u\circ L_\nu(t,x)$ is a (strong) solution to \eqref{nlkg} on $(-\eps,\eps)\times\R^2$.
\item $\nu \mapsto \bigl(u\circ L_\nu(0,x), [u \circ L_\nu]_t(0,x) \bigr)$ is continuous with values in $H^1_x\times L^2_x$.
\item The energy and momentum obey Einstein's relation:
\begin{equation}\label{E:boosted EP}
\bigl(E(u\circ L_\nu), P(u\circ L_\nu) \bigr)=L_\nu^{-1}\bigl(E(u), P(u)\bigr).
\end{equation}
In particular,
\begin{equation}\label{E:Einstein}
E\bigl(u\circ L_\nu(t)\bigr)^2 - P\bigl(u\circ L_\nu(t)\bigr)^2 \quad\text{is independent of $t$ and $\nu$.}
\end{equation}
\end{SL}
\end{corollary}

\begin{remark}
The square-root of the quantity in \eqref{E:Einstein} is usually known as the \emph{rest mass}; compare the famous equation $E^2=P^2c^2+m^2c^4$, or its $P=0$ case $E=mc^2$.
We will not use the term \emph{rest mass} here so as to avoid confusion with the (non-conserved) quantity $M(u(t))$.
\end{remark}

\begin{proof}
It is not difficult to verify that if $|\nu|\jp{\nu}^{-1}<\eps$ then $L_\nu$ maps $(-\eps,\eps)\times\R^2$ into $\Omega$.  Thus $u\circ L_\nu$ is defined in the region
claimed.  As $L_\nu$ is volume preserving, we also see that
\begin{equation}\label{boost 44}
\int_{-\eps}^\eps \int_{\R^2} \bigl|u\circ L_\nu(t,x)\bigr|^4\,dx\,dt \leq \iint_\Omega \bigl|u(t,x)\bigr|^4\,dx\,dt < \infty.
\end{equation}
This estimate allows one to justify the elementary manipulations which guarantee that $u\circ L_\nu$ is a distributional solution.
To prove that it is a strong solution we need to prove that it belongs to $C^0_t(H^1_x\times L^2_x)$.  We will settle both this and part (ii) of the
corollary by showing that
$$
(t,\nu) \mapsto \bigl(u\circ L_\nu(t,x), [u \circ L_\nu]_t(t,x) \bigr)
$$
is a continuous function from $\{|t|<\eps,\ |\nu|<\eps\}$ to $H^1_x\times L^2_x$.

Let $\ulin$ denote the solution to the linear Klein--Gordon equation that has the same initial data as $u$, namely,
$$
\ulin(t) = \cos(t\jpn) u(0) + \jpn^{-1}\sin(t\jpn) u_t(0),
$$
and let $\tilde u = u - \ulin$ denote the difference.  The action of Lorentz boosts on solutions of the linear equation was described in subsection~\ref{SS:boosts};
in particular,
\begin{equation} \label{E:lin o L is Wa}
\ulin\circ L_\nu (t,x) + i\jpn^{-1}\partial_t [\ulin\circ L_\nu] (t,x) = e^{-it\jpn} \W_\nu [ u(0) + i \jpn^{-1} u_t(0) ].
\end{equation}
The action of $e^{-it\jpn}\W_\nu$ on the Fourier side (cf. Lemma~\ref{L:boost action}) clearly shows that
$$
(t,\nu) \mapsto \bigl(\ulin\circ L_\nu(t,x), [\ulin \circ L_\nu]_t(t,x) \bigr)
$$
has the required continuity.  This leaves us to consider the effect of Lorentz boosts on $\tilde u$, which obeys
$$
\tilde u_{tt} - \Delta \tilde u + \tilde u=-\mu u^3
\qtq{and} \tilde u(0,x)=\tilde u_t(0,x)=0.
$$
As $\tilde u=u-\ulin$, Lemma~\ref{L:boostable} and the Strichartz inequality imply
\begin{gather}\label{tilde S}
\| \tilde u \|_{L^q_t L^r_x(\Omega)} + \| \nabla_{t,x} \tilde u \|_{L^\infty_t L^2_x(\Omega)} < \infty
    \quad\text{for each $2<q\leq\infty$ and $\tfrac1q + \tfrac1r=\tfrac12$.}
\end{gather}
We also have
\begin{align}\label{E:boost tightness 2}
\lim_{R\to\infty} \sup_{|t|<\eps R} \int_{|x|>R} \bigl[ |\tilde u_t(t,x)|^2 + |\nabla \tilde u(t,x)|^2 + |\tilde u(t,x)|^2 \bigr] \,dx = 0.
\end{align}
Again this follows by writing $\tilde u=u-\ulin$:  For $u$ we use \eqref{E:boost tightness}; the analogous estimate for $\ulin$ follows from
finite speed of propagation and energy conservation (cf. the proof of \eqref{E:boost tightness}).

We now turn to the main part of the argument.  We will give complete details for the proof that $\tilde u \circ L_\nu$ and its time
derivative are bounded in the requisite spaces and that
continuity holds at the point $(t=0,\nu=0)$.  The reader should have little difficulty adapting the argument to prove continuity
at other points, for example, by using the group property of the transformations.

Let $\mathcal{T}$ denote the stress-energy tensor for $\tilde u$, which has components
\begin{gather*}
\mathcal{T}^{00}= \tfrac12 |\tilde u_t|^2 + \tfrac12|\nabla \tilde u|^2 + \tfrac12|\tilde u|^2,
    \qquad \mathcal{T}^{0j}=\mathcal{T}^{j0}=-\tilde u_t \tilde u_j,\\
\text{and} \qquad \mathcal{T}^{jk} = \tilde u_j \tilde u_k -\delta_{jk}\bigl[ \mathcal{T}^{00} - |\tilde u_t|^2 \bigr],
\end{gather*}
where $j,k\in\{1,2\}$.  Note that this is the stress-energy tensor associated to the linear Klein--Gordon equation (cf. \eqref{E:stressE} below).
For our immediate purposes, we do not need to consider the full tensor, but merely the $3$-vector $\mathfrak{p}$ with components
$$
\mathfrak{p}^\alpha := \jp{\nu} \mathcal{T}^{0\alpha} + \nu_1 \mathcal{T}^{1\alpha} + \nu_2 \mathcal{T}^{2\alpha}, \qquad \alpha\in\{0,1,2\}.
$$
This vector has divergence
$$
\nabla_{t,x} \cdot \mathfrak{p} = \partial_t \mathfrak{p}^0 + \partial_1 \mathfrak{p}^1 + \partial_2 \mathfrak{p}^2
    = -\mu u^3[\jp{\nu} \tilde u_t - \nu\cdot\nabla\tilde u],
$$
and was deliberately constructed so that
\begin{align*}
\int_{L_\nu(t,\R^2)} \mathfrak{p} \cdot d\mathbf{S} &= \int_{\R^2} [\jp{\nu} \mathfrak{p}^0 + \nu_j \mathfrak{p}^j]\circ L_\nu(t,x) \,dx \\
    &=  \tfrac12 \int_{\R^2} |\partial_t (\tilde u \circ L_{\nu})|^2 + |\nabla(\tilde u \circ L_{\nu})|^2 + |\tilde u \circ L_{\nu}|^2 \,dx,
\end{align*}
where $d\mathbf{S}$ denotes surface measure times the unit normal.

Both $\tilde u$ and $\nabla_{t,x}\tilde u$ vanish on the surface $t=0$. Thus we may estimate the required norm by applying the divergence theorem
to $\mathfrak{p}$ on the region
$$
\Omega_{t,\nu} := \{(s,y) : 0< s < \jp{\nu}^{-1} (t-\nu\cdot y)\} \cup \{ (s,y) : \jp{\nu}^{-1}(t-\nu\cdot y) < s <0\}\subseteq \Omega,
$$
whose boundary comprises $(0, \R^2)\cup L_\nu (t, \R^2)$.
There are two technical obstacles to doing this: $\mathfrak{p}$ may not be smooth enough and $\Omega_{t,\nu}$ is not compact. The former can be dealt with
by the usual mollification technique of convolving with a $C^\infty_c$ function.  The latter was the reason for proving \eqref{E:boost tightness 2}, as we will explain.

The main estimate required to prove boundedness and continuity is the following:
\begin{align} \label{E:1}
\iint_{\Omega_{t,\nu}} |\nabla_{t,x} \cdot \mathfrak{p}| \,dy\,ds
    & \lesssim \jp{\nu} \iint_{\Omega_{t,\nu}} |u(s,y)|^3 |\nabla_{t,y} \tilde u(s,y)| \,dy\,ds \\ \notag
&\lesssim \| u \|_{L^3_s L^6_y(\Omega_{t,\nu})} \| \nabla_{t,x} \tilde u \|_{L^\infty_s L^2_y(\Omega)}\to 0 \qtq{as} (t,\nu)\to 0.
\end{align}
The last step follows from Lemma~\ref{L:boostable}, \eqref{tilde S}, and the dominated convergence theorem since $\Omega_{t,\nu}\to \varnothing$ as $(t,\nu)\to 0$.

We are now ready to apply the divergence theorem.  Let $\phi:\R\to[0,1]$ be a smooth function with $\phi(r)=1$ when $r<1$ and $\phi(r)=0$ when $r>2$.
Now let $\psi_R(s,y)=\phi(\tfrac{|s|+|y|}{R})$ where $R>0$ will be sent to infinity.  Applying the divergence theorem to $\mathfrak{p} \psi_R$
and invoking \eqref{E:boost tightness 2} and \eqref{E:1} yields
\begin{align} \notag
\tfrac12 \int_{\R^2}|\partial_t (\tilde u \circ L_{\nu})|^2 &+ |\nabla(\tilde u \circ L_{\nu})|^2 + |\tilde u \circ L_{\nu}|^2 \,dx\\ \notag
&= \lim_{R\to \infty}\tfrac12 \int_{\R^2} \bigl[ |\partial_t (\tilde u \circ L_{\nu})|^2 + |\nabla(\tilde u \circ L_{\nu})|^2 + |\tilde u \circ L_{\nu}|^2\bigr]\psi_R \,dx \\ \notag
&\leq \limsup_{R\to \infty} \iint_{\Omega_{t,\nu}} |\psi_R\nabla_{t,x} \cdot \mathfrak{p}| + |\mathfrak{p} \cdot \nabla_{t,x} \psi_R|\,dy\,ds\\ \notag
&\leq \iint_{\Omega_{t,\nu}} |\nabla_{t,x} \cdot \mathfrak{p}| + \limsup_{R\to \infty} \frac1R \! \int_{-\eps R}^{\eps R} \int_{|x|\sim R} |\jp{\nabla_{t,x}} \tilde u|^2\,dx\,dt\\ \label{E:2}
&\leq  \iint_{\Omega_{t,\nu}} |\nabla_{t,x} \cdot \mathfrak{p}| \longrightarrow 0 \qtq{as} (t,\nu)\to 0.
\end{align}
This settles parts (i) and (ii) of the corollary.

The proof of part (iii) revolves around the stress-energy tensor associated to the nonlinear Klein--Gordon equation:
\begin{equation}\label{E:stressE}
\begin{gathered}
\mathcal{T}^{00}= \tfrac12  u_t^2 + \tfrac12|\nabla u|^2 + \tfrac12|u|^2 + \tfrac{\mu}4 |u|^4,
    \qquad \mathcal{T}^{0j}=\mathcal{T}^{j0}=-u_t u_j,\\
\text{and} \qquad \mathcal{T}^{jk} = u_j  u_k -\delta_{jk}\bigl[ \mathcal{T}^{00} - | u_t|^2 \bigr].
\end{gathered}
\end{equation}
As $u$ is a solution, this is divergence free, that is,
$$
\partial_t \mathcal{T}^{\alpha 0} + \partial_1 \mathcal{T}^{\alpha 1} + \partial_2 \mathcal{T}^{\alpha 2}=0 \qtq{for all} \alpha\in\{0, 1,2\}.
$$
Applying the divergence theorem for all values of $\alpha$, one deduces that \eqref{E:boosted EP}, and hence \eqref{E:Einstein}, holds.  The estimates needed to deal with the non-compactness
of $\Omega_{t,\nu}$ can be found in Lemma~\ref{L:boostable}.
\end{proof}

\begin{corollary}[Boosting to zero momentum] \label{C:rest mass} Let $(u_0, u_1) \in H^1_x \times L^2_x$.  In the focusing case assume also that $M(u(0)) < M(Q)$ and $E(u) < E(Q)$.
Let $u:\R \times \R^2 \to \R$ be the solution to \eqref{nlkg} with this initial data.  Then there exists $\nu \in \R^2$ such that $u^{\nu} := u \circ L_{\nu}$ is a global (strong)
solution to \eqref{nlkg} with
\begin{align}
\label{E:P=0 C}
P(u^{\nu}) &= 0,\\
\label{E:E<E(u)}
E(u^{\nu}) &\leq E(u),\\
\label{E:M<M0}
M(u^{\nu}(0)) &< M(Q) \qtq{in the focusing case.}
\end{align}
\end{corollary}

\begin{remark}\label{R:3:10} In view of \eqref{E:boosted EP}, if $u$ is not identically zero then there is at most one value of $\nu$ (cf.\ \eqref{E:def nu}) so that $P(u^\nu)=0$.
Moreover, by \eqref{E:Einstein} we have $E(u^{\nu}) = E(u)$ if and only if $P(u^{\nu}) = P(u) =  0$.
\end{remark}

\begin{proof}
Without loss of generality, we may assume that $P(u)\neq 0$, the result being trivial otherwise.  In particular, this means that $u(t)\not\equiv 0$ for all $t$.

By Remark~\ref{R:big eps}, for every $t_0 \in \R$, the function $u(\cdot - t_0)$ satisfies the conclusions of Lemma~\ref{L:boostable} for every $0 < \eps < 1$.
Thus by Corollary~\ref{C:boostable}, $u^{\nu}$ is a global strong solution to \eqref{nlkg} for each $\nu \in \R^2$.

By \eqref{E:boosted EP}, it is possible to find $\nu$ such that $P(u^{\nu}) = 0$ if and only if
\begin{equation} \label{E:P<E}
|P(u)| < E(u).
\end{equation}
Moreover, in this case, one must choose
\begin{equation} \label{E:def nu}
\nu = - \frac{P(u)}{\sqrt{E(u)^2-|P(u)|^2}}.
\end{equation}

To see that \eqref{E:P<E} holds in our case we observe that
\begin{equation}\label{Why P<E}
|P(u)| < \|\nabla u\|_{L^2_x} \|u_t\|_{L^2_x} \leq \tfrac12\|\nabla u\|_{L^2_x}^2 + \tfrac12 \|u_t\|_{L^2_x}^2 \leq E(u).
\end{equation}
The first step here is the Cauchy--Schwarz inequality; note that equality would imply that $\nabla u(t)$ points in only one direction for all $x\in\R$, which is inconsistent with
$u(t)\not\equiv 0$.  The last inequality in \eqref{Why P<E} rests solely on the definition of $E(u)$ in the defocusing case and is an application of \eqref{E:Grad trap} in
the focusing case.

The bound \eqref{E:E<E(u)} is a trivial consequence of \eqref{E:Einstein}.  Indeed, we have
\begin{equation} \label{E:snu}
E(u^{s\nu}) \leq E(u) \qtq{for all} 0 \leq s \leq 1.
\end{equation}
This follows from the fact that
$$
E(u^{s\nu}) = \jp{s\nu} E(u) + s\nu \cdot P(u),
$$
which is convex in $s$ and is bounded by $E(u)$ at both 0 and 1.  By the sharp Gagliardo--Nirenberg inequality (cf.\ the proof of \eqref{E:Mass trap}), $M(u^{s\nu}(0)) < M(Q)$ implies that
$$
M(u^{s\nu}(0)) \leq 2E(u^{s\nu}) \leq 2E(u) < M(Q).
$$
Thus by the continuity established in Corollary~\ref{C:boostable}, we must have $M(u^{s\nu}(0)) < M(Q)$ for each $0 \leq s \leq 1$, and in particular at $s=1$.  This settles \eqref{E:M<M0}
and completes the proof of the corollary.
\end{proof}


\section{Refinements of the Strichartz inequality}

The goal of this section is to prove an inverse Strichartz inequality (Theorem~\ref{T:InvStrich}), which is an essential ingredient in the concentration compactness argument.
As described in the introduction, we will use Tao's sharp bilinear restriction estimate; sharpness here refers to the minimality of the spacetime integrability exponent.
This sharpness is not important for our purposes; what is important is that it provides exponents beyond the range of the Strichartz inequality.

To apply Tao's estimate, we need to reduce to a fixed compact set in Fourier space.  This is the role of Lemmas~\ref{L:Annular}~and~\ref{L:BilinearStrich}.  The former lemma demonstrates that
a linear evolution cannot be big without a significant contribution from one of its Littlewood--Paley pieces.  The proof employs the same ideas as proofs of the inverse Sobolev inequalities
needed in energy-critical cases; however, the precise formulation and argument reflects the authors' particular perspective, as advertised in \cite{ClayNotes}.

In the $L^2_x$-critical case (as opposed to the $L^2_x$-supercritical regime), the characteristic frequency scale of a minimal blowup solution can be arbitrarily small.  Also, to a first
approximation, Lorentz boosts act as translations on the Fourier side (cf. Lemma~\ref{L:boost action}); thus we see that the characteristic length scale of a wave packet is not indicated
by the Littlewood--Paley annulus to which it belongs.  We capture the dominant portion of the boost parameter by subdividing annuli into tubes of unit width; these tubes are natural since
they are images of the unit cube under the action of $\ell_\nu$.  The proof that a large linear evolution can be attributed to some tube rests on the bilinear Strichartz inequality
Lemma~\ref{L:BilinearStrich}.

\begin{lemma}[Annular decoupling]\label{L:Annular} For $f\in H^{1/2}_x$,
\begin{equation*}
\bigl\| e^{-it\jpn} f \bigr\|_{L^{4}_{t,x}}^2 \lesssim \sup_N \bigl\| e^{-it\jpn} f_N \bigr\|_{L^{4}_{t,x}}  \bigl\| f \bigr\|_{H^{1/2}_x} .
\end{equation*}
\end{lemma}

\begin{proof}
Using the Littlewood--Paley square function estimate, the Strichartz inequality \eqref{E:half Strichartz}, and Bernstein's inequality,
\begin{align*}
\bigl\| e^{-it\jpn} & f \bigr\|_{L^{4}_{t,x}}^4 \\
&\lesssim \sum_{M,N} \int_\R \int_{\R^2}  \bigl|e^{-it\jpn} f_M \bigr|^2 \bigl|e^{-it\jpn} f_N\bigr|^2  \,dx\,dt \\
&\lesssim \sum_{M\leq N} \bigl\|e^{-it\jpn} f_M\bigr\|_{L^3_tL^6_x} \bigl\|e^{-it\jpn} f_M\bigr\|_{L^4_{t,x}} \bigl\|e^{-it\jpn} f_N \bigr\|_{L^4_{t,x}} \bigl\|e^{-it\jpn} f_N\bigr\|_{L^6_tL^3_x} \\
&\lesssim \sup_K \bigl\|e^{-it\jpn} f_K\bigr\|_{L^4_{t,x}}^2 \sum_{M\leq N} \bigl\| f_M\bigr\|_{H^{2/3}_x} \bigl\| f_N \bigr\|_{H^{1/3}_x} \\
&\lesssim \sup_K \bigl\|e^{-it\jpn} f_K\bigr\|_{L^4_{t,x}}^2 \sum_{M\leq N} \frac{\jp{M}^{1/6}}{\jp{N}^{1/6}} \bigl\| f_M \bigr\|_{H^{1/2}_x} \bigl\| f_N \bigr\|_{H^{1/2}_x} \\
&\lesssim \sup_K \bigl\|e^{-it\jpn} f_K\bigr\|_{L^4_{t,x}}^2 \bigl\| f \bigr\|_{H^{1/2}_x}^2.
\end{align*}
The last step is an application of Schur's test.
\end{proof}

As noted, this lemma shows that if the $L^4_{t,x}$~norm of the free evolution of $f$ is large, then one of its Littlewood--Paley pieces must take responsibility for this.
To find the wave packets inside the evolution of $f$ that are responsible, we need to subdivide each dyadic annulus into tubes.  More accurately, they are strips or sectors
in the two-dimensional case we are discussing; nevertheless, we use vocabulary adapted to the case of arbitrary dimension.

\begin{definition}\label{D:Tubes}
Given $N\in 2^\Z$ with $N\geq 1$, we (almost-everywhere) cover the Fourier support of $P_N$ by a finite collection $\mathcal T_N$ of non-overlapping tubes.  When $N=1$, the collection
consists of just one element, $[-\tfrac{99}{98},\tfrac{99}{98}]^2$.  For $N\geq 2$, we choose
$\mathcal T_N:=\{T_N^k : 0\leq k < 20 N \}$ with
$$
T_N^k := \bigl\{\xi : \tfrac12 N < |\xi| < \tfrac{99}{98} N \text{ and } |\Arg(\xi)-\tfrac{2\pi k}{20\,N}| < \tfrac{\pi}{20\,N}  \bigr\},
$$
where $\Arg(\xi)$ is the angle between $\xi$ and the positive horizontal axis.  Given a tube $T\in\mathcal T:=\cup_N \mathcal T_N$, we define the \emph{center} as follows:
$c(T)=0$ if $T$ is the unique $T\in\mathcal T_1$ and $c(T_N^k)$ is defined by $|c(T_N^k)|=3N/4$ and $\Arg(c(T_N^k))=\tfrac{2\pi k}{20\,N}$.

Associated to each $T\in\mathcal T_N$, we define a Fourier restriction operator $P_T$.  To this end, let $\psi:\R\to[0,1]$ be smooth,
obey $\psi(\theta)=1$ when $|\theta|\leq\tfrac{4}{10}$, $\supp(\psi)\subseteq [-\tfrac6{10},\tfrac6{10}]$, and form a partition of unity via $\sum_{k\in\Z} \psi(\theta-k) =1$
for all $\theta\in\R$.  Now define
$$
\widehat{f_{T_N^k}} (\xi) := \widehat{\! P_{T_N^k} f}(\xi) := \psi\bigl( \tfrac{20N}{2\pi}\Arg(\xi) - k\bigr) \widehat{\!P_N f}(\xi)
$$
except in the special case $N=1$.  In this latter case, $T=[-\tfrac{99}{98},\tfrac{99}{98}]^2$ and
$$
\widehat{f_T} (\xi) := \widehat{\! P_T f}(\xi) := \widehat{\!P_1 f}(\xi).
$$
\end{definition}

\begin{remark}\label{R:boost tube}
Careful but mundane computation shows that the Lorentz transformation with parameter $\nu=c(T)$ maps the tube $T$ into
the ball $|\xi|\leq 2$.  More precisely, the Fourier support of $\W_\nu^{-1} f_T$ lies inside this ball.
\end{remark}

The principal significance of this remark is that it allows us to prove the following, which will be used in the proof of Theorem~\ref{T:InvStrich}.

\begin{lemma}\label{L:ditch tube} Given a tube $T\in\mathcal T_N$ let $\nu=c(T)$ be its center.  Then
\begin{equation*}
\bigl\| \W_{\nu}^{-1} P_T  f \bigr\|_{L^4_x} \lesssim \| P_{\leq 2} \W_{\nu}^{-1} f\|_{L^4_x}
\end{equation*}
uniformly in $T$ and $N$.
\end{lemma}

\begin{proof}
In view of Remark~\ref{R:boost tube}, \eqref{E:LP defn}, and \eqref{E:Fourier Boost}, we need to prove boundedness of the multiplier $m(\nabla)$ defined by
$$
m(\tilde\xi) = [\phi(\xi/N) - \phi(2\xi/N)] \psi\bigl( \tfrac{20N}{2\pi}\Arg(\xi) - k\bigr) =: \tilde m(\xi),
$$
where $\tilde\xi=\ell_\nu(\xi)$ as in \eqref{tildexi from xi}.  (Minor modifications are needed when $N=1$.)

We will apply the Mikhlin Multiplier Theorem (in the form of \cite[\S IV.3.2]{stein:small}); as $m$ is supported inside $\{|\tilde \xi|\leq 2\}$, this is perhaps a little excessive.
For the single tube when $N=1$ (or indeed any finite collection of $N$) the result follows easily.  Our only obligation is to check the uniformity as $N\to\infty$.

The computations that led to \eqref{E:Boost Jacobian} show that
\begin{equation}
\begin{bmatrix}\partial_{\tilde\xi^\|} \\ \partial_{\tilde\xi^\perp} \end{bmatrix}  = A(\tilde\xi) \begin{bmatrix} \jp{\nu} \partial_{\xi^\|} \\ \partial_{\xi^\perp} \end{bmatrix}
\end{equation}
where the entries of the matrix $A(\tilde\xi)$ obey symbol estimates of order zero uniformly in $\nu$.  On the other hand, direct computation shows
$$
\Bigl| \bigl(\jp{\nu} \partial_{\xi^\|}\bigr)^\alpha (\partial_{\xi^\perp})^\beta \tilde m(\xi) \Bigr| \lesssim_{\alpha,\beta} 1.
$$
The result then follows by combining these estimates via standard symbol manipulations.
\end{proof}

\begin{lemma}[Bilinear Strichartz]\label{L:BilinearStrich}  Fix $N\geq 1$ and let $T_1,T_2\in \mathcal T_N$.  Suppose $f,g\in L^2_x(\R^2)$ obey $\supp \hat f \subseteq T_1$
and $\supp \hat g \subseteq T_2$.  Then
\begin{equation}
\bigl\| e^{-it\jpn} f \,e^{-it\jpn} g \bigr\|_{L^{2}_{t,x}} \lesssim \frac{N}{\jp{\dist(T_1,T_2)}} \|f\|_{L^2_x} \|g\|_{L^2_x}.
\end{equation}
\end{lemma}

\begin{proof}
By the Strichartz inequality Lemma~\ref{L:half Strichartz}, we need only consider the case when $\dist(T_1,T_2)\geq 100$.  Moreover, by rotation symmetry we may assume that
$T_1$ lies along the $\xi_1$ axis.

Given a generic $F\in L^2(\R\times\R^2)$, let
\begin{align*}
I:={}& \iint_{\R\times\R^2} \hat F(t,x) [e^{-it\jpn} f](x) [e^{-it\jpn} g](x)\,dx\,dt \\
={}& \iint_{\R^2\times\R^2} F\bigl(-\jp{\xi}-\jp{\eta},\xi+\eta\bigr) \hat f(\xi) \hat g(\eta) \,d\xi\,d\eta.
\end{align*}
Next we change variables according to $\omega=-\jp{\xi}-\jp{\eta}$, $\zeta=\xi+\eta$, and $\beta=\xi_2$.
On the support of the integrand, the Jacobian satisfies
$$
J^{-1} = \Bigl| \frac{\partial(\omega,\zeta,\beta)}{\partial(\xi,\eta)}\Bigr| = \Bigl| \frac{\xi_1}{\jp{\xi}} - \frac{\eta_1}{\jp{\eta}} \Bigr| \gtrsim \frac{\jp{\dist(T_1,T_2)}^2}{N^2}
$$
and $\beta$ varies over an interval of length $O(1)$. Using this information and the Cauchy--Schwarz inequality,
\begin{align*}
|I| &\lesssim \biggl| \iiint_{\R\times\R^2\times\R} F(\omega,\zeta) \hat f(\xi(\omega,\zeta,\beta)) \hat g(\eta(\omega,\zeta,\beta)) J\,d\omega\,d\zeta\,d\beta \biggr| \\
&\lesssim \|F\|_{L^2} \int_\R \biggl( \iint_{\R^2\times\R} \bigl|\hat f(\xi(\omega,\zeta,\beta))\bigr|^2 \bigl|\hat g(\eta(\omega,\zeta,\beta))\bigr|^2 J^2 \,d\omega\,d\zeta \biggr)^{\frac12} \,d\beta\\
&\lesssim \|F\|_{L^2} \biggl( \iiint_{\R\times\R^2\times\R} \bigl|\hat f(\xi(\omega,\zeta,\beta))\bigr|^2 \bigl|\hat g(\eta(\omega,\zeta,\beta))\bigr|^2 J^2 \,d\omega\,d\zeta\,d\beta\biggr)^{\frac12} \\
&\lesssim \|F\|_{L^2} \biggl( \iint_{\R^2\times\R^2} \bigl|\hat f(\xi)\bigr|^2 \bigl|\hat g(\eta)\bigr|^2 J \,d\xi\,d\eta\biggr)^{\frac12}\\
&\lesssim \|F\|_{L^2} \|J\|_{L^\infty}^{1/2} \|f\|_{L^2_x} \|g\|_{L^2_x} .
\end{align*}
The inequality now follows from duality and our bound on $J$.
\end{proof}

\begin{corollary}[Tube decoupling]\label{C:tube decoupling}
Given a dyadic $N\geq 1$ and an $f\in L^2_x(\R^2)$,
\begin{equation}
\bigl\| e^{-it\jpn} P_N f \bigr\|_{L^{4}_{t,x}}^4 \lesssim N^{5/2} \sup_{T\in\mathcal T_N} \bigl\| e^{-it\jpn} P_T f \bigr\|_{L^{4}_{t,x}}  \bigl\| P_N f \bigr\|_{L^2_x}^3
\end{equation}
where $\mathcal T_N$ and $P_T$ are as in Definition~\ref{D:Tubes}.
\end{corollary}

\begin{proof}
By H\"older's inequality, the fact that $\sum_{T\in \mathcal T_N} P_T = P_N$, and then Lemma~\ref{L:BilinearStrich},
\begin{align*}
\bigl\| e^{-it\jpn} & P_N f \bigr\|_{L^4_{t,x}}^2 \\
&\lesssim \sum_{T,T'} \bigl\| e^{-it\jpn} P_T f \bigr\|_{L^4_{t,x}}^{1/4} \bigl\| e^{-it\jpn} P_{T'} f \bigr\|_{L^4_{t,x}}^{1/4}
    \bigl\| e^{-it\jpn} P_T f e^{-it\jpn} P_{T'} f \bigr\|_{L^2_{t,x}}^{3/4} \\
&\lesssim \sup_{T''\in\mathcal T_N} \bigl\| e^{-it\jpn} P_{T''} f \bigr\|_{L^{4}_{t,x}}^{1/2}  \sum_{T,T'} \frac{N^{3/4}}{\jp{\dist(T,T')}^{3/4} } \|P_T f\|_{L^2_x}^{3/4} \|P_{T'} f\|_{L^2_x}^{3/4}  .
\end{align*}
Next, by applying the Hardy--Littlewood--Sobolev inequality in the sum and then the orthogonality of all but adjacent $f_T$, we get
\begin{align*}
\bigl\| e^{-it\jpn} P_N  f \bigr\|_{L^4_{t,x}}^2
&\lesssim N^{3/4} \sup_{T''\in\mathcal T_N} \bigl\| e^{-it\jpn} P_{T''} f \bigr\|_{L^{4}_{t,x}}^{1/2} \biggl(\sum_T \|P_T f\|_{L^2_x}^{6/5}\biggr)^{5/4}\\
&\lesssim N^{5/4} \sup_{T''\in\mathcal T_N} \bigl\| e^{-it\jpn} P_{T''} f \bigr\|_{L^{4}_{t,x}}^{1/2} \bigl\| P_N f \bigr\|_{L^2_x}^{3/2},
\end{align*}
which yields the claim.
\end{proof}

\begin{theorem}[Bilinear Restriction, \cite{taoGAFA03}]\label{T:bilin tao}
Let $f,g\in L^2_x(\R^2)$ have Fourier support in the region $|\xi|\leq 4$ and suppose that for some $c>0$,
$$
M:=\dist(\supp \hat f,\supp \hat g) \geq c \max \{\diam(\supp \hat f),\diam(\supp \hat g)\}.
$$
Then for $q>\frac{5}{3}$,
$$
\bigl\| [e^{-it\jpn} f][e^{-it\jpn} g]\bigr\|_{L^q_{t,x}} \lesssim_c M^{2-\frac{4}{q}}\|f\|_{L^2_x} \|g\|_{L^2_x}.
$$
\end{theorem}

\begin{proof}
While the main result in \cite{taoGAFA03} is stated for the Schr\"odinger propagator, the discussion in Section~9 of that paper explains how it extends to compact surfaces
whose principal curvatures are strictly positive, in our case, $\{(\jp\xi,\xi) : |\xi|\leq 4\}$.
\end{proof}

This bilinear estimate can be used to obtain a form of refined Strichartz inequality showing that if the free evolution is large, then the Fourier transform of the
initial data must concentrate on some cube.  This idea was first developed in \cite{MVV1999} in the 2D Schr\"odinger setting.  By modifying their arguments (and
those of \cite{BegoutVargas} for higher dimensions) it was shown in \cite[\S4.4]{ClayNotes} that the free evolution of this bubble of Fourier concentration must have nontrivial
spacetime norm.  These arguments apply equally well to the Klein--Gordon setting once one has the appropriate bilinear estimate, Theorem~\ref{T:bilin tao}.  In this
way, we obtain the following:

\begin{corollary}[Cube decoupling]\label{C:cube decoupling}
For $f\in L^2_x(\R^2)$ with $\supp \hat f\subseteq\{|\xi|\leq4\}$,
\begin{align*}
\bigl\| e^{-it\jpn} f \bigr\|_{L^{4}_{t,x}(\R\times\R^{2})}
    &\lesssim  \|f\|_{L_x^2(\R^2)}^{\frac34} \biggl( \sup_{Q} |Q|^{-\frac{3}{22}}
        \bigl\| e^{-it\jpn} f_Q \bigr\|_{L^{11/2}_{t,x}(\R\times\R^2)} \biggr)^{\frac14}.
\end{align*}
Here the supremum is taken over all dyadic cubes $Q$ and $f_Q=P_Q f$ denotes the Fourier restriction of $f$ to $Q$.  Without loss of generality,
we may require the cube $Q$ to have side-length no more than eight.
\end{corollary}

\begin{theorem}[Inverse Strichartz]\label{T:InvStrich}
Let $\{f_n\}\subseteq H^1_x(\R^2)$ and suppose that
\begin{equation}\label{E:InvStrich input}
\lim_{n\to\infty}\|f_n\|_{H^1_x(\R^2)} = A \quad\text{and}\quad
\lim_{n\to\infty}\|e^{-it\jpn} f_n\|_{L^4_{t,x}(\R\times\R^2)} = \eps > 0.
\end{equation}
Then there exist a subsequence in $n$, $\phi\in L^2_x(\R^2)$, $\{\lambda_n\}\subseteq [\tfrac{1}{8},\infty)$,
$\{\nu_n\}\subseteq\R^2$, and $\{(t_n,x_n)\}\subseteq \R\times\R^2$ so that we have the following:
\begin{gather}\label{E:InvStrich params converge}
\lambda_n\to\lambda_\infty\in [\tfrac{1}{8},\infty] \qtq{and} \nu_n\to\nu\in\R^2\\
\lambda_\infty < \infty \implies \phi\in H^1_x.
\end{gather}
The functions $f_n$ in the subsequence contain a nontrivial wave packet
\begin{gather} \label{E:InvStrich phi_n}
\phi_n :=\begin{cases} T_{x_n} e^{it_n\jpn} \W_{\nu_n} D_{\lambda_n} \phi   & \text{if $\lambda_{\infty} < \infty$} \\
     T_{x_n} e^{it_n\jpn} \W_{\nu_n} D_{\lambda_n} P_{\leq \lambda_n^{\theta}}\phi  &\text{if $\lambda_{\infty} = \infty$} \end{cases}
\end{gather}
$($with power $\theta=\tfrac1{100})$ in the sense that
\begin{gather}
\label{E:InvStrich H1 decoup}
\lim_{n\to\infty} \| f_n\|_{H^1_x}^2 - \| f_n - \phi_n \|_{H^1_x}^2 - \| \phi_n \|_{H^1_x}^2  =0, \\
\label{E:InvStrich H1 non 0}
\liminf_{n\to\infty} \| \phi_n \|_{H^1_x} \gtrsim \eps \bigl( \tfrac{\eps}{A} \bigr)^{\frac{397}{3}}, \\
\label{E:InvStrich Snorm}
\limsup_{n\to\infty} \bigl\| e^{-it\jpn} (f_n - \phi_n) \bigr\|_{L^4_{t,x}(\R\times\R^2)}
    \leq \eps \Bigl[ 1 - c \bigl(\tfrac{\eps}A\bigr)^C\Bigr]^{1/4},\\
\intertext{and lastly,}
\label{E:InvStrich1}
D_{\lambda_n}^{-1} \W_{\nu_n}^{-1} T_{x_n}^{-1} e^{-it_n\jpn} f_n
    \rightharpoonup \phi \quad\text{weakly in $L^2_x(\R^2)$.}
\end{gather}
Here $c$ and $C$ are constants and all limits as $n\to\infty$ are along the subsequence.
\end{theorem}

\begin{proof}
The final subsequence appearing in the conclusion of the theorem is the result of passing to subsequences on several successive occasions.
For the sake of simplicity/clarity, we do not attempt to represent this with any notation; any ${n\to\infty}$ limit is understood to be
along the subsequence extracted at that point.

By Lemma~\ref{L:Annular} there are dyadic $N_n$ so that
\begin{equation*}
 \bigl\| e^{-it\jpn}  P_{N_n} f_n \bigr\|_{L^4_{t,x}} \gtrsim \eps^2 A^{-1}.
\end{equation*}
Further decomposing this frequency shell into tubes and applying Corollary~\ref{C:tube decoupling} we deduce the
existence of tubes $T_n \in \mathcal T_{N_n}$ so that
\begin{equation*}
\liminf_{n\to \infty} \bigl\| e^{-it\jpn}  P_{T_n} f_n \bigr\|_{L^4_{t,x}} \gtrsim \eps^8 A^{-7} N_n^{1/2} \gtrsim \eps^8 A^{-7}.
\end{equation*}
Incidentally, since the LHS here is $\lesssim A$ by virtue of the Strichartz inequality, we may infer that $N_n\lesssim (A/\eps)^{16}$.

Next, to bring ourselves into the setting of Corollary~\ref{C:cube decoupling}, we apply a Lorentz boost with parameter
$\tilde\nu_n=c(T_n)$ to transport $T_n$ to the origin.  As Lorentz boosts preserve the $L^4_{t,x}$ norm (they are spacetime volume preserving) we have
\begin{equation*}
\liminf_{n\to \infty} \bigl\| e^{-it\jpn} \W_{\tilde\nu_n}^{-1}  P_{T_n} f_n \bigr\|_{L^4_{t,x}} \gtrsim \eps^8 A^{-7}
\end{equation*}
and, as noted in Remark~\ref{R:boost tube}, the Fourier support of $\W_{\tilde\nu_n}^{-1}  P_{T_n} f$ is contained inside the ball $|\xi|\leq 2$.  This remark allowed us
to prove Lemma~\ref{L:ditch tube}, which we now apply to obtain
\begin{equation*}
\liminf_{n\to \infty} \bigl\| e^{-it\jpn} P_{\leq 2} \W_{\tilde\nu_n}^{-1}  f_n \bigr\|_{L^4_{t,x}} \gtrsim \eps^8 A^{-7}
\end{equation*}
and thence, via Corollary~\ref{C:cube decoupling}, the existence of a dyadic cube $Q_n$, so that
\begin{equation}\label{E:cube selected}
 \bigl\| e^{-it\jpn} P_{Q_n}  P_{\leq 2}  \W_{\tilde\nu_n}^{-1}  f_n \bigr\|_{L^{11/2}_{t,x}} \gtrsim \eps^{32} A^{-31} \lambda_n^{-3/11},
\end{equation}
where $\lambda_n^{-1}\leq 8$ denotes the side-length of $Q_n$. We used Lemma~\ref{L:boost action} here to see that
\begin{equation}\label{E:tmp bound}
\bigl\| P_{\leq 2}  \W_{\tilde\nu_n}^{-1}  f_n \bigr\|_{L^2_x} \leq \bigl\| \W_{\tilde\nu_n}^{-1}  f_n \bigr\|_{L^2_x}
\leq \bigl\| \W_{\tilde\nu_n}^{-1}  f_n \bigr\|_{H^{1/2}_x} = \bigl\| f_n \bigr\|_{H^{1/2}_x} \lesssim A.
\end{equation}
As $\lambda_{n}^{-1}\in(0,8]$ we may pass to a subsequence so that $\lambda_n$ converges to some $\lambda_\infty\in [1/8,\infty]$.  As we will need this
later, we set $\xi_n$ equal to the center of the cube $Q_n$; these form a bounded sequence and so passing to a further subsequence we may assume that they converge
to some $\xi_\infty$.

To continue from \eqref{E:cube selected} we first use $L^p$-boundedness of $P_{\leq 2}$ to discard this operator.  Then combining the result with H\"older's inequality
and the Strichartz inequality we obtain
\begin{align*}
\eps^{32} A^{-31} \lambda_n^{-3/11} &\lesssim \bigl\| e^{-it\jpn} P_{Q_n}  \W_{\tilde\nu_n}^{-1}  f_n \bigr\|_{L^{4}_{t,x}} ^{8/11}
    \bigl\| e^{-it\jpn} P_{Q_n} \W_{\tilde\nu_n}^{-1}  f_n \bigr\|_{L^{\infty}_{t,x}} ^{3/11} \\
&\lesssim A^{8/11} \bigl\| e^{-it\jpn} P_{Q_n} \W_{\tilde\nu_n}^{-1}  f_n \bigr\|_{L^{\infty}_{t,x}} ^{3/11},
\end{align*}
which then implies the existence of $\tilde t_n\in\R$ and $\tilde x_n\in \R^2$ so that
\begin{equation}\label{E:big somewhere}
\bigl| P_{Q_n} e^{-i\tilde t_n\jpn} \W_{\tilde\nu_n}^{-1}  f_n \bigr| (\tilde x_n) \gtrsim \lambda_n^{-1} \eps^{\frac{352}3} A^{-\frac{349}3}.
\end{equation}

We have now isolated all the parameters needed to find our bubble.  At present, some are adorned with a tilde because they will be replaced later when we reorder the
symmetries (which do not commute) and because $Q_n$ is not centered at the origin which necessitates an additional Lorentz boost with parameter $\xi_n$.
As in \eqref{E:tmp bound}, Lemma~\ref{L:boost action} implies that
\begin{equation}\label{E:fake seq}
 D_{\lambda_n}^{-1} \W_{\xi_n}^{-1} T_{\tilde x_n}^{-1} e^{-i\tilde t_n\jpn} \W_{\tilde\nu_n}^{-1}  f_n
\end{equation}
form a bounded sequence in $L^2_x$; indeed, since $|\xi_n|\lesssim 1$ and $|\tilde \nu_n|\lesssim N_n\lesssim (A/\eps)^{16}$, the sequence will also be bounded in $H^1_x$
if $\lambda_n$ is bounded, that is, if $\lambda_\infty<\infty$.  Hence, after passing to a subsequence, we have a weak limit, say, $\tilde\phi\in L^2_x$;
moreover, $\tilde\phi\in H^1_x$ when $\lambda_\infty<\infty$.  That this limit has nontrivial norm follows from \eqref{E:big somewhere} as we will now explain.
Let $\hat h$ be the characteristic function of $[-\tfrac12,\tfrac12]^2$ and define
\begin{equation}\label{E:h_n defn}
h_n :=  D_{\lambda_n}^{-1} \W_{\xi_n}^{-1} m_0(\nabla)^{-1} e^{i\xi_n x} D_{\lambda_n} h
\end{equation}
where $m_0(\nabla)=m_0(\nabla,\xi_n)$ is as in Lemma~\ref{L:boost action}.
By Lemma~\ref{L:h_n} these functions lie in a compact set in $L^2_x$ and so we may pass to a subsequence along which we have strong (=norm) convergence, say $h_n\to h_\infty$.  The lemma
also implies that $\| h_n \|_{L^2_x}\lesssim 1$, so
\begin{equation}\label{E:phi v h_n}
\| \tilde\phi \|_{L^2_x} \gtrsim \lim_{n\to\infty} \bigl| \langle h_n,\ \tilde\phi\rangle_{L^2_x} \bigr|.
\end{equation}
Next, using the definition of $\tilde\phi$, Lemma~\ref{L:boost action}, the unitarity of the other symmetries, and finally \eqref{E:big somewhere}, we have
\begin{align}
\lim_{n\to\infty} \langle h_n,\ \tilde\phi\rangle_{L^2_x}
    &= \lim_{n\to\infty} \bigl\langle h_n,\ D_{\lambda_n}^{-1} \W_{\xi_n}^{-1} T_{\tilde x_n}^{-1} e^{-i\tilde t_n\jpn} \W_{\tilde\nu_n}^{-1}  f_n \bigr\rangle_{L^2_x} \notag\\
    &= \lim_{n\to\infty} \bigl\langle T_{\tilde x_n} e^{i\xi_n x} D_{\lambda_n} h,\ e^{-i\tilde t_n\jpn} \W_{\tilde\nu_n}^{-1}  f_n \bigr\rangle_{L^2_x} \notag\\
    &= \lim_{n\to\infty} \lambda_n \bigl[ P_{Q_n} e^{-i\tilde t_n\jpn} \W_{\tilde\nu_n}^{-1}  f_n \bigr] (\tilde x_n) \notag\\
    &\gtrsim \eps^{\frac{352}3} A^{-\frac{349}3}. \label{phi v h}
\end{align}
Combining this with \eqref{E:phi v h_n} yields non-triviality of $\tilde\phi$:
\begin{equation}
\| \tilde\phi \|_{L^2_x} \gtrsim \eps^{\frac{352}3} A^{-\frac{349}3} = \eps \,(\tfrac\eps A\bigr)^{\frac{349}3}.
\end{equation}

Next we use \eqref{E:Boost translates} to reorder the symmetries on the sequence \eqref{E:fake seq} that converges (weakly) to $\tilde\phi$:
\begin{equation}
 D_{\lambda_n}^{-1} \W_{\xi_n}^{-1} T_{\tilde x_n}^{-1} e^{-i\tilde t_n\jpn} \W_{\tilde\nu_n}^{-1}  f_n
    = D_{\lambda_n}^{-1} \W_{\xi_n}^{-1} \W_{\tilde\nu_n}^{-1} T_{x_n}^{-1} e^{-it_n\jpn} f_n,
\end{equation}
where $(-\tilde t_n, -\tilde x_n)=L_{\tilde \nu_n}(-t_n,-x_n)$.  In general, the composition of Lorentz boosts is not a pure boost but also includes a spatial rotation.
We define $\nu_n$ so that $\W_{\xi_n}^{-1} \W_{\tilde\nu_n}^{-1} = R_n \W_{\nu_n}^{-1}$ for some rotation $R_n\in SO(2)$.  As this is a compact group,
we may pass to a subsequence so that $R_n\to R$ and then define $\phi=R^{-1}\tilde \phi$.  Rotations commute with
dilations and preserve the $L^2_x$ and $H^1_x$ norms. Thus
\begin{equation}\label{fn to phi}
D_{\lambda_n}^{-1} \W_{\nu_n}^{-1} T_{x_n}^{-1} e^{-it_n\jpn} f_n \underset{n\to\infty} \rightharpoonup \phi
    \qtq{with} \|\phi \|_{L^2_x} \gtrsim \eps^{\frac{352}3} A^{-\frac{349}3}.
\end{equation}
The weak convergence is in $H^1_x$ when $\lambda_\infty <\infty$ and merely in $L^2_x$ when $\lambda_\infty =\infty$.  Note that the new sequence of
boost parameters $\nu_n$ inherits boundedness from $\tilde\nu_n$ and $\xi_n$; more precisely, $|\nu_n|\lesssim (A/\eps)^{16}$.
By passing to a further subsequence, we can guarantee convergence of $\nu_n$ as stated in \eqref{E:InvStrich params converge}.
In addition, \eqref{fn to phi} settles \eqref{E:InvStrich1}.

Let us now turn to proving decoupling of the $H^1_x$ norm.  We begin with \eqref{E:InvStrich H1 non 0}, treating only the case $\lambda_\infty=\infty$
since the case $\lambda_\infty<\infty$ is similar but easier.  By Lemma~\ref{L:boost action},
\begin{align*}
\| \phi_n \|_{H^1_x} &= \| \W_{\nu_n} D_{\lambda_n} P_{\leq \lambda_n^{\theta}}\phi \|_{H^1_x}
    \gtrsim \jp{\nu_n}^{-1} \| D_{\lambda_n} P_{\leq \lambda_n^{\theta}}\phi \|_{H^1_x} \gtrsim \jp{\nu_n}^{-1} \| P_{\leq \lambda_n^{\theta}}\phi \|_{L^2_x}
\end{align*}
and thence by $\lambda_n\to\infty$, $|\nu_n| \lesssim (A/\eps)^{16}$, and \eqref{fn to phi},
\begin{align*}
\liminf_{n\to\infty} \| \phi_n \|_{H^1_x} \gtrsim (\eps/A)^{16} \| \phi \|_{L^2_x} \gtrsim \bigl(\tfrac\eps A\bigr)^{\frac{397}{3}} \, \eps.
\end{align*}
This yields \eqref{E:InvStrich H1 non 0} as promised.  Next we consider \eqref{E:InvStrich H1 decoup}, beginning with the basic Hilbert space identity
\begin{equation*}
\| f_n\|_{H^1_x}^2 - \| f_n - \phi_n \|_{H^1_x}^2 - \| \phi_n \|_{H^1_x}^2 = 2 \langle f_n - \phi_n,\ \phi_n \rangle_{H^1_x}.
\end{equation*}
Again we confine our discussion to the more interesting $\lambda_n\to\infty$ case, where
\begin{align*}
\langle f_n - \phi_n,&\ \phi_n \rangle_{H^1_x}
    = \bigl\langle T_{x_n}^{-1} e^{-it_n\jpn} f_n - \W_{\nu_n} D_{\lambda_n} P_{\leq \lambda_n^{\theta}}\phi,
        \  \W_{\nu_n} D_{\lambda_n} P_{\leq \lambda_n^{\theta}}\phi \bigr\rangle_{H^1_x}\\
&= \bigl\langle \W_{\nu_n}^{-1} T_{x_n}^{-1} e^{-it_n\jpn} f_n - D_{\lambda_n} P_{\leq \lambda_n^{\theta}}\phi,
    \  m_1(\nabla;\nu_n)^{-1} D_{\lambda_n} P_{\leq \lambda_n^{\theta}}\phi \bigr\rangle_{H^1_x}\\
&= \bigl\langle D_{\lambda_n}^{-1} \W_{\nu_n}^{-1} T_{x_n}^{-1} e^{-it_n\jpn} f_n - P_{\leq \lambda_n^{\theta}}\phi,
    \  \jp{\lambda_n^{-1}\nabla}^2 m_1(\lambda_n^{-1} \nabla;\nu_n)^{-1} P_{\leq \lambda_n^{\theta}}\phi \bigr\rangle_{L^2_x}
\end{align*}
by Lemma~\ref{L:boost action} and~\eqref{E:dilate m}.  Thus, by \eqref{fn to phi},
\begin{gather*}
P_{\leq \lambda_n^{\theta}}\phi \to \phi, \qtq{and}
    \jp{\lambda_n^{-1}\nabla}^2 m_1(\lambda_n^{-1} \nabla;\nu_n)^{-1} P_{\leq \lambda_n^{\theta}}\phi \to \jp{\nu_\infty}^{-1} \phi \quad\text{in $L^2_x$},
\end{gather*}
we deduce $\langle f_n - \phi_n,\phi_n \rangle_{H^1_x} \to 0$ and so \eqref{E:InvStrich H1 decoup} follows.

Finally, we turn to the proof of \eqref{E:InvStrich Snorm} which shows that the free evolution of $\phi_n$ captures (at least as $n\to\infty$)
a positive proportion of the evolution of $f_n$.  Much of the dirty work has been encapsulated for us in Lemma~\ref{L:pointwise}, as we will see.
We present the details in the case $\lambda_n\lesssim 1$ as adapting the argument to $\lambda_n\to\infty$ just involves minor modifications to the formulae that follow.

Using \eqref{E:dilate m} and making the change of variables $t=\lambda_n^2 s$, $x=\lambda_n y$, to account
for the fact that $D_{\lambda_n}^{-1}$ is not an isometry on $L^4_{x}$, we have
\begin{align}
    \notag
\bigl\| e^{-it\jpn} (f_n - \phi_n) \bigr\|_{L^4_{t,x}} &= \bigl\| e^{-it\jpn} ( \W_{\nu_n}^{-1} e^{-it_n\jpn} T_{x_n}^{-1} f_n - D_{\lambda_n} \phi) \bigr\|_{L^4_{t,x}} \\
&= \bigl\| e^{-i\lambda_n^2 s\jp{\lambda_n^{-1}\nabla}} ( D_{\lambda_n}^{-1} \W_{\nu_n}^{-1} e^{-it_n\jpn} T_{x_n}^{-1} f_n - \phi) \bigr\|_{L^4_{s,y}},
    \label{E:almost a.e.}
\end{align}
which throws us into the path of Lemma~\ref{L:pointwise} with $g_n=D_{\lambda_n}^{-1} \W_{\nu_n}^{-1} e^{-it_n\jpn} T_{x_n}^{-1} f_n$ and $g = \phi$.  This lemma allows us to
apply Lemma~\ref{L:BrezisLieb}, or more precisely \eqref{E:BrezisLieb}, and so obtain
\begin{align*}
 \limsup_{n\to\infty} &\bigl\| e^{-i\lambda_n^2 s\jp{\lambda_n^{-1}\nabla}} ( D_{\lambda_n}^{-1} \W_{\nu_n}^{-1} e^{-it_n\jpn} T_{x_n}^{-1} f_n - \phi) \bigr\|_{L^4_{s,y}}^4 \\
\leq{}&\limsup_{n\to\infty}  \bigl\|e^{-i\lambda_n^2 s\jp{\lambda_n^{-1}\nabla}} D_{\lambda_n}^{-1} \W_{\nu_n}^{-1} e^{-it_n\jpn} T_{x_n}^{-1} f_n \bigr\|_{L^4_{s,y}}^4
    - \bigl\| e^{-i\lambda_\infty^2 s\jp{\lambda_\infty^{-1}\nabla}} \phi \bigr\|_{L^4_{s,y}}^4.
\end{align*}
Reversing the computations in \eqref{E:almost a.e.} this becomes
\begin{align*}
\limsup_{n\to\infty} \bigl\| e^{-it\jpn} (f_n - \phi_n)  \bigr\|_{L^4_{t,x}}^4 \leq  \limsup_{n\to\infty}  \bigl\| e^{-it\jpn} f_n  \bigr\|_{L^4_{t,x}}^4
    - \bigl\| e^{-i\lambda_\infty^2 s\jp{\lambda_\infty^{-1}\nabla}} \phi  \bigr\|_{L^4_{s,y}}^4.
\end{align*}
This is very close to implying \eqref{E:InvStrich Snorm}; in view of \eqref{E:InvStrich input}, all that is missing is a lower bound of the form
\begin{equation}\label{S phi >0}
     \bigl\| e^{-i\lambda_\infty^2 s\jp{\lambda_\infty^{-1}\nabla}} \phi  \bigr\|_{L^4_{s,y}} \gtrsim  \eps \bigl(\tfrac\eps A\bigr)^C.
\end{equation}

We will not bother to determine the value of $C$ appearing in \eqref{S phi >0}.  In accordance with this, we will use the symbol $C$
to denote a variety of powers which vary from place to place as we develop the proof of \eqref{S phi >0}.

From \eqref{phi v h} we have
$$
| \langle R^{-1}h_\infty,\ \phi\rangle | \gtrsim \eps \bigl( \tfrac{\eps}{A}\bigr)^C,
$$
where $h_\infty$ is the $L^2_x$-limit of the functions $h_n$ defined in \eqref{E:h_n defn} and $R$ is the limiting rotation defined above; recall $\tilde \phi=R\phi$.
Note that $h_\infty$ inherits the estimates \eqref{E:K fourier decay} enjoyed by $h_n$.  Together with $\|\phi\|_{L^2_x} \lesssim A$ (which follows from
\eqref{E:InvStrich input} and \eqref{E:InvStrich1}) and $|\nu_n|\lesssim (A/\eps)^C$, these estimates guarantee the existence of radii $M,r\sim (A/\eps)^C$ so that
$$
\bigl| \langle e^{-i\lambda_\infty^2 s\jp{\lambda_\infty^{-1}\nabla}} \tilde h,\ e^{-i\lambda_\infty^2 s\jp{\lambda_\infty^{-1}\nabla}} \phi\rangle \bigr|
    = \bigl| \langle \tilde h,\ \phi\rangle\bigr| \gtrsim \eps \bigl( \tfrac{\eps}{A}\bigr)^C,
$$
where $\tilde h = P_{\leq M} \, \chi R^{-1}h_\infty$ and $\chi$ is a smooth cutoff to $\{|x|\leq r\}$.

To complete the proof of \eqref{S phi >0}, we merely need to show
$$
\| e^{-i\lambda_\infty^2 s\jp{\lambda_\infty^{-1}\nabla}} \tilde h \|_{L^{4/3}_x} \lesssim \bigl(\tfrac{A}\eps\bigr)^C
$$
uniformly in $\lambda_\infty$ and $s\in[-1,1]$.
First we note that $\| \tilde h\|_{L^{4/3}_x} \lesssim (A/\eps)^C$, by construction.  To extend this to non-zero values of $s$ we
use the Mikhlin multiplier theorem.  The requisite input is
$$
\bigl| \partial^\alpha_\xi \ s\lambda^2 \jp{\lambda^{-1}\xi} \bigr|  \lesssim_\alpha 1 \qtq{when} |\alpha|\geq 1
$$
uniformly for $s\in[-1,1]$ and $\lambda\gtrsim 1$.  This suffices since $\tilde h$ has Fourier support inside a bounded region, namely, $|\xi|\lesssim (A/\eps)^C$.
\end{proof}

\begin{corollary} \label{Cor:Simplify}
After passing to a further subsequence in $n$ (and possibly changing $\phi$ and $x_n$), we may assume in addition that the parameters in the conclusion of Theorem~\ref{T:InvStrich}
satisfy the following: If $\lambda_n$ does not converge to $+\infty$, then $\lambda_n\equiv 1$ and $\nu_n\equiv 0$.  Moreover, irrespective of the behaviour of
$\lambda_n$, we may assume that $t_n$ obeys $\tfrac{t_n}{\lambda_n^2}\to \pm \infty$ or $t_n\equiv 0$.
\end{corollary}

\begin{proof}
Suppose $\lambda_n \to \lambda_{\infty} \in [\tfrac 18,\infty)$.  Then $D_{\lambda_n}$ and $D_{\lambda_n}^{-1}$ converge strongly to $D_{\lambda_{\infty}}$ and
$D_{\lambda_{\infty}}^{-1}$, respectively, as operators both on $L^2_x(\R^2)$ and $H^1_x(\R^2)$.  Thus we may replace $\phi$ by $D_{\lambda_{\infty}} \phi$
and set $\lambda_n\equiv 1$, whilst retaining the conclusions of Theorem~\ref{T:InvStrich}.  In the case of \eqref{E:InvStrich Snorm}, we invoke the Strichartz inequality.
By the same reasoning, we may replace $\phi$ with $\W_{\nu}\phi$ and set $\nu_n\equiv 0$.

We turn now to the discussion of $t_n$.  By passing to a subsequence, we may assume $\tfrac{t_n}{\jp{\nu_n}\lambda_n^2} \to \tilde t_{\infty} \in [-\infty, \infty]$;
we just need to treat the case $\tilde t_\infty\in \R$.  Invoking \eqref{E:Boost translates}, we have
\begin{align*}
T_{x_n}e^{it_n\jpn} \W_{\nu_n} = T_{x_n - \frac{\nu_n}{\jp{\nu_n}}t_n} \W_{\nu_n} e^{i\jp{\nu_n}^{-1}t_n \jpn}.
\end{align*}
If $\lambda_n \equiv 1$ and $\nu_n \equiv 0$, our assumption guarantees $t_n \to \tilde t_{\infty}$, and hence $e^{it_n \jpn}\phi \to e^{i\tilde t_{\infty}\jpn}\phi$ in $H^1_x(\R^2)$.
In this case, we replace $\phi$ by $e^{i\tilde t_{\infty}\jpn}\phi$ and argue as above.

If $\lambda_n \to \infty$, we pass to a subsequence so that $e^{i\jp{\nu_n}^{-1}t_n}\to e^{it_\infty}$ for some $t_\infty\in [0, 2\pi)$.  Thus by \eqref{phase approx},
$$
e^{i\jp{\nu_n}^{-1}t_n} e^{i\jp{\nu_n}^{-1}t_n(\jp{\lambda_n^{-1}\nabla}-1)}P_{\leq \lambda_n^{\theta}}\phi - e^{it_{\infty}}e^{-i\tilde t_{\infty} \Delta/2}P_{\leq \lambda_n^{\theta}} \phi
        \to 0 \qtq{in $L^2_x(\R^2)$.}
$$
As $D_{\lambda_n}P_{\leq \lambda_n^{\theta}}$ is a bounded operator from $L^2_x(\R^2)$ to $H^1_x(\R^2)$, we may replace $\phi$ by $e^{it_{\infty}}e^{-i\tilde t_{\infty} \Delta/2}\phi$,
set $t_n\equiv 0$, and change $x_n$ to $x_n-\tfrac{\nu_n}{\jp{\nu_n}}t_n$.  We can then argue as above.
\end{proof}


\section{Linear profile decomposition} \label{S:lpd}


From Theorem~\ref{T:InvStrich}, we may deduce the existence of a linear profile decomposition.  We continue to work with the first order Klein--Gordon  equation \eqref{nlkg1st}.

\begin{theorem}[Linear profile decomposition] \label{T:lpd}
Let $\{v_n\}$ be a bounded sequence of $H^1_x(\R^2)$ functions.  Then, after passing to a subsequence, there exists $J_0 \in [1,\infty]$ and for each integer $1\leq j < J_0$ there also exist
\begin{CI}
\item a function $0 \neq \phi^j \in L^2_x(\R^2)$,
\item a sequence $\{\lambda_n^j\} \subset [1,\infty)$ such that either $\lambda_n^j \to \infty$ or $\lambda_n^j \equiv 1$,
\item a sequence $\nu_n^j \to \nu^j \in \R^2$ which is identically $0$ if $\lambda_n^j \equiv 1$,
\item a sequence $\{(t_n^j,x_n^j)\} \subset \R \times \R^2$ such either $t_n^j/(\lambda_n^j)^2\to \pm \infty$ or $t_n^j\equiv 0$.
\end{CI}
Let $P_n^j$ denote the projections defined by
$$
P_n^j \phi^j :=
\begin{cases} \phi^j \in H^1_x(\R^2), &\qtq{if $\lambda^j_n \equiv  1$} \\
P_{\leq (\lambda_n^j)^{\theta}}\phi^j, &\qtq{if $\lambda_n^j \to \infty$,}
\end{cases}
$$
with $\theta = \tfrac{1}{100}$.  Then for all $1 \leq J < J_0$, we have a decomposition
\begin{equation} \label{E:lpd}
v_n = \sum_{j=1}^J T_{x_n^j} e^{it_n^j \jpn}\W_{\nu_n^j}  D_{\lambda_n^j} P_n^j \phi^j + w_n^J,
\end{equation}
satisfying
\begin{gather}
\label{E:lpd ST norms to zero}
\lim_{J \to \infty} \limsup_{n \to \infty} \|e^{-it\jpn}w_n^J\|_{L^4_{t,x}(\R \times \R^2)} = 0, \\
\label{E:lpd H1 decoup}
\lim_{n \to \infty} \biggl\{ \|v_n\|_{H^1_x}^2 - \sum_{j=1}^J \| T_{x_n^j} e^{it_n^j \jpn}  \W_{\nu_n^j}D_{\lambda_n^j} P_n^j \phi^j\|_{H^1_x}^2 - \|w_n^J\|_{H^1_x}^2 \biggr\}= 0, \\
\label{E:lpd wk lim}
D_{\lambda_n^j}^{-1} \W_{\nu_n^j}^{-1}T_{x_n^j}^{-1}  e^{-it_n^j \jpn} w_n^J \rightharpoonup 0, \qtq{weakly in $L^2_x(\R^2)$ for any $j \leq J$.}
\end{gather}
Finally, we have the following orthogonality condition: for any $j \neq j'$,
\begin{gather} \label{E:lpd orthogonality}
\lim_{n \to \infty}\biggl\{ \frac{\lambda_n^j}{\lambda_n^{j'}} + \frac{\lambda_n^{j'}}{\lambda_n^j} + \lambda_n^{j}|\nu_n^j - \nu_n^{j'}|
        +  \frac{|s_n^{jj'}|}{(\lambda_n^{j'})^2}  +  \frac{|y_n^{jj'}|}{\lambda_n^{j'}} \biggr\}= \infty,
\end{gather}
where $(-s_n^{jj'},y_n^{jj'}) := L_{\nu_n^{j'}}(t_n^{j'}-t_n^{j},x_n^{j'}-x_n^{j})$.
\end{theorem}

\begin{proof} Given a sequence $v_n$ as above, by passing to a subsequence we may assume that for some $A_0, \eps_0 \geq 0$,
$$
\lim_{n \to \infty} \|v_n\|_{H_x^1(\R^2)} = A_0 \qtq{and} \lim_{n \to \infty} \|e^{-it\jpn} v_n\|_{L^4_{t,x}(\R \times \R^2)} = \eps_0.
$$
(Note that the Strichartz inequality guarantees $\eps_0 \lesssim A_0$.)  If $\eps_0 = 0$, then we set $J_0 = 1$  and the claim follows; note that in this case there are no $\phi^j$'s.
Otherwise, we apply Theorem~\ref{T:InvStrich} as strengthened by Corollary~\ref{Cor:Simplify} to find $\{\nu_n^1\}$, $\{\lambda_n^1\}$, $\{(t_n^1,x_n^1)\}$, and $\phi^1$.  We set
$$
w_n^1: = v_n - T_{x_n^1}e^{it_n^1\jpn}\W_{\nu_n^1}D_{\lambda_n^1}P_n^1 \phi^1.
$$
Note that by \eqref{E:InvStrich1},
$$\
D_{\lambda_n^1}^{-1} \W_{\nu_n^1}^{-1}T_{x_n^1}^{-1} e^{-it_n^1 \jpn}  w_n^1 \rightharpoonup 0 \qtq{weakly in $L^2_x(\R^2)$,}
$$
which gives \eqref{E:lpd wk lim} when $J=1$.  Moreover, by \eqref{E:InvStrich H1 decoup},
\begin{equation} \label{E:decoup J=1}
\lim_{n \to \infty}\Bigl\{ \|v_n\|_{H^1_x(\R^2)}^2 - \|T_{x_n^1}e^{it_n^1\jpn}\W_{\nu_n^1}D_{\lambda_n^1}P_n^1 \phi^1\|_{H^1_x(\R^2)}^2 - \|w_n^1\|_{H^1_x(\R^2)}^2\Bigr\} = 0,
\end{equation}
which is \eqref{E:lpd H1 decoup} when $J=1$.

By passing to a further subsequence if necessary, we may now assume that
$$
\lim_{n \to \infty}\|w_n^1\|_{H_x^1(\R^2)} = A_1 \qtq{and} \lim_{n \to \infty} \|e^{-it\jpn} w_n^1\|_{L^4_{t,x}(\R \times \R^2)} = \eps_1
$$
for some $A_1, \eps_1\geq0$.  Furthermore, by \eqref{E:decoup J=1} we have $A_1 \leq A_0$, while by \eqref{E:InvStrich Snorm},
$$
\eps_1 \leq \eps_0\bigl[1-c(\tfrac{\eps_0}{A_0})^C\bigr]^{1/4}.
$$
If $\eps_1 = 0$, then we set $J_0 = 2$ and stop.  Otherwise, we apply Theorem~\ref{T:InvStrich} to $w_n^1$ to obtain parameters $\nu_n^2$, $\lambda_n^2$, $(t_n^2,x_n^2)$,
and a function $\phi^2$.  We then set
$$
w_n^2 := w_n^1 - T_{x_n^2}e^{it_n^2\jpn}\W_{\nu_n^2}D_{\lambda_n^2}P_n^2 \phi^2.
$$

Arguing as above, we obtain \eqref{E:lpd wk lim} when $j=J=2$, and also
\begin{equation} \label{E:decoup J=2}
\lim_{n \to \infty} \Bigl\{\|w_n^1\|_{H^1_x(\R^2)}^2 - \|T_{x_n^2}e^{it_n^2\jpn}\W_{\nu_n^2}D_{\lambda_n^2}P_n^2 \phi^2\|_{H^1_x(\R^2)}^2 - \|w_n^2\|_{H^1_x(\R^2)}^2 \Bigr\}= 0.
\end{equation}
Adding \eqref{E:decoup J=1} and \eqref{E:decoup J=2}, we obtain \eqref{E:lpd H1 decoup} when $J=2$.

Continuing in this fashion, one of two things occurs.  Either we reach some finite $j$ so that $\eps_j = 0$, in which case we set $J_0 = j+1$, or we obtain an infinite number of sequences
of parameters $\nu_n^j$, $\lambda_n^j$, $(t_n^j,x_n^j)$, and an infinite collection of functions $\phi^j$, in which case we set $J_0 = \infty$.

At this point, tracing back through the definition of the $w_n^J$, we have a decomposition of the form \eqref{E:lpd}.  That \eqref{E:lpd ST norms to zero} holds is a tautology if the
algorithm terminates and follows from \eqref{E:InvStrich Snorm} otherwise.  The claim \eqref{E:lpd H1 decoup} may be established inductively, arguing as for the case $J=2$ above.
As above, the weak limit in \eqref{E:lpd wk lim} is zero when $j = J$, but to conclude \eqref{E:lpd wk lim} in the case $1\leq j<J$,  as well as the orthogonality condition
\eqref{E:lpd orthogonality}, we will have to do a little more work.  In particular, we will make use of the following:

\begin{lemma}[Orthogonality] \label{L:orthogonality}
For $j \neq j'$ we define a sequence of operators
\begin{equation} \label{E:ortho op}
A_n^{jj'}:=D_{\lambda_n^{j}}^{-1} \W_{\nu_n^{j}}^{-1} T_{x_n^{j}}^{-1} e^{-it_n^{j}\jpn}  T_{x_n^{j'}} e^{it_n^{j'} \jpn} \W_{\nu_n^{j'}} D_{\lambda_n^{j'}}.
\end{equation}
If the orthogonality condition \eqref{E:lpd orthogonality} holds, then $A_n^{jj'}$ converges to zero in the weak operator topology on $\mathcal{B}(L^2_x(\R^2))$.
Conversely, if the orthogonality condition fails, then after passing to a subsequence, both $A_n^{jj'}$ and its adjoint converge to injective operators in the
strong operator topology on $\mathcal{B}(L^2_x(\R^2))$.
\end{lemma}

The deduction of the remaining conclusions in the theorem from this lemma is straightforward.   Indeed, \eqref{E:lpd wk lim} for $1\leq j<J$ follows from \eqref{E:lpd orthogonality} (which
we have yet to prove) and Lemma~\ref{L:orthogonality}.

To obtain \eqref{E:lpd orthogonality}, we argue by contradiction and use the crucial fact that, if defined, $\phi^j \neq 0$.   As an example, consider $j=2$ and choose the minimal $j'$
for which \eqref{E:lpd orthogonality} fails.  For the sake of this example, suppose $j'=4$.  As the orthogonality condition fails, Lemma~\ref{L:orthogonality} guarantees that
$$
A_n^{2\,4}\phi^4\to \psi\neq 0
$$
strongly in $L^2_x$.  On the other hand, as $\phi^4=\wlim  D_{\lambda_n^4}^{-1} \W_{\nu_n^4}^{-1} T_{x_n^4}^{-1}e^{-it_n^4\jpn} w_n^3$ and the adjoint of $A_n^{2\,4}$ converges strongly,
\begin{align*}
\psi&=\wlim_{n\to \infty}A_n^{2\,4}\Bigl[D_{\lambda_n^4}^{-1}\W_{\nu_n^4}^{-1}T_{x_n^4}^{-1}e^{-it_n^4\jpn}
    \bigl(w_n^2 - T_{x_n^3} e^{it_n^3 \jpn} \W_{\nu_n^3}D_{\lambda_n^3}\phi^3\bigr)\Bigr]=0,
\end{align*}
which contradicts $\psi\neq 0$.  Note here that we used the previously proved $j=J$ case of \eqref{E:lpd wk lim} to treat the first term and the minimality of $j'$
and Lemma~\ref{L:orthogonality} to treat the second term.
\end{proof}

We move now to the proof of the lemma.

\begin{proof}[Proof of Lemma~\ref{L:orthogonality}]
Let us first note that the adjoint of $A_n^{jj'}$ is $A_n^{j'j}$; this follows from Lemma~\ref{L:boost action} and the fact that $m_0$ appearing there commutes with translations
and free evolutions. Thus, strong convergence of the adjoint will follow by reversing $j$ and $j'$.

We begin by rewriting $A_n^{jj'}$ in a more convenient form; more precisely, using \eqref{E:Boost translates} we have
$$
A_n^{jj'}=D_{\lambda_n^{j}}^{-1} \W_{\nu_n^{j}}^{-1} \W_{\nu_n^{j'}} T_{y_n^{jj'}} e^{-is_n^{jj'}\jpn} D_{\lambda_n^{j'}} .
$$
Writing $ \W_{\nu_n^{j}}^{-1} \W_{\nu_n^{j'}}= R_n^{jj'}\W_{\nu_n^{jj'}}$ and using the fact that rotations and dilations commute, we obtain
$$
A_n^{jj'}= R_n^{jj'} D_{\lambda_n^{j}}^{-1} \W_{\nu_n^{jj'}}T_{y_n^{jj'}} e^{-is_n^{jj'}\jpn} D_{\lambda_n^{j'}} .
$$
Note that $|\nu_n^{jj'}|  \sim |\nu_n^{j'} - \nu_n^j|$, with the implicit constant depending on the upper bound for $|\nu_n^j|+|\nu_n^{j'}|$.
As rotations form a compact group of unitary operators, we may neglect $R_n^{jj'}$ in what follows.

On the Fourier side, a careful computation using \eqref{E:Fourier Boost} yields that for a function $f \in L^2_x(\R^2)$,
\begin{align*}
\widehat{A_n^{jj'}f}(\xi) &=  \tfrac{\lambda_n^{j'}}{\lambda^j_n}  \tfrac{\jp{\ell_{\nu_n^{jj'}}(\xi/\lambda_n^j)}}{\jp{\xi/\lambda_n^j}} e^{-iy_n^{jj'} \ell_{\nu_n^{jj'}}(\xi/\lambda_n^{j})}
 e^{-is_n^{jj'} \langle \ell_{\nu_n^{jj'}}(\xi/\lambda_n^{j})\rangle}\hat{f}\bigl(\lambda_n^{j'} \ell_{\nu_n^{jj'}}(\tfrac{\xi}{\lambda_n^{j}})\bigr),
\end{align*}
which we rewrite as
\begin{align*}
\widehat{A_n^{jj'}f}(\xi) &= B_n^{jj'}C_n^{jj'}E_n^{jj'}F_n^{jj'}G_n^{jj'} \widehat{f}(\xi)
\end{align*}
with
\begin{gather*}
B_n^{jj'} :=  \tfrac{\jp{\ell_{\nu_n^{jj'}}(\xi/\lambda_n^{j})}}{\jp{\xi/\lambda_n^{j}}},
\qquad
C_n^{jj'} := D_{\lambda_n^j/\lambda_n^{j'}},
\qquad
E_n^{jj'} :=(\ell_{\nu_n^{jj'}}^{\lambda_n^{j'}})^{*}, \\
F_n^{jj'} := e^{-is_n^{jj'}\jp{\xi/\lambda_n^{j'}}},
\qquad
G_n^{jj'} := e^{-iy_n^{jj'}\xi/\lambda_n^{j'}}.
\end{gather*}
Here, we used the notation
$$
\ell_{\nu}(\xi) := \xi^{\perp} + \jp{\nu}\xi^{\parallel} - \jp{\xi}\nu \qtq{and} \ell_{\nu}^{\lambda}(\xi) := \lambda \ell_{\nu}(\lambda^{-1}\xi)
        = \xi^{\perp} + \jp{\nu}\xi^{\parallel} - \jp{\tfrac{\xi}{\lambda}}\lambda \nu,
$$
while $^{*}$ is used to denote the pullback, that is,  $(\ell_{\nu}^{\lambda})^{*} \widehat{f}:= \widehat{f} \circ \ell_{\nu}^{\lambda}$.

To continue, by passing to a subsequence, we may assume the following:
\begin{enumerate}
\item[(i)] Either $\tfrac{\lambda_n^{j'}}{\lambda_n^j} + \tfrac{\lambda_n^j}{\lambda_n^{j'}} \to \infty$ or $\tfrac{\lambda_n^{j'}}{\lambda_n^j} \to \lambda_{\infty} \in (0,\infty)$.
\item[(ii)] Either $\lambda_n^{j'}|\nu_n^j - \nu_n^{j'}| \to \infty$ or there exists a diffeomorphism $\ell_{\infty}$, whose Jacobian is bounded both above and below, such that
$\ell_{\nu_n^{jj'}}^{\lambda_n^{j'}} \to \ell_{\infty}$, uniformly on compact subsets of $\R^2$.
\item[(iii)] Either $|s_n^{jj'}|/(\lambda_n^{j'})^2 \to \infty$ or $s_n^{jj'}/(\lambda_n^{j'})^2 \to s_{\infty} \in \R$.
\item[(iv)] Either $|y_n^{jj'}|/\lambda_n^{j'} \to \infty$ or $y_n^{jj'}/\lambda_n^{j'} \to y_{\infty} \in \R$.
\end{enumerate}

We start by addressing the second half of the lemma.  Assume therefore that \eqref{E:lpd orthogonality} fails.  In this case, it is easy to check that each of $B_n^{jj'}$ through $G_n^{jj'}$
converges strongly to an injective operator and hence so does their product.

It remains therefore to consider the case when \eqref{E:lpd orthogonality} holds.  Directly from the definition, we see that $A_n^{jj'}$ forms a uniformly bounded sequence
of operators on $L_x^2(\R^2)$; thus it suffices to show that
$$
\lim_{n \to \infty} \langle A_n^{jj'} \phi, \psi \rangle_{L^2_x(\R^2)} = 0
$$
for every pair of Schwartz functions $\phi$ and $\psi$ with compact Fourier support.

The sequence of functions
$$
 \xi \mapsto \frac{\bigl\langle \ell_{\nu_n^{jj'}}(\xi/\lambda_n^{j})\bigr\rangle}{\jp{\xi/\lambda_n^{j}}}
$$
is uniformly bounded and converges uniformly on compact sets by virtue of our assumption that the sequences $\nu_n^j,\nu_n^{j'}$ converge in $\R^2$ and that the sequence
$\lambda_n^j$ converges in $[1,\infty]$.  Consequently, $B_n^{jj'}$ and its adjoint converge in the strong operator topology on $\mathcal{B}(L^2_x(\R^2))$, and we may
disregard $B_n^{jj'}$ in what follows.

As $E_n^{jj'}$ through $G_n^{jj'}$ are isometries on $L^\infty_\xi$,
$$
\bigl|\langle A_n^{jj'} \phi, \psi \rangle_{L^2_x(\R^2)}\bigr|\lesssim \frac{\lambda_n^{j'}}{\lambda_n^j} \|\hat \phi\|_{L^\infty_\xi} \|\hat \psi\|_{L^1_\xi};
$$
indeed, the first factor on the right is the norm of $C_n^{jj'}$ in $L_\xi^\infty$.  This proves weak convergence to zero in the case when $\lambda_n^{j'}/\lambda_n^j\to 0$.
The case when $\lambda_n^{j'}/\lambda_n^j\to \infty$ can be handled by reversing the roles of $j$ and $j'$ and recalling that the adjoint of $A_n^{jj'}$ is $A_n^{j'j}$.
This leaves the case when the ratio converges to a finite positive number; in this scenario, $C_n^{jj'}$ and its adjoint both converge strongly and so may be neglected in what follows.

Looking back at the definition of $\ell_\nu^\lambda$, we see that if $\lambda_n^{j'}|\nu_n^{jj'}| \to \infty$, then
$$
E_n^{jj'}F_n^{jj'}G_n^{jj'} \hat{\phi} \qtq{and} \hat \psi
$$
have disjoint supports for large $n$, which proves weak convergence to zero.  If on the contrary $\lambda_n^{j'}\nu_n^{jj'}$ converges then, by observation (ii) above,
$E_n^{jj'}$ and its adjoint converge in the strong operator topology on $\mathcal{B}(L^2_x(\R^2))$, and we may disregard it in what follows.

It remains only to show that
\begin{equation}\label{E:RL}
\lim_{n\to\infty} \int_{\R^2} \hat\phi(\xi) \overline{\hat\psi(\xi)} \exp\bigl\{-is_n^{jj'}\jp{\xi/\lambda_n^{j'}} - i \xi y_n^{jj'}/\lambda_n^{j'}  \bigr\} \,d\xi = 0
\end{equation}
whenever $(\lambda_n^{j'})^{-2}|s_n^{jj'}|\to\infty$ or $|y_n^{jj'}|/\lambda_n^{j'} \to\infty$.  In the former case, this follows from the van der Corput lemma after noting that
$\Delta_\xi\, \jp{\xi/\lambda} \geq \lambda^{-2} \jp{\xi/\lambda}^{-1}$.  If only the latter sequence of parameters diverge, then \eqref{E:RL} reduces to the Riemann--Lebesgue lemma.

This completes the proof of Lemma~\ref{L:orthogonality}.
\end{proof}

We end this section with a few propositions that will be useful when we apply the linear profile decomposition in Section~\ref{S:min blowup} to extract a minimal blowup solution.

\begin{proposition}[Energy decoupling] \label{P:nrg decoup}
Suppose $\{v_n\}$ is a bounded sequence of $H^1_x$ functions.  Then after passing to a subsequence, the linear profile decomposition \eqref{E:lpd} satisfies the following: for each $J < J_0$,
\begin{equation} \label{E:nrg decoup}
\lim_{n \to \infty} \Bigl\{ E(v_n) - \sum_{j=1}^J E(T_{x_n^j} e^{it_n^j \jpn} \W_{\nu_n^j} D_{\lambda_n^j}P_n^j \phi^j) - E(w_n^J) \Bigr\} = 0.
\end{equation}
\end{proposition}

\begin{proof}
We will prove that the energy decouples in the inverse Strichartz theorem, that is, in the case $J=1$; the general case follows by induction.  Furthermore, by \eqref{E:lpd H1 decoup},
it suffices to show that
\begin{equation} \label{E:lpd L4x decoup}
\lim_{n \to 0} \Bigl\{\|\Re v_n\|_{L^4_x}^4 - \| \Re \phi_n\|_{L^4_x}^4 - \|\Re w_n\|_{L^4_x}^4\Bigr\} = 0,
\end{equation}
where
$$
\phi_n = T_{x_n} e^{it_n\jpn} \W_{\nu_n} D_{\lambda_n} P_n \phi,
$$
with $P_n = 1$ if $\lambda_n \equiv 1$ and $P_n = P_{\leq \lambda_n^{\theta}}$ if $\lambda_n \to \infty$.

We start by considering the case when $\lambda_n \equiv 1$; recall that in this case we in fact have $\phi \in H^1_x$,
$$
\phi_n = T_{x_n}e^{it_n\jpn}\phi
$$
(because $\nu_n \equiv 0$ and $P_n$ is the identity), and either $t_n\to \pm \infty$ or $t_n\equiv 0$.  Approximating $\phi$ in $H^1_x$ by Schwartz functions and applying the dispersive estimate,
we see that
$$
\|e^{it_n\jpn} \phi\|_{L^4_x} \to 0  \qtq{when} t_n \to \pm \infty.
$$
Claim \eqref{E:lpd L4x decoup} now follows easily.   Next we consider the case $\lambda_n\equiv 1$ and $t_n \equiv 0$.  By \eqref{E:lpd wk lim}, we have $T_{-x_n} w_n \rightharpoonup 0$, weakly
in $H^1_x$.   Thus, by Rellich's theorem, a subsequence of $T_{-x_n} w_n$ converges a.e.\ to $0$, and so \eqref{E:lpd L4x decoup} follows by applying  Lemma~\ref{L:BrezisLieb} with
$F_n=\Re T_{-x_n}v_n$ and $F=\Re\phi$.

It remains to consider the case when $\lambda_n \to \infty$, which we treat with the following lemma.

\begin{lemma} \label{L:L4x to zero} If $\lambda_n \to \infty$, then
\begin{equation} \label{E:L4x to zero}
\lim_{n \to \infty} \|\phi_n\|_{L^4_x} = 0.
\end{equation}
\end{lemma}

\begin{proof}
We will use Bernstein's inequality, which implies that
\begin{align}  \label{E:Bernstein app}
\|\phi_n\|_{L^4_x} &\lesssim \bigl[\diam(\supp \widehat{\phi}_n)\bigr]^{1/2} \|\phi_n\|_{L^2_x}.
\end{align}

We note that since $|\partial_{\xi_j} \ell_{\nu_n}(\xi)| \lesssim \jp{\nu_n}$ for $\xi \in \R^2$, by \eqref{E:Fourier Boost}  we have
\begin{align}
\notag \diam(\supp \widehat{\phi}_n) &= \diam(\supp ( [\W_{\nu_n} D_{\lambda_n} P_n \phi]\widehat{\ })) \\
\label{E:bound support} &\lesssim \jp{\nu_n} \diam(\supp( [D_{\lambda_n}P_n \phi]\widehat{\ })) \lesssim \jp{\nu_n} \lambda_n^{\theta-1}.
\end{align}
Furthermore, by Lemma~\ref{L:boost action}, we have
\begin{align} \label{E:bound L2}
\|\phi_n\|_{L^2_x} = \|\W_{\nu_n} D_{\lambda_n} P_n \phi\|_{L^2_x} \lesssim \jp{\nu_n} \|\phi\|_{L^2_x}.
\end{align}

Using \eqref{E:bound support} and \eqref{E:bound L2} to bound the right side of \eqref{E:Bernstein app}, and then using boundedness of the $\nu_n$, we obtain
$$
\|\phi_n\|_{L^4_x} \lesssim \jp{\nu_n} \lambda_n^{\frac{\theta-1}2}\|\phi\|_{L^2_x} \to 0,
$$
and the lemma is proved.
\end{proof}

By Lemma~\ref{L:L4x to zero}, we have $\|v_n - w_n\|_{L^4_x} \to 0$, and \eqref{E:lpd L4x decoup} follows.  This completes the proof of the proposition.
\end{proof}

\begin{proposition}[Decoupling of nonlinear profiles] \label{P:nonlinear decoup} Let $\psi^j$ and $\psi^{j'}$ be functions in $C^{\infty}_c(\R \times \R^2)$.  Given parameters
$\nu_n^j$, $\nu_n^{j'}$, $(t_n^j,x_n^j)$, $(t_n^{j'},x_n^{j'})$, $\lambda_n^j$, $\lambda_n^{j'}$ as above, we define $\psi_n^j$ by
\begin{align} \label{E:nl profile approx}
\bigl[\psi^j_n(\cdot + t_n^j,\cdot + x_n^j) \circ L_{\nu_n^j}^{-1}\bigr](t,x) := \tfrac{e^{-it}}{\lambda_n^j} \psi^j\Bigl(\tfrac{t}{(\lambda_n^j)^2},\tfrac{x}{\lambda_n^j}\Bigr),
\end{align}
and similarly for $\psi^{j'}_n$.  Then under the orthogonality condition \eqref{E:lpd orthogonality}, we have
\begin{equation} \label{E:nonlinear decoup}
\lim_{n \to \infty}
\| \psi^j_n \psi^{j'}_n\|_{L^2_{t,x}(\R \times \R^2)} = 0.
\end{equation}
\end{proposition}

\begin{proof}
Let
$$
L_{\nu}^{\lambda}(t,x) := \bigl(\jp{\nu}t-\tfrac{\nu}{\lambda} \cdot x, x^{\perp} + \jp{\nu}x^{\parallel} - \lambda \nu t\bigr)
$$
and let $\lambda_n^{jj'} = \lambda_n^j/\lambda_n^{j'}$.  Additionally, let $R_n^{jj'}$ be spatial rotations and $\nu_n^{jj'}$ be boost parameters such that
$$
L_{\nu_n^{j'}} \circ L_{\nu_n^j}^{-1}(t,x) = L_{\nu_n^{jj'}}(t,R_n^{jj'}x)
$$
(cf.\ the proof of Lemma~\ref{L:orthogonality}).  Recall that $|\nu_n^{jj'}| \sim |\nu_n^j - \nu_n^{j'}|$.

With this notation (using the fact that spatial dilations and rotations commute), we compute
\begin{align} \label{E:psijj' mess}
&\int_{\R \times \R^2}|\psi_n^j \psi_n^{j'}|^2\, dx\, dt\\  \notag
&\qquad = \int_{\R \times \R^2}\Bigl|\tfrac{1}{\lambda_n^{jj'}}\psi^j\Bigl(\tfrac{t}{(\lambda_n^{jj'})^2},\tfrac{(R_n^{jj'})^Tx}{\lambda_n^{jj'}}\Bigr)
        \psi^{j'}\Bigl(\cdot - \tfrac{s_n^{jj'}}{(\lambda_n^{j'})^2},\cdot - \tfrac{y_n^{jj'}}{\lambda_n^{j'}}\Bigr) \circ L_{\nu_n^{jj'}}^{\lambda_n^{j'}}(t,x)\Bigr|^2\, dx\, dt.
\end{align}
As before, by passing to a subsequence, we may assume that $R_n^{jj'}$ converges.  By absorbing the limit into $\psi^j$ and using continuity, it suffices to treat the case
when $R_n^{jj'}$ is the identity.

If $\lambda_n^{jj'} \to 0$, then by H\"older's inequality and \eqref{E:psijj' mess}, we have
$$
\int_{\R \times \R^2} |\psi_n^j \psi_n^{j'}|^2\, dx\, dt \leq (\lambda_n^{jj'})^2\|\psi^j\|_{L^2_{t,x}}^2\|\psi^{j'}\|_{L^{\infty}_{t,x}}^2 \to 0.
$$
Similarly, if $\lambda_n^{jj'} \to \infty$, we have
$$
\int_{\R \times \R^2} |\psi_n^j \psi_n^{j'}|^2\, dx\, dt \leq (\lambda_n^{jj'})^{-2}\|\psi^j\|_{L^{\infty}_{t,x}}^2\|\psi^{j'}\|_{L^2_{t,x}}^2 \to 0;
$$
here we have used the fact that $L_{\nu}^{\lambda}$ is volume-preserving.  We may thus assume that $\lambda_n^{jj'}$ converges to some positive number,
and arguing as above, it suffices to treat the case when $\lambda_n^{jj'} \equiv 1$.

Let
$$
S_n^{jj'} = \supp(\psi^j) \cap \supp\Bigl(\psi^{j'}\Bigl(\cdot - \tfrac{s_n^{jj'}}{(\lambda_n^{j'})^2},\cdot - \tfrac{y_n^{jj'}}{\lambda_n^{j'}}\Bigr) \circ L_{\nu_n^{jj'}}^{\lambda_n^{j'}}\Bigr).
$$
If $(t,x) \in S_n^{jj'}$, then $|t|+|x| \lesssim 1$ because $\supp \psi^j$ is compact, and
$$
\Bigl|x^{\perp} + \jp{\nu_n^{jj'}}x^{\parallel} - \lambda_n^{j'}\nu_n^{jj'}t - \tfrac{y_n^{jj'}}{\lambda_n^{j'}}\Bigr| \lesssim 1
$$
because $\supp \psi^{j'}$ is compact.  Thus if $|\lambda_n^{j'}\nu_n^{jj'}| \to \infty$, then by H\"older's inequality and \eqref{E:psijj' mess} we have that
$$
\int_{\R \times \R^2} |\psi_n^j \psi_n^{j'}|^2\, dx\, dt \leq |S_n^{jj'}|\|\psi^j\|_{L^{\infty}_{t,x}}^2 \|\psi^{j'}\|_{L^{\infty}_{t,x}}^2 \lesssim |\lambda_n^{j'}\nu_n^{jj'}|^{-1} \to 0.
$$
Finally, if $\lambda_n^{j'} \nu_n^{jj'}$ remains bounded while $\Bigl|\Bigl(\frac{s_n^{jj'}}{(\lambda_n^{jj'})^2},\frac{y_n^{j'}}{\lambda_n^{j'}}\Bigr)\Bigr| \to \infty$,
then $S_n^{jj'}$ is eventually empty. This completes the proof of the proposition.
\end{proof}


\section{Isolating NLS inside nonlinear Klein--Gordon} \label{S:embed}


In this section, we consider the mass-critical nonlinear Schr\"odinger equation in the form
\begin{equation} \label{NLS}
(i \partial_t + \tfrac{1}{2}\Delta)w =\mu\tfrac{3}{8}|w|^2w,
\end{equation}
with $\mu=\pm1$ as in \eqref{nlkg} and \eqref{nlkg1st}.  This normalization of the nonlinear Schr\"odinger equation appears naturally
in connection to NLKG and can be reduced to \eqref{normal nls} by rescaling $w$ and $x$.  Correspondingly, the ground state solution associated to \eqref{NLS}
is $w_Q(t,x):=e^{-it}\sqrt{\tfrac83}Q(\sqrt2 \,x)$ where $Q$ is as in \eqref{E:Q defn}.  Note that $M(w_Q) =\tfrac43M(Q)$.  Under this rescaling, Conjecture~\ref{Conj:NLS} takes the
following form:

\begin{conjecture}[Global well-posedness of NLS]  \label{C:NLS}  Fix a value of $\mu = \pm 1$.  Let $w_0 \in L^2_x(\R^2)$ and in the focusing case assume that $M(w_0)<\tfrac43M(Q)$.
Then there exists a unique global solution $w$ to \eqref{NLS} with $w(0) = w_0$.  Furthermore, this solution satisfies
$$
 \|w\|_{L^4_{t,x}(\R \times \R^2)} \leq C(M(w_0)),
$$
for some continuous function $C$.  As a consequence, the solution $w$ scatters both forward and backward in time, that is, there exist $w_\pm\in L_x^2$ such that
$$
\|w(t) -e^{it\Delta/2} w_\pm\|_{L_x^2} \to 0 \qtq{as} t\to \pm \infty.
$$
Conversely, for each $w_\pm$ there is a global solution $w$ to \eqref{NLS} so that the above holds.
\end{conjecture}

The goal of this section is to prove the following theorem:

\begin{theorem} \label{T:embed}
Assume that Conjecture~\ref{C:NLS} holds.  Let sequences $\nu_n \to \nu \in \R^2$, $\lambda_n \to \infty$, and $\{t_n\} \in \R$ be given.  Assume that either $t_n \equiv 0$ or
$t_n/\lambda_n^2 \to \pm \infty$.  Let $\phi \in L^2_x(\R^2)$, and in the focusing case assume also that $M(\phi) < \tfrac43 M(Q)$.  If we define
$$
\phi_n := T_{x_n}e^{it_n\jpn}\W_{\nu_n} D_{\lambda_n}P_{\leq \lambda_n^{\theta}} \phi
$$
for $\theta=\tfrac1{100}$, then for each $n$ sufficiently large, there exists a global solution $v_n$ to \eqref{nlkg1st} with initial data $v_n(0) = \phi_n$, which satisfies
$$
S_{\R}(v_n) \lesssim_{M(\phi)} 1.
$$
Furthermore, for every $\eps > 0$, there exist $N_{\eps}$ and a function $\psi_{\eps} \in C^{\infty}_c(\R \times \R^2)$ such that for all $n > N_{\eps}$,
\begin{equation}\label{E:Cinfty approx0}
\bigl\| \Re\bigl\{v_n \circ L_{\nu_n}^{-1}(t+ \tilde t_n, x + \tilde x_n)
        - \tfrac{e^{-it}}{\lambda_n}\psi_{\eps}(\tfrac{t}{\lambda_n^2},\tfrac{x}{\lambda_n})\bigr\}\bigr\|_{L^4_{t,x}(\R \times \R^2)} < \eps
\end{equation}
where $(\tilde t_n,\tilde x_n):=L_{\nu_n}(t_n,x_n)$ is the center of the wave packet in boosted coordinates.
\end{theorem}

\begin{remark}
As we will see in due course, the proof of Theorem~\ref{T:ST bounds} in the focusing case relies only on the consideration of those $\phi$ with $M(\phi)<M(Q)$.
Thus, the full Conjecture~\ref{Conj:NLS} (or equivalently, Conjecture~\ref{C:NLS}) covers more cases than are needed here.  In a similar vein,
a proof of Conjecture~\ref{Conj:NLS} in the focusing case up to some intermediate mass threshold $M_*<M(Q)$ yields a corresponding result for NLKG.
\end{remark}

We begin with the proof in the case $\nu_n\equiv 0$, which adapts the ideas in \cite[\S4]{KKSV}.  This result is then used to treat the general case.

\begin{proof}[Proof in the case $\nu_n\equiv 0$]  As the first order nonlinear Klein--Gordon equation is invariant under spatial translations, we may assume $x_n\equiv 0$.
Thus, \eqref{E:Cinfty approx0} will follow from
\begin{equation} \label{E:Cinfty approx}
\bigl\| v_n(t+t_n,x) -  \tfrac{e^{-it}}{\lambda_n}\psi_{\eps}(\tfrac{t}{\lambda_n^2},\tfrac{x}{\lambda_n})\bigr\|_{L^4_{t,x}(\R \times \R^2)} < \eps.
\end{equation}

We begin by defining solutions to \eqref{NLS}; we will later modify these to produce approximate solutions to \eqref{nlkg1st}.

In the case when $t_n\equiv 0$, we let $w_n$ be the solution to \eqref{NLS} with initial data
$$
w_n(0) = P_{\leq \lambda_n^{\theta}} \phi.
$$
Similarly, we let $w_\infty$ be the solution to \eqref{NLS} with initial data
$$
w_\infty(0)=\phi.
$$

In the case when $t_n/\lambda_n^2\to -\infty$ (respectively $t_n/\lambda_n^2 \to +\infty$), we denote by $w_n$ the solutions to \eqref{NLS} that scatter forward (respectively backward)
in time to $e^{it \Delta/2} P_{\leq \lambda_n^{\theta}}\phi$.  Correspondingly, we define $w_\infty$ to be the solution to \eqref{NLS} that scatters forward (respectively backward)
in time to $e^{it \Delta/2} \phi$.  (The signs here are correct because if the bulk of the solution is living around time $t_n \to -\infty$, then time 0 is well inside the forward scattering regime.)

As we assume Conjecture~\ref{C:NLS} holds (and that $M(\phi) < \frac43 M( Q)$ in the focusing case), all the solutions to \eqref{NLS} defined above are global and moreover,
$$
S_{\R}(w_n) +S_\R(w_\infty) \lesssim_{M(\phi)} 1.
$$

We begin with a few basic observations about the sequence $w_n$, which will be helpful in what follows.

\begin{lemma}\label{L:basic NLS facts}
For $s\geq 0$, the solutions $w_n$ defined above satisfy
\begin{equation} \label{E:persistence of reg}
\||\nabla|^s w_n\|_{L^{\infty}_tL^2_x(\R \times \R^2)} + \||\nabla|^s w_n\|_{L^4_{t,x}(\R \times \R^2)}\lesssim_{M(\phi)} \lambda_n^{s\theta}
\end{equation}
and
\begin{equation}
\label{E:time derivative}
\|\jpn^s \partial_t w_n\|_{L^4_{t,x}} \lesssim_{M(\phi)}  \lambda_n^{(2+s)\theta}.
\end{equation}
Furthermore, we have the approximation
\begin{equation} \label{E:limit w_infty}
\lim_{n \to \infty}\Bigl\{ \|w_n - w_{\infty}\|_{L^{\infty}_t L^2_x} +\|w_n-w_{\infty}\|_{L^4_{t,x}} + \|D_{\lambda_n}(w_n - P_{\leq \lambda_n^{\theta}}w_{\infty})\|_{L^{\infty}_t H^{\frac12}_x}\Bigr\} = 0.
\end{equation}
\end{lemma}

\begin{proof}
The first inequality follows from persistence of regularity arguments (cf. \cite[Lemma 3.10]{Matador}) and the fact that by Bernstein's inequality,
$$
\||\nabla|^s P_{\leq \lambda_n^{\theta}}\phi \|_{L^2_x(\R^2)} \lesssim \lambda_n^{s\theta}\|\phi\|_{L^2_x(\R^2)}.
$$

To prove inequality \eqref{E:time derivative}, we use the equation \eqref{NLS} together with \eqref{E:persistence of reg}, H\"older's inequality, and Sobolev embedding:
\begin{align*}
\|\jpn^s\partial_t &w_n\|_{L^4_{t,x}} \\
&\lesssim \|\jpn^s\Delta w_n\|_{L^4_{t,x}(\R \times \R^2)} + \|\jpn^sw_n \|_{L_t^4 L^{12}_x(\R\times \R^2)}\|w_n \|_{L_t^\infty L^{12}_x(\R\times \R^2)}^2\\
&\lesssim_{M(\phi)} \lambda_n^{(2+s)\theta} + \|\jpn^{1/3+s} w_n\|_{L^4_{t,x}(\R \times \R^2)}\||\nabla|^{5/6} w_n\|_{L_t^\infty L_x^2(\R \times \R^2)}^2\\
&\lesssim_{M(\phi)} \lambda_n^{(2+s)\theta}.
\end{align*}

That the first two terms in \eqref{E:limit w_infty} tend to zero is a consequence of the stability theory for the mass-critical NLS; this result may be found in \cite{ClayNotes} or \cite{Matador}.

We turn now to the final term on the left side of \eqref{E:limit w_infty}; changing variables and using the triangle inequality and \eqref{E:persistence of reg}, we obtain
\begin{align*}
\| D_{\lambda_n}(w_n & - P_{\leq \lambda_n^{\theta}}w_{\infty})\|_{L^{\infty}_t H^{1/2}_x}\\
&=\|\jp{\lambda_n^{-1} \nabla}^{1/2}(w_n - P_{\leq \lambda_n^{\theta}}w_{\infty})\|_{L^{\infty}_t L^2_x}\\
&\leq \|\jp{\lambda_n^{-1} \nabla}^{1/2}P_{\geq \lambda_n} w_n\|_{L^{\infty}_t L^2_x} + \|\jp{\lambda_n^{-1} \nabla}^{1/2}P_{\leq \lambda_n} (w_n - w_{\infty})\|_{L^{\infty}_t L^2_x} \\
&\qquad +  \|\jp{\lambda_n^{-1} \nabla}^{1/2}P_{\lambda_n^{\theta}\leq \cdot\leq \lambda_n} w_{\infty}\|_{L^{\infty}_t L^2_x}\\
&\lesssim \lambda_n^{-1/2} \| |\nabla|^{1/2} w_n\|_{L_t^\infty L_x^2} + \|w_n-w_\infty\|_{L_t^\infty L_x^2} +  \|P_{\geq \lambda_n^{\theta}} w_{\infty}\|_{L^{\infty}_t L^2_x}\\
&\lesssim_{M(\phi)}  \lambda_n^{-1/2+\theta/2} +  \|w_n-w_\infty\|_{L_t^\infty L_x^2}  +\|P_{\geq\lambda_n^{\theta}} w_{\infty}\|_{L^{\infty}_t L^2_x}.
\end{align*}
It is immediate that the first term on the right-hand side above converges to zero as $n\to \infty$, while the convergence to zero of the second term follows from \eqref{E:persistence of reg}.
It remains to consider the third term.

By our assumption that Conjecture~\ref{C:NLS} holds, $w_{\infty}$ scatters both forward and backward in time; let $w_\pm \in L_x^2$ be the scattering states.  Then
\begin{align*}
\|P_{\geq \lambda_n^{\theta}} & w_{\infty}\|_{L^{\infty}_tL^2_x([T,\infty)\times \R^2)} + \|P_{\geq \lambda_n^{\theta}} w_{\infty}\|_{L^{\infty}_tL^2_x((-\infty, -T]\times \R^2)}\\
&\lesssim \|w_{\infty} - e^{it\Delta/2}w_+\|_{L^{\infty}_t L^2_x([T,\infty)\times \R^2)} + \|w_{\infty} - e^{it\Delta/2}w_-\|_{L^{\infty}_t L^2_x((-\infty, -T]\times \R^2)} \\
&\quad + \|P_{\geq \lambda_n^{\theta}}w_+\|_{L^2_x(\R^2)} +\|P_{\geq \lambda_n^{\theta}}w_-\|_{L^2_x(\R^2)},
\end{align*}
which can be made arbitrarily small by choosing $T$ and $n$ sufficiently large.   On the other hand, for each fixed $T>0$, the continuity of the mapping $t \mapsto w_{\infty}(t)$
together with the compactness of $[-T,T]$ and the fact that the sequence of operators $P_{\geq \lambda_n^{\theta}}$ is uniformly bounded and converges to zero in the strong operator
topology on $L^2_x(\R^2)$ yield
$$
\lim_{n\to \infty} \|P_{\geq \lambda_n^{\theta}} w_{\infty}\|_{L^{\infty}_t L^2_x([-T,T]\times \R^2)}=0.
$$
This completes the proof of the lemma.
\end{proof}

We now use the solutions $w_n$ to NLS to construct approximate solutions to \eqref{nlkg1st}.  Let $T$ be a large, positive real number, to be determined later.  We define
\begin{align} \label{E:approx solutions}
\tilde{v}_n (t) :=
\begin{cases}
e^{-it}D_{\lambda_n}w_n(t/\lambda_n^2), &\text{if $|t| \leq T\lambda_n^2$,} \\
e^{-i(t-T\lambda_n^2)\jpn}\tilde{v}_n(T\lambda_n^2), &\text{if $t > T\lambda_n^2$,} \\
e^{-i(t+T\lambda_n^2)\jpn}\tilde{v}_n(-T\lambda_n^2), &\text{if $t<-T\lambda_n^2$,}
\end{cases}
\end{align}
with the idea that $v_n(t)-\tilde v_n(t-t_n)$ will be small.  Ultimately however, our approximate solution will be a further modification of this.  In particular, we will make an adjustment
on the middle interval which is small in $L^{\infty}_t H^{1/2}_x(\R \times \R^2)$.  This will result in an analogous change on the outer intervals, but as we will show
(using the Strichartz inequality), the modification is negligible in this regime.

Forgetting the above-mentioned technical issues for now, we give an explanation as to why $\tilde{v}_n$ might be an approximate solution to \eqref{nlkg1st}.  On the middle interval,
we can use the estimate \eqref{phase approx} to show that the above transformation takes solutions to the linear Schr\"odinger equation to approximate solutions of the first-order
linear Klein--Gordon equation.  The behavior of the nonlinearities on this interval is a bit more mysterious, but the specific factor $\frac{3}{8}$ appearing in \eqref{NLS} will ensure
that certain resonant error terms cancel, while the non-resonant error terms will be subdued via the use of $X^{s,b}$-type estimates.  As $t$ tends to infinity,
the differences in the two dispersion relations become amplified and the approximation breaks down.  Fortunately, by this time the NLS solution is well dispersed and so
resembles a free evolution.  This behaviour is inherited by the Klein--Gordon evolution as we intimated in \eqref{E:approx solutions}; see also Lemma~\ref{L:skg dispersion}.

We now set about supplying the details behind these heuristics.  As each heuristic introduces some small error, we will need to use the stability theory
(Proposition~\ref{P:stability}) to construct the final solution $v_n$. This eventuality dictates which estimates we need to prove, beginning with the following:
By the Strichartz inequality and Lemma~\ref{L:basic NLS facts},
\begin{align}\label{E:vn bounds}
\|\tilde{v}_n\|_{L^{\infty}_tH^{1/2}_x(\R \times \R^2)} + \|\tilde{v}_n\|_{L^4_{t,x}(\R \times \R^2)}
&\lesssim \|D_{\lambda_n} w_n\|_{L_t^\infty H_x^{1/2}} +\|D_{\lambda_n} w_n(t/\lambda_n^2)\|_{L_{t,x}^4}\notag\\
&\lesssim_{M(\phi)} 1 + \lambda_n^{-1/2} \| |\nabla|^{1/2} w_n\|_{L_t^\infty L_x^2}\notag\\
&\lesssim_{M(\phi)} 1+ \lambda_n^{-1/2 +\theta/2}\lesssim_{M(\phi)} 1.
\end{align}

\begin{lemma}[Matching initial data] \label{L:initial data}If $\tilde{v}_n$ is given by \eqref{E:approx solutions}, then
\begin{align} \label{E:approx initial data}
\lim_{T \to \infty} \limsup_{n \to \infty} \|\tilde{v}_n(-t_n) - \phi_n\|_{H^{1/2}_x(\R^2)} = 0.
\end{align}
\end{lemma}

\begin{proof}
In the case when $t_n\equiv 0$ (recall that $\phi_n=e^{it_n\jpn}D_{\lambda_n}P_{\leq \lambda_n^\theta} \phi$, as we assume $x_n \equiv \nu_n\equiv0$) we have $\tilde v_n(-t_n)=\phi_n$, so there is nothing more to prove.

Consider now the case when $t_n/\lambda_n^2\to -\infty$; the case $t_n/\lambda_n^2\to \infty$ can be treated similarly.  In this case, for any finite choice of $T$, we eventually have that $-t_n > \lambda_n^2 T$.
Thus, we can rewrite \eqref{E:approx initial data} as
\begin{align}\label{infty match}
\lim_{T \to \infty} \limsup_{n \to \infty} \|e^{-iT\lambda_n^2}D_{\lambda_n}w_n(T) - e^{-i\lambda_n^2T\jpn}D_{\lambda_n} P_{\leq \lambda_n^\theta}\phi\|_{H^{1/2}_x(\R^2)} = 0
\end{align}

By the triangle inequality,
\begin{align*}
 \|&e^{-iT\lambda_n^2}D_{\lambda_n}w_n(T) - e^{-i\lambda_n^2T\jpn}D_{\lambda_n} P_{\leq \lambda_n^\theta}\phi\|_{H^{1/2}_x(\R^2)}\\
&\lesssim  \bigl\|D_{\lambda_n}\bigl[w_n(T) - P_{\leq \lambda_n^\theta}w_\infty(T)\bigr]\bigr\|_{H^{1/2}_x(\R^2)}\\
&\quad +\bigl \|D_{\lambda_n}P_{\leq \lambda_n^\theta}\bigl[w_\infty(T) - e^{i\lambda_n^2T(1-\jp{\lambda_n^{-1}\nabla})} \phi\bigr]\bigr\|_{H^{1/2}_x(\R^2)}\\
&\lesssim  \bigl\|D_{\lambda_n}\bigl[w_n(T) - P_{\leq \lambda_n^\theta}w_\infty(T)\bigr]\bigr\|_{H^{1/2}_x(\R^2)} + \|w_\infty(T) - e^{i\lambda_n^2T(1-\jp{\lambda_n^{-1}\nabla})} \phi\|_{L^2_x(\R^2)}\\
&\lesssim  \bigl\|D_{\lambda_n}\bigl[w_n(T) - P_{\leq \lambda_n^\theta}w_\infty(T)\bigr]\bigr\|_{H^{1/2}_x(\R^2)} + \|w_\infty(T) - e^{iT\Delta/2} \phi\|_{L^2_x(\R^2)}\\
&\quad + \bigl\|\,\bigl[1-e^{iT\lambda_n^2(1-\jp{\lambda_n^{-1}\xi}+\frac12 \lambda_n^{-2}|\xi|^2)} \bigr]\hat\phi\bigr\|_{L^2_{\xi}(\R^2)}.
\end{align*}
Now \eqref{infty match} follows from \eqref{E:limit w_infty}, the definition of $w_\infty$ (in the case $t_n/\lambda_n^2\to -\infty$), and the dominated convergence theorem, respectively.
\end{proof}

Next, we show that $\tilde{v}_n$ are approximate solutions to \eqref{nlkg1st}, starting with the large time intervals.   More precisely, we will prove

\begin{proposition} [Large time intervals]  \label{P:large times approx} With the notation above,
$$
\lim_{T \to \infty} \limsup_{n \to \infty} \|e^{-i(t-\lambda_n^2T)\jpn} \tilde{v}_n(T\lambda_n^2)\|_{L^4_{t,x}((\lambda_n^2T,\infty) \times \R^2)} = 0
$$
and analogously on the time interval $(-\infty,-\lambda_n^2 T)$.
\end{proposition}

This proposition is an immediate consequence of the two lemmas that follow; indeed, one merely needs to combine them with the triangle inequality.
The first lemma shows that for a large enough time, we can safely approximate the nonlinear solutions $w_n$ by solutions to the free Schr\"odinger equation.

\begin{lemma} \label{L:kinda scattering}
Let $w_+$ be the forward-in-time scattering state of the NLS solution $w_{\infty}$ defined above.  Then
$$
\lim_{T \to \infty} \limsup_{n \to \infty} \bigl\|e^{-i(t-\lambda_n^2T)\jpn} \bigl[\tilde{v}_n(T\lambda_n^2)
- e^{-iT\lambda_n^2}D_{\lambda_n} e^{iT\Delta/2} P_{\leq \lambda_n^{\theta}} w_+\bigr]\bigr\|_{L^4_{t,x}((\lambda_n^2T,\infty) \times \R^2)} = 0.
$$
A similar statement holds on the time interval $(-\infty,-\lambda_n^2T)$, but with $w_-$ in place of $w_+$.
\end{lemma}

\begin{proof}  We will give the argument for the forward-in-time statement.  By the Strichartz inequality for the first-order Klein--Gordon propagator, it suffices to prove
$$
\lim_{T \to \infty} \limsup_{n \to \infty} \|\tilde{v}_n(T\lambda_n^2) - e^{-iT\lambda_n^2}D_{\lambda_n} e^{iT\Delta/2} P_{\leq \lambda_n^{\theta}} w_+\|_{H^{1/2}_x(\R^2)} = 0.
$$
By Lemma~\ref{L:basic NLS facts}, it suffices to show that
$$
\lim_{T \to \infty} \limsup_{n \to \infty}\bigl\|D_{\lambda_n}P_{\leq \lambda_n^{\theta}}\bigl[w_{\infty}(T) - e^{iT\Delta/2}w_+\bigr]\bigr\|_{H^{1/2}_x(\R^2)}= 0.
$$
To see this, we note that
$$
\bigl\|D_{\lambda_n}P_{\leq \lambda_n^{\theta}}\bigl[w_{\infty}(T) - e^{iT\Delta/2}w_+\bigr]\bigr\|_{H^{1/2}_x(\R^2)} \lesssim \| w_{\infty}(T) - e^{iT\Delta/2}w_+\|_{L^2_x(\R^2)},
$$
which converges to zero as $T\to \infty$.
\end{proof}

From the previous lemma, we know that $\tilde v_n(T\lambda_n^2)$ is well approximated by the free Schr\"odinger evolution of a specific function $w_+$.  The second step in the proof
of Proposition~\ref{P:large times approx} is to show that for $T$ large, the free first-order Klein--Gordon evolution of this function (into the future) is small.  Colloquially,
a solution that is well-dispersed for Schr\"odinger is also well-dispersed for Klein--Gordon.

\begin{lemma} \label{L:skg dispersion}
Let $\psi \in L^2_x(\R^2)$ and let $\lambda_n \to \infty$ be a sequence of positive numbers.  Then
\begin{equation} \label{E:skg dispersion}
\lim_{T \to \infty} \limsup_{n \to \infty} \|e^{-i(t-\lambda_n^2T)\jpn}e^{-iT\lambda_n^2}D_{\lambda_n}e^{iT\Delta/2}P_{\leq \lambda_n^{\theta}} \psi\|_{L^4_{t,x}((\lambda_n^2T,\infty) \times \R^2)} = 0.
\end{equation}
\end{lemma}

\begin{proof}
By the Strichartz inequality \eqref{E:half Strichartz} together with the easy inequality
$$
\|\jpn^{1/2} D_{\lambda_n} P_{\leq \lambda_n^{\theta}} (\psi-\tilde{\psi})\|_{L^2_x(\R^2)} \lesssim \|\psi-\tilde{\psi}\|_{L^2_x(\R^2)},
$$
we may assume that $\psi$ is a Schwartz function with compact frequency support.  Consequently, for $n$ sufficiently large,  $\supp \hat \psi \subset \{ |\xi| < \lambda_n^\theta\}$
and we can therefore ignore the projection operator $P_{\leq \lambda_n^\theta}$ in what follows.

Next, we set
$$
f_n(t) := e^{-it \jp{\lambda_n^{-1}\nabla}}e^{iT\lambda_n^2[\jp{\lambda_n^{-1}\nabla} - 1 + \Delta/(2\lambda_n^2)]}\psi
$$
and observe that with this notation, \eqref{E:skg dispersion} becomes
$$
\lim_{T \to \infty} \limsup_{n \to \infty} \lambda_n^{-1/2}\|f_n\|_{L_{t,x}^4((\lambda_n^2T,\infty) \times \R^2)}=0.
$$
The proof will be via a stationary phase argument.
We write
\begin{align*}
f_n(t,x) &= \int_{\R^2} e^{-i (t/\lambda_n^2)h_{n,x}(\xi)}\hat{\psi}(\xi)\, d\xi,
\end{align*}
where
$$
h_{n,x}(\xi):= - \tfrac{\lambda_n^2}t \Bigl\{ x\cdot\xi - t \jp{\lambda_n^{-1}\xi} +T\lambda_n^2\bigl[\jp{\lambda_n^{-1}\xi} - 1 - \lambda_n^{-2}|\xi|^2/2\bigr]\Bigl\}.
$$
A computation yields
\begin{gather*}
\partial_i\partial_j h_{n,x}(\xi) = \Bigl[\tfrac{\delta_{ij}}{\jp{\lambda_n^{-1}\xi}} - \tfrac{\xi_i\xi_j}{\lambda_n^2\jp{\lambda_n^{-1}\xi}^3}\Bigr] \Bigl( 1 - \tfrac{\lambda_n^2T}{t}\Bigr)
        + \delta_{ij} \tfrac{\lambda_n^2T}{t},
\end{gather*}
where $\delta_{ij}$ denotes the Kronecker delta.  Thus,
\begin{equation} \label{E:hessian h}
\partial_i\partial_j h_{n,x}(\xi) = \delta_{ij}  + O(\lambda_n^{-2(1-2\theta)}),
\end{equation}
uniformly for $|\xi| \lesssim \lambda_n^{2\theta}$ and $t \geq \lambda_n^2T$.  If $\nabla h_{n,x}$ does not vanish in $\{|\xi| \lesssim \lambda_n^{2\theta}\}$, then by \eqref{E:hessian h},
$|\nabla h_{n,x}| \gtrsim 1$ uniformly on $\supp \hat{\psi} \subset \{|\xi| \lesssim \lambda_n^{\theta}\}$.  If $\nabla h_{n,x}$ vanishes in $\{|\xi| \lesssim \lambda_n^{2\theta}\}$,
then by the Morse lemma (cf.\ \cite[p.~346]{stein:large}) and \eqref{E:hessian h}, there exists $\eta_{n,x}$ which is a diffeomorphism (uniformly in $n,x,\xi$) on $\supp \hat{\psi}$
and such that $h_{n,x}(\xi) = |\eta_{n,x}(\xi)|^2 + c_{n,x}$.  In either case, by the principle of stationary phase, we have for $t \geq \lambda_n^2 T$ and $x \in \R^2$ that
$$
|f_n(t,x)| \lesssim_{\psi} \tfrac{\lambda_n^2}{t}.
$$
By interpolation with the trivial $L^2_x$ bound, we obtain
$$
\|f_n(t)\|_{L^4_x(\R^2)} \lesssim_{\psi} \bigl(\tfrac{\lambda_n^2}{t}\bigr)^{1/2}.
$$
Integrating this with respect to time we get
$$
\|f_n\|_{L^4_{t,x}((\lambda_n^2T,\infty)\times \R^2)} \lesssim_{\psi} \lambda_n^{1/2}T^{-1/4},
$$
which completes the proof the lemma.
\end{proof}

We now turn to showing that $\tilde v_n$ is an approximate solution to \eqref{nlkg1st} on the middle time interval.  A computation shows that on this middle time interval, $\tilde v_n$ satisfies
the following approximate first-order Klein--Gordon equation:
\begin{align} \label{E:middle interval}
(-i \partial_t + \jpn)\tilde{v}_n + \mu {\jpn}^{-1} (\Re \tilde{v}_n)^3 = e_1 + e_2 + e_3 + e_4,
\end{align}
where
\begin{align*}
e_1 &:= e^{-it}D_{\lambda_n}\Bigl\{\bigl[\jp{\lambda_n^{-1}\nabla} -1 + \tfrac{\Delta}{2\lambda_n^2} \bigr]w_n\bigl(\tfrac t{\lambda_n^2}\bigl)\Bigr\}, \\
e_2 &:=  \mu \bigl[\jpn^{-1} -1\bigr] (\Re \tilde{v}_n)^3\\
e_3 &:= \tfrac{1}{4}\Re e^{-3it}\Bigl[ D_{\lambda_n}w_n\bigl(\tfrac t{\lambda_n^2}\bigl)\Bigr]^3, \\
e_4 &:= \tfrac{3}{8}  e^{it} \bigl|D_{\lambda_n}w_n\bigl(\tfrac t{\lambda_n^2}\bigl)\bigr|^2 \overline{D_{\lambda_n}w_n\bigl(\tfrac t{\lambda_n^2}\bigl)}.
\end{align*}

The error terms $e_1$ and $e_2$ can be estimated in spaces for which Proposition~\ref{P:stability} applies.  We will estimate $e_1$ in $L^1_t H^{1/2}_x$.  As
$$
\Bigl|\jp{\lambda_n^{-1}\xi} - 1 - \tfrac{|\xi|^2}{2\lambda_n^2}\Bigr| \leq \tfrac12 \tfrac{|\xi|^4}{\lambda_n^4},
$$
by H\"older's inequality and \eqref{E:persistence of reg}, we obtain
\begin{align}\label{E:e_1}
\|e_1 & \|_{L^1_t H^{1/2}_x([-\lambda_n^2T,\lambda_n^2T] \times \R^2)} \notag \\
&\lesssim T\lambda_n^{-2} \Bigl( \|\Delta^2 w_n\|_{L^{\infty}_t L^2_x(\R \times \R^2)} + \lambda_n^{-1/2} \|\jpn^{9/2} w_n\|_{L^{\infty}_t L^2_x(\R \times \R^2)} \Bigr)\notag\\
&\lesssim T\lambda_n^{-2 +4 \theta} \to 0 \qtq{as} n\to \infty.
\end{align}

Next we estimate $e_2$ in $L^{4/3}_t W^{1, 4/3}_x$.   Noting that
$$
\jpn [\jpn^{-1}-1]= \nabla \frac{\nabla}{1+\jpn}
$$
and that by the Mikhlin multiplier theorem the second factor on the right is bounded on $L_{t,x}^{4/3}$, we have
\begin{align}\label{E:e_2}
\|\jpn e_2&\|_{L_{t,x}^{4/3} ([-\lambda_n^2T,\lambda_n^2T] \times \R^2)} \notag\\
&\lesssim \| \nabla(\Re \tilde{v}_n)^3\|_{L_{t,x}^{4/3}([-\lambda_n^2T,\lambda_n^2T] \times \R^2)} \notag\\
&\lesssim \bigl\|D_{\lambda_n}\tfrac{\nabla}{\lambda_n} w_n\bigl(\tfrac t{\lambda_n^2}\bigl)\bigr\|_{L_{t,x}^4[-\lambda_n^2T,\lambda_n^2T] \times \R^2)}
        \bigl\|D_{\lambda_n} w_n\bigl(\tfrac t{\lambda_n^2}\bigr)\bigr\|_{L_{t,x}^4(\R \times \R^2)}^2 \notag\\
&\lesssim \lambda_n^{-1+\theta}   \to 0 \qtq{as} n\to \infty.
\end{align}
In the last inequality we used \eqref{E:persistence of reg}.

Unfortunately, the error terms $e_3$ and $e_4$ are not small in either of the spaces $L_t^1H^{1/2}_x$ or $L_t^{4/3} W_x^{1, 4/3}$ for which Proposition~\ref{P:stability} applies.
However, they oscillate in spacetime like $e^{\pm 3it}$ and $e^{it}$, respectively, and the frequencies $\pm(3,0)$ and $(1,0)$ are far from the surface $\{(-\jp{\xi},\xi)\}$.
This allows us to use $X^{s,b}$-type arguments in the manner of \cite{CCT}.

\begin{lemma} \label{L:Xsb-type}
For $j=3,4$ let $f_{n,j}$ solve the equation
$$
(-i \partial_t + \jpn)f_{n,j} = e_j,  \qquad f_{n,j}(0) = 0.
$$
Then
$$
\|f_{n,j}\|_{L^\infty_t H^{1/2}_x([-\lambda_n^2 T,\lambda_n^2 T]\times \R^2)} +  \|f_{n,j}\|_{L^4_{t,x}([-\lambda_n^2T,\lambda_n^2 T]\times \R^2)} \lesssim \lambda_n^{-2+3\theta}.
$$
\end{lemma}

\begin{proof}
We will prove the lemma for $j=4$.  The argument for $j=3$ is almost identical.  We compute
\begin{align}\label{E:fnj}
(-i \partial_t + \jpn)(f_{n,4} - \tfrac12 e_4)
&= \tfrac{3}{16}\Bigl\{\frac{ie^{it}}{\lambda_n^5} \bigl[\partial_t (w_n \overline{w_n}^2)\bigr](\lambda_n^{-2}t,\lambda_n^{-1}x) \\
&\quad \qquad - \frac{e^{it}}{ \lambda_n^3}\bigl[(\jp{\lambda_n^{-1}\nabla}-1)(w_n \overline{w_n}^2)\bigr](\lambda_n^{-2}t,\lambda_n^{-1}x)\Bigr\}.\notag
\end{align}
Thus by Lemma~\ref{L:Strichartz}  and the triangle inequality, it suffices to bound $e_4$ in $L^\infty_t H^{1/2}_x$ and $L_{t,x}^4$ and the right-hand side of the identity above
in dual Strichartz spaces.  Estimating much as in \eqref{E:e_2}, we obtain
\begin{align*}
\|e_4\|_{L^\infty_t H^{1/2}_x([-\lambda_n^2 T,\lambda_n^2 T]\times \R^2)}
&\lesssim \|\jpn^{1/2} D_{\lambda_n}w_n\|_{L_t^\infty L_x^6(\R\times\R^2)} \|D_{\lambda_n}w_n\|_{L_t^\infty L_x^6(\R\times\R^2)}^{2}\\
&\lesssim \lambda_n^{-2 +2\theta}.
\end{align*}
Similarly, using Sobolev embedding and \eqref{E:persistence of reg}, we have
\begin{align*}
\| e_4\|_{L^4_{t,x}([-\lambda_n^2T,\lambda_n^2T]\times \R^2)}
&\lesssim \lambda_n^{-2} \|w_n^3\|_{L_{t,x}^4(\R\times\R^2)}\\
&\lesssim \lambda_n^{-2} \|w_n\|_{L_t^4L_x^{12}(\R\times\R^2)} \|w_n\|_{L_t^\infty L_x^{12}(\R\times\R^2)}^2\\
&\lesssim \lambda_n^{-2} \||\nabla|^{1/3} w_n\|_{L_{t,x}^4(\R\times\R^2)} \||\nabla|^{5/6} w_n\|_{L_t^\infty L_x^2(\R\times\R^2)}^2\\
&\lesssim \lambda_n^{-2+2\theta}.
\end{align*}

We turn now to estimating the right-hand side of \eqref{E:fnj}.  Noting that
$$
\jpn [\jpn-1]= -\Delta [1+\jpn^{-1}]^{-1}
$$
and that by the Mikhlin multiplier theorem the second factor on the right is bounded on $L_{t,x}^{4/3}$, we estimate
\begin{align*}
\Bigl\| \jpn\Bigl\{& \frac{e^{it}}{ \lambda_n^3}\bigl[(\jp{\lambda_n^{-1}\nabla}-1)(w_n \overline{w_n}^2)\bigr](\lambda_n^{-2}t,\lambda_n^{-1}x)\Bigr\}\Bigr\|_{L_{t,x}^{4/3}([-\lambda_n^2 T,\lambda_n^2 T]\times \R^2)} \\
&\lesssim \lambda_n^{-2} \Bigl\{ \|\Delta w_n\|_{L_{t,x}^4(\R\times\R^2)} \|w_n\|_{L_{t,x}^4(\R\times\R^2)}^2 + \|\nabla w_n\|_{L_{t,x}^4(\R\times\R^2)}^2 \| w_n\|_{L_{t,x}^4(\R\times\R^2)}  \Bigr\}\\
&\lesssim \lambda_n^{-2+2\theta}.
\end{align*}
Note the application of \eqref{E:persistence of reg} in the last step.

Finally, using \eqref{E:persistence of reg} and \eqref{E:time derivative},
\begin{align*}
\Bigl\| \jpn\Bigl\{& \frac{ie^{it}}{\lambda_n^5} \bigl[\partial_t (w_n \overline{w_n}^2)\bigr](\lambda_n^{-2}t,\lambda_n^{-1}x)\Bigr\}\Bigr\|_{L_{t,x}^{4/3}([-\lambda_n^2 T,\lambda_n^2 T]\times \R^2)} \\
&\lesssim \lambda_n^{-2} \|\jpn \partial_t(w_n \overline{w_n}^2)\|_{L_{t,x}^{4/3}(\R\times \R^2)} \\
&\lesssim \lambda_n^{-2} \Bigl\{ \|\jpn \partial_t w_n\|_{L_{t,x}^4(\R\times \R^2)}  \|w_n\|_{L_{t,x}^4(\R\times \R^2)}^2\\
&\quad +  \|\partial_t w_n\|_{L_{t,x}^4(\R\times \R^2)} \|\jpn w_n\|_{L_{t,x}^4(\R\times \R^2)} \|w_n\|_{L_{t,x}^4(\R\times \R^2)}    \Bigr\}\\
&\lesssim \lambda_n^{-2+3\theta}.
\end{align*}

This completes the proof of the lemma.
\end{proof}

We now make the promised modification in $\tilde{v}_n$ and show that the modified sequence approximately solves \eqref{nlkg1st}.  With $f_{n,j}$ as defined above, we consider the sequence
\begin{equation}\label{E:tilde tilde v}
\vtt_n(t) :=
\begin{cases}
\tilde{v}_n(t) - f_{n,3}(t) - f_{n,4}(t), &\qtq{if $|t| \leq T\lambda_n^2$,} \\
e^{-i(t-T\lambda_n^2)\jpn}\vtt_n(T\lambda_n^2), &\qtq{if $t > T\lambda_n^2$,} \\
e^{-i(t+T\lambda_n^2)\jpn}\vtt_n(-T\lambda_n^2), &\qtq{if $t<-T\lambda_n^2$.}
\end{cases}
\end{equation}
The key facts about $\vtt_n$ are that it is a good enough approximate solution to allow us to invoke Proposition~\ref{P:stability} and that it is close enough to $\tilde v_n$ to allow us to deduce that
the resulting solutions $v_n$ obey \eqref{E:Cinfty approx0}.  We collect these together in the following proposition:

\begin{proposition}  \label{L:vtt}
For each $\eps>0$ there exist $T$ and $N$ so that for each $n\geq N$,
$$
(-i \partial_t + \jpn)\vtt_n + \mu \jpn^{-1} (\Re \vtt_n)^3 = \tilde{e}_1 + \tilde{e}_2 +\tilde{e}_3,
$$
with
$$
\|\tilde{e}_1\|_{L^1_tH^{1/2}_x(\R\times \R^2)} + \|\jpn(\tilde{e}_2+\tilde{e}_3)\|_{L^{4/3}_{t,x}(\R\times \R^2)}  \leq \eps.
$$
Moreover,
\begin{align}\label{E:diff vn}
\bigl\| \vtt_n -\tilde v_n \bigr\|_{L^\infty_t H^{1/2}_x(\R\times \R^2)}  +  \|\vtt_n -\tilde v_n\|_{L_{t,x}^4(\R\times \R^2)}  \leq \eps.
\end{align}
In particular, by \eqref{E:vn bounds} we have $\| \vtt_n \|_{L^\infty_t H^{1/2}_x(\R\times \R^2)}  +  \|\vtt_n \|_{L_{t,x}^4(\R\times \R^2)}  \lesssim _{M(\phi)} 1$.
\end{proposition}

\begin{proof}
Let $I_n := [-\lambda_n^2T,\lambda_n^2T]$.  On this interval, direct computation reveals
\begin{gather*}
\tilde{e}_1 = e_1, \quad \tilde{e}_2 = e_2, \qtq{and}
\tilde{e}_3 = \mu \jpn^{-1}\Bigl[\bigl(\Re (\tilde{v}_n - f_{n,3} - f_{n,4})\bigr)^3 - (\Re \tilde{v}_n)^3\Bigr].
\end{gather*}
By  \eqref{E:e_1} and \eqref{E:e_2}, we have
$$
\|\tilde{e}_1\|_{L^1_t H^{1/2}_x(I_n\times \R^2)} + \|\jpn \tilde{e}_2\|_{L^{4/3}_{t,x}(I_n \times \R^2)} \lesssim T \lambda_n^{-2+4\theta} + \lambda_n^{-1+\theta}.
$$
On the other hand, by the triangle inequality, H\"older, \eqref{E:vn bounds}, and Lemma~\ref{L:Xsb-type},
\begin{align*}
\|\jpn\tilde{e}_3\|_{L^{4/3}_{t,x}(I_n\times \R^2)}
&\lesssim \bigl[ \|\tilde{v}_n\|_{L^4_{t,x}(I_n\times \R^2)}^2 + \|f_{n,3}\|_{L^4_{t,x}(I_n\times \R^2)}^2 + \|f_{n,4}\|_{L^4_{t,x}(I_n\times \R^2)}^2 \bigr] \\
&\qquad \times \bigl[ \|f_{n,3}\|_{L^4_{t,x}(I_n\times \R^2)} + \|f_{n,4}\|_{L^4_{t,x}(I_n\times \R^2)} \bigr]\\
&\lesssim \lambda_n^{-2+3\theta}.
\end{align*}
These bounds show that for any $T$ one may choose $N$ sufficiently large so that for $n\geq N$,
$$
\|\tilde{e}_1\|_{L^1_tH^{1/2}_x(I_n\times \R^2)} + \|\jpn(\tilde{e}_2+\tilde{e}_3)\|_{L^{4/3}_{t,x}(I_n\times \R^2)}  \leq \tfrac12 \eps.
$$

For the complementary time intervals, $\tilde e_1\equiv \tilde e_2 \equiv 0$ and $\tilde e_3 = \mu \jpn^{-1} (\Re \vtt_n)^3$.   By Proposition~\ref{P:large times approx},
Lemma~\ref{L:Xsb-type}, and the Strichartz inequality, we have that for $T$ and $n$ sufficiently large,
\begin{align*}
\bigl\| \jpn \tilde e_3 \bigr\|_{L^{4/3}_{t,x}(|t|\geq T \lambda_n^2)} &\lesssim \bigl\| \vtt_n \bigr\|_{L^{4}_{t,x}(|t|\geq T \lambda_n^2)}^3 \\
&\lesssim \bigl\| \tilde v_n \bigr\|_{L^{4}_{t,x}(|t|\geq T \lambda_n^2)}^3 +  \bigl\| f_{n,3} \bigr\|_{L^{\infty}_tH^{1/2}_x}^3 + \bigl\| f_{n,4} \bigr\|_{L^{\infty}_tH^{1/2}_x}^3 \\
&\leq \tfrac12 \eps.
\end{align*}

We now turn our attention to \eqref{E:diff vn}.  The contribution from $I_n$ is controlled by Lemma~\ref{L:Xsb-type}, while the contribution from the complementary time intervals
is controlled by combining this lemma with the Strichartz inequality.

This completes the proof of the proposition.
\end{proof}

We are now ready to complete the proof of Theorem~\ref{T:embed} in the case $\nu_n\equiv 0$.  Combining Lemma~\ref{L:initial data} and Proposition~\ref{L:vtt},
we are in a position to apply the stability result Proposition~\ref{P:stability} with $\vtt_n$ as the approximate solution, and so obtain (for $n$ sufficiently large)
a solution $v_n$ to \eqref{nlkg1st} with initial data $v_n(0)=\phi_n$ and finite scattering size.  Moreover, by \eqref{E:diff vn},
\begin{align}\label{E:vn is tvn}
\lim_{n\to \infty} \bigl\{\| v_n (t) -\tilde v_n(t-t_n) \|_{L^\infty_t H^{1/2}_x(\R\times \R^2)}  +  \|v_n(t) -\tilde v_n(t-t_n)\|_{L_{t,x}^4(\R\times \R^2)}\bigr\}=0.
\end{align}

Finally, we verify \eqref{E:Cinfty approx}.  By the density of $C^\infty_c$ in $L^4_{t,x}$ we may choose $\psi_\eps$ so that
$$
\bigl\| e^{-it} D_{\lambda_n} \bigl[ \psi_\eps(\lambda_n^{-2}t) - w_\infty(\lambda_n^{-2}t) \bigr] \bigr\|_{L^4_{t,x}} = \| \psi_\eps - w_\infty \|_{L^4_{t,x}} < \tfrac12 \eps.
$$
Combining this with \eqref{E:vn is tvn}, we see that it suffices to show
\begin{align*}
\| \tilde v_n - e^{-it} D_{\lambda_n} w_\infty(\lambda_n^{-2}t) \|_{L_{t,x}^4(\R\times \R^2)} < \tfrac12 \eps \quad\text{for $n$ sufficiently large}.
\end{align*}
By the definition of $\tilde v_n$ and the triangle inequality,
\begin{align*}
\| \tilde v_n - e^{-it}& D_{\lambda_n} w_\infty(\lambda_n^{-2}t) \|_{L_{t,x}^4(\R\times \R^2)}  \\
& \lesssim \| \tilde v_n \|_{L_{t,x}^4(|t|> T\lambda_n^2)}  + \| w_n - w_\infty \|_{L_{t,x}^4(\R\times \R^2)} + \| w_\infty \|_{L_{t,x}^4(|t|> T)}.
\end{align*}
Each of these can be made arbitrarily small by first choosing $T$ large and then $n$ also sufficiently large; specifically, we apply Proposition~\ref{P:large times approx},
\eqref{E:limit w_infty}, and the dominated convergence theorem, respectively.

This completes the treatment of Theorem~\ref{T:embed} in the case $\nu_n\equiv0$.
\end{proof}

We now turn to the general case, in which our only assumption on $\nu_n$ is that $\nu_n \to \nu \in \R^2$.

\begin{proof}[Proof in the general case]
Recall that $(\tilde t_n,\tilde x_n) :=L_{\nu_n}(t_n,x_n)$ and hence, by the commutation rule \eqref{E:Boost translates},
\begin{equation}\label{E:Boosted phi}
\phi_n = T_{x_n} e^{it_n \jpn} \W_{\nu_n} D_{\lambda_n} P_{\leq \lambda_n^{\theta}} \phi
    = \W_{\nu_n} T_{\tilde x_n} e^{i\tilde t_n\jpn} D_{\lambda_n}P_{\leq \lambda_n^{\theta}} \phi.
\end{equation}
By spatial translation invariance, we may alter $x_n$ to whatever value is convenient (previously, we set $x_n\equiv 0$).  For this part of the proof we
choose $x_n =\nu_nt_n/\jp{\nu_n}$, which has the effect that $\tilde x_n \equiv 0$ and $\tilde t_n = t_n/\jp{\nu_n}$.

By our proof in the case $\nu_n \equiv 0$, for $n$ sufficiently large there is a global solution $v_n^0$ to \eqref{nlkg1st} with initial data
\begin{equation}\label{E:vn0 data}
v_n^0 (0) = T_{\tilde x_n} e^{i\tilde t_n\jpn} D_{\lambda_n}P_{\leq \lambda_n^{\theta}} \phi.
\end{equation}
Moreover, it obeys $S_{\R}(v_n^0) \lesssim_{M(\phi)} 1$ and for each $\eps > 0$, there exists $\psi_{\eps}^0 \in C^{\infty}_c(\R \times \R^2)$ and $N_{\eps}^0$ such that
\begin{equation} \label{E:approx vn0}
\bigl\| \Re\bigl\{ v_n^0(t+\tilde t_n,x+\tilde x_n) - \tfrac{e^{-it}}{\lambda_n} \psi_{\eps}^0(\tfrac{t}{\lambda_n^2}, \tfrac{x}{\lambda_n}) \bigr\} \bigr\|_{L^4_{t,x}} < \eps
\end{equation}
whenever $n \geq N_{\eps}^0$.

As $v_n^0$ solves \eqref{nlkg1st}, so $u_n^0:=\Re v_n^0$ solves \eqref{nlkg} and thus by Lorentz invariance, $u_n^1:=u_n^0\circ L_{\nu_n}$ also solves \eqref{nlkg}.  Note
that it is necessary to pass through real solutions here since
\begin{equation}\label{E:vn1 defn}
v_n^1 := (1+i\jpn^{-1}\partial_t) u_n^1 = (1+i\jpn^{-1}\partial_t) \Re v_n^0\circ L_{\nu_n}
\end{equation}
(cf. \eqref{E:real to complex}) will not equal $v_n^0\circ L_{\nu_n}$ except in some exceptional circumstances.  The former solves \eqref{nlkg1st}, while the latter need not.

By the manner in which it is constructed, we expect $v_n^1(t)$ to be a good approximation to the sought-after $v_n(t)$.  Both are solutions to \eqref{nlkg1st};
however, they have different initial data because $\W_{\nu_n}$ only faithfully represents the action of Lorentz boosts on solutions of the \emph{linear} Klein--Gordon equation.
Nevertheless, we will now prove the discrepancy to be small, which will allow us to apply the stability result Proposition~\ref{P:stability}.

\begin{proposition} \label{P:vn1}
For $n$ sufficiently large, $v_n^1$ is a global strong solution to \eqref{nlkg1st}.  Moreover, $\sup_n  S_\R(v_n^1) \lesssim_{M(\phi)} 1$ and
\begin{equation}\label{E:vn1-phin}
     \lim_{n \to \infty} \|v_n^1(0) - \phi_n\|_{H^1_x} = 0.
\end{equation}
\end{proposition}

\begin{proof}
By Corollary~\ref{C:boostable}, we have that $u_n^1$ is a strong solution to \eqref{nlkg}.  This implies that $v_n^1$ is a strong solution to \eqref{nlkg1st}.  As $S_\R(v_n^1)=S_\R(v_n^0)$,
the spacetime bound is inherited directly from $v_n^0$.

We now turn to the initial data.  As in the proof of Corollary~\ref{C:boostable}, we decompose
$$
u_n^0 = \uolin_n + \tilde u_n^0,
$$
where $\uolin_n$ solves the linear Klein--Gordon equation with initial data
$$
(1+i\jpn^{-1}\partial_t)\uolin_n(0) = v_n(0) = \W_{\nu_n}^{-1} \phi_n.
$$
Then by the action of $\W_\nu$ on linear solutions (cf. \eqref{E:Boost linear soln} or \eqref{E:lin o L is Wa}), we have
\begin{align*}
(1+i\jpn^{-1}\partial_t) [\uolin_n\circ L_{\nu_n}](0) &= \W_{\nu_n} v_n(0) = \phi_n,
\end{align*}
from which we deduce that $\|v_n^1(0) - \phi_n\|_{H^1_x} = \|\tilde{u}_n^0 \circ L_{\nu_n}(0,\cdot) \|_{H^1_x}$.

By construction, $\tilde{u}_n^0(0,\cdot)\equiv 0$.  Thus, we need only estimate the contribution from the nonlinearity in the spacetime region
\begin{align*}
\Omega_n = \bigl\{(t,x) : 0 < \jp{\nu_n}t < -\nu_n \cdot x\bigr\}  \cup \bigl\{(t,x) : -\nu_n \cdot x < \jp{\nu_n} t < 0 \bigr\}.
\end{align*}

To do this, we argue in much the same manner as in the proof of Corollary~\ref{C:boostable}, using $\tilde u_n^0$ in place of the $\tilde{u}$ appearing there.
As in \eqref{E:2},
\begin{align*}
\limsup_{n\to\infty} \tfrac12 \|\tilde{u}_n^0 \circ L_{\nu_n}(0,\cdot) \|_{H^1_x}^2
    &\leq \limsup_{n\to\infty}  \iint_{\Omega_n} |\nabla_{t,x} \cdot \mathfrak{p}_n| \,dx\,dt \\
&\lesssim \limsup_{n\to\infty}  \iint_{\Omega_n}|u_n^0(t,x)|^3|\nabla_{t,x}\tilde{u}_n^0(t,x)|\, dx\, dt \\
&\lesssim \limsup_{n\to\infty} \|u_n^0\|_{L^4_{t,x}(\Omega_n)}^3 \|\nabla_{t,x}\tilde{u}_n^0\|_{L^4_{t,x}(\R \times \R^2)}.
\end{align*}
Therefore, to complete the proof of the proposition, we merely need to verify the following:  For $n$ sufficiently large,
\begin{equation} \label{E:control grad tildeu}
\|\nabla_{t,x}\tilde{u}_n^0\|_{L^4_{t,x}(\R \times \R^2)} \lesssim_{M(\phi)} 1
\end{equation}
and
\begin{equation} \label{E:u small on omega}
\lim_{n \to \infty}  \|u_n^0\|_{L^4_{t,x}(\Omega_n)} = 0.
\end{equation}

We begin with \eqref{E:control grad tildeu}.  By the triangle inequality,
$$
\|\nabla_{t,x}\tilde{u}_n^0\|_{L^4_{t,x}} \leq \|\nabla_{t,x}u_n^0\|_{L^4_{t,x}} + \|\nabla_{t,x}\uolin_n\|_{L^4_{t,x}}.
$$
Furthermore, by the Strichartz inequality, the linear term satisfies the bound
\begin{align*}
\|\nabla_{t,x}\uolin_n\|_{L^4_{t,x}} &\lesssim \| v_n^0(0) \|_{H^{3/2}_x} = \|D_{\lambda_n}P_{\leq\lambda_n^{\theta}}\phi\|_{H^{3/2}_x} \lesssim_{M(\phi)} 1.
\end{align*}

To bound the contribution coming from $u_n^0$, we use persistence of regularity \eqref{E:nlkg persistence} and the fact that $S_{\R}(u_n^0) \lesssim_{M(\phi)} 1$ to see that
$$
\|\nabla_{t,x} u_n^0\|_{L^4_{t,x}} \lesssim_{M(\phi)} \|\jpn^{3/2} D_{\lambda_n}P_{\leq \lambda_n^{\theta}} \phi\|_{L^2_x} \lesssim_{M(\phi)} 1.
$$

This completes the proof of \eqref{E:control grad tildeu}; we turn now to \eqref{E:u small on omega}.

By the approximation \eqref{E:approx vn0} and the triangle inequality, it suffices to show that
\begin{equation} \label{E:u small 1}
\lim_{n \to \infty} \int_{\Omega_n} \lambda_n^{-4} \bigl| \psi\bigl(\tfrac{t-\tilde t_n}{\lambda_n^2},\tfrac{x-\tilde x_n}{\lambda_n}\bigr) \bigr|^4 \,dx\,dt = 0
\end{equation}
for every $\psi \in C^{\infty}_c(\R \times \R^2)$.

To do this, we consider the support of the integrand.  As $\tilde x_n\equiv 0$,
$$
\bigl(\tfrac{t-\tilde t_n}{\lambda_n^2},\tfrac{x-\tilde x_n}{\lambda_n}\bigr) \in \supp \psi \implies |x|\lesssim_\psi \lambda_n,
\qtq{while} (t,x) \in \Omega_n \implies |t| < |x|.
$$
Therefore, the support of the integrand has measure $\lesssim_\psi \lambda_n^3$ and so
$$
\text{LHS\eqref{E:u small 1}} \lesssim_\psi \lambda_n^{-4} \lambda_n^3 \|\psi\|_{L^{\infty}_{t,x}} \lesssim_\psi \lambda_n^{-1} \to 0.
$$
This completes the proof of \eqref{E:u small on omega}, and so the proof of Proposition~\ref{P:vn1}.
\end{proof}

We now return to the proof of Theorem~\ref{T:embed} in the general case.  Combining Proposition~\ref{P:vn1} with Proposition~\ref{P:stability}, we deduce
that for $n$ sufficiently large there exists a global solution $v_n$ to \eqref{nlkg1st} with $v_n(0) = \phi_n$ and $S_\R(v_n)\lesssim_{M(\phi)} 1$.
Moreover,
$$
\lim_{n \to \infty} \bigl\|\Re\bigl\{v_n - v_n^1 \bigr\}\bigr\|_{L^4_{t,x}} = 0.
$$

Combining this estimate with $\Re v_n^0 = \Re v_n^1\circ L_{\nu_n}^{-1}$ and \eqref{E:approx vn0} yields \eqref{E:Cinfty approx0}, thus completing the proof of Theorem~\ref{T:embed}.
\end{proof}


\section{Minimal-energy blowup solutions}\label{S:min blowup}


The goal of this section is to prove Theorem~\ref{T:reduct}, which asserts that failure of of our main result, Theorem~\ref{T:ST bounds}, would imply the existence
and almost periodicity (modulo translations) of minimal-energy counterexamples.

As described in the introduction, if Conjecture~\ref{Conj:NLKG} were to fail, then there would exist a \emph{critical energy} $E_c >0$ (also $E_c < E(Q)$ in the focusing case),
defined to be the unique positive number possessing the following properties:  First, if $u$ is a real-valued, global solution to \eqref{nlkg} with $E(u) < E_c$
(and $M(u(0)) < M(Q)$ in the focusing case), then $S_{\R}(u) \lesssim_{E(u)} 1$; second, there exists a sequence $u_n$ of global solutions to \eqref{nlkg}
such that $E(u_n) \leq E_c$ (and $M(u_n(0)) < M(Q)$ in the focusing case), $\lim_{n \to \infty} E(u_n) = E_c$, and $\lim_{n \to \infty} S_{\R}(u_n) = \infty$.

We pause now for two remarks on the preceding discussion.  First, the fact that $E_c>0$ is a consequence of the small-data theory presented in Proposition~\ref{P:lwp}.
In the focusing case, we also invoke \eqref{E:coercive E} to see that for $M(u(0)) < M(Q)$, small energy implies small $H^1_x$ norm.

Second, the solutions $u_n$ used in the definition of $E_c$ are stated to be global (in time).  This involves no loss of generality as can be seen in at least two ways:
Either (a) we choose $E(u_n)$ to converge to $E_c$ from below; thus, not only is $u_n$ global but, by the definition of $E_c$, even admits global spacetime bounds.
Or (b) we invoke Corollary~\ref{Cor:gwp}.

The main step in proving Theorem~\ref{T:reduct} is the following proposition.

\begin{proposition}[Palais--Smale condition modulo translations] \label{P:palais smale}
Fix $\mu = \pm1$, suppose Conjecture~\ref{Conj:NLS} holds but Conjecture~\ref{Conj:NLKG} fails for this value of $\mu$, and let $E_c$ denote the critical energy.
Accordingly, let $u_n$ be a sequence of global solutions to \eqref{nlkg} such that  the following hold
\begin{gather}
E(u_n) \leq E_c \text{ and } E_c = \lim_{n \to \infty}E(u_n),  \\
M(u_n(0)) < M(Q) \text{ in the focusing case, and}  \label{E:wee mass}\\
\lim_{n \to \infty} S_{\leq 0}(u_n) = \lim_{n \to \infty} S_{\geq 0}(u_n) = \infty.  \label{E:infinte S}
\end{gather}
Then after passing to a subsequence, $(u_n(0),\partial_t u_n(0))$ converges in $H^1_x \times L^2_x$, modulo translations.
\end{proposition}

\begin{proof}
We will continue to work with the first-order version of our equation, \eqref{nlkg1st}.  Correspondingly, let
$$
v_n := u_n + i \jpn^{-1}\partial_t u_n .
$$
Thus, our goal is to prove that, after passing to a subsequence, $v_n(0)$ converges modulo translations in $H^1_x$.
Using Proposition~\ref{P:coercive} in conjunction with \eqref{E:wee mass}, we observe that in the focusing case, $v_n$ satisfies
\begin{align}
\label{E:mass vn}
  \| v_n(0) \|_{L^2_x}^2 =: M(v_n(0)) & \leq 2E_c < M(Q),
\end{align}
and that in both the focusing and defocusing cases, we have
\begin{align}
\label{E:H1 vn}
\|v_n(0)\|_{H^1_x}^2 & \lesssim E(v_n) \leq E_c.
\end{align}

As it is bounded in $H^1_x$, we may apply Theorem~\ref{T:lpd} to the sequence $v_n(0)$ to obtain a linear profile decomposition
\begin{equation} \label{E:ps lpd}
v_n(0) = \sum_{j=1}^J \phi_n^j + w_n^J, \quad 1 \leq J < J_0,
\end{equation}
where
\begin{equation} \label{E:phinj}
\phi_n^j = T_{x_n^j} e^{it_n^j \jpn} \W_{\nu_n^j} D_{\lambda_n^j} P_n^j \phi^j.
\end{equation}
Note that $J_0 > 1$, for otherwise, \eqref{E:infinte S} would be inconsistent with \eqref{E:lpd ST norms to zero}.  Passing to a further subsequence,
we may assume that $M(\phi_n^j)$ and $E(\phi_n^j)$ converge for each $1\leq j <J_0$.

From Proposition~\ref{P:nrg decoup}, we also know that the energy decouples:
\begin{equation} \label{E:bound nrg sum}
\lim_{n \to \infty} \sum_{j=1}^J E(\phi_n^j) + E(w_n^J) = \lim_{n \to \infty} E(v_n) = E_c,
\end{equation}
for each $1 \leq J < J_0$.

\begin{lemma} \label{L:setup 2 cases}
After passing to a subsequence, one of the following scenarios occurs:
{\rm\textbf{Case I.}} There is only a single profile and it satisfies
\begin{equation} \label{E:one profile}
\lim_{n \to \infty} E(\phi_n^1) = E_c.
\end{equation}
{\rm\textbf{Case II.}} There exists $\delta > 0$ such that for every $1 \leq j < J_0$,
\begin{equation} \label{E:small profiles}
\lim_{n \to \infty} E(\phi_n^j) < E_c - \delta.
\end{equation}
Irrespective of the above, in the focusing case we also have that for each $j$ and $J$,
\begin{align}\label{E:bound single mass}
 M( \phi_n^j) < M(Q) \qtq{and} M(w_n^J)<M(Q)
\end{align}
when $n$ is sufficiently large $($possibly depending on $j$, $J)$.
\end{lemma}

\begin{proof}
We begin with \eqref{E:bound single mass}.  By \eqref{E:lpd wk lim} and a simple inductive argument (as in the proof of \eqref{E:lpd H1 decoup}), we see that the mass decouples, that is,
$$
\lim_{n \to \infty} \Bigl\{  M(v_n(0)) - \sum_{j=1}^J M(\phi_n^j) - M(w_n^J) \Bigr\} =0 \qtq{for each} 1 \leq J < J_0.
$$
As masses are non-negative, this and \eqref{E:mass vn} imply the validity of \eqref{E:bound single mass}.

With \eqref{E:bound single mass} proved, \eqref{E:coercive E} implies that the summands in \eqref{E:bound nrg sum} are positive for $n$ sufficiently large (depending on $J$)
in the focusing case; in the defocusing case this is manifestly true.  This positivity allows us to deduce that if \eqref{E:one profile} fails then \eqref{E:small profiles} must hold.  For $j=1$, this is obvious.  For $j\geq 2$
we may set $\delta = \lim_{n\to\infty} E(\phi_n^1)$, which is positive since $\phi^1\neq 0$, by construction.
\end{proof}

Now we move to the consideration of the two scenarios described in Lemma~\ref{L:setup 2 cases}.

\medskip

{\bf Case I:}  Assume that \eqref{E:one profile} occurs.  Then by \eqref{E:bound nrg sum}, together with \eqref{E:coercive E} invoked for $w_n^J$, we have that
\begin{equation} \label{E:wn1 to 0}
v_n - \phi_n^1 = w_n^1 \to 0 \qtq{in $H^1_x$.}
\end{equation}
Our analysis now breaks into three sub-cases.

\medskip

{\bf Case IA:}  Suppose that $\lambda_n^1 \to \infty$.  We will apply Theorem~\ref{T:embed}, but in the focusing case we must first verify the following:

\begin{lemma} \label{L:check embed hypothesis}
Assume that we are in the focusing case.  Let $1 \leq j < J_0$ and assume that $\lim_{n \to \infty} \lambda_n^j = \infty$.  Then
$$
M(\phi^j) < M(Q).
$$
\end{lemma}

\begin{proof}
By \eqref{E:bound nrg sum} and Lemma~\ref{L:boost action}, we have
\begin{align*}
2E_c &\geq \lim_{n \to \infty} 2E(\phi_n^j) = \lim_{n \to \infty} \|T_{x_n^j}e^{it_n^j\jpn} \W_{\nu_n^j}D_{\lambda_n^j}P_n^j \phi^j\|_{H^1_x}^2\\
&= \lim_{n \to \infty} \int  \jp{\xi/\lambda_n^j} \bigl\langle\ell_{-\nu_n^j}(\xi/\lambda_n^j)\bigr\rangle |P_{\leq (\lambda_n^j)^{\theta}} \widehat{\phi}^j(\xi)|^2\, d\xi
    = \jp{\nu_{\infty}^j}\|\phi^j\|_{L^2_x}^2.
\end{align*}
Since $2E_c < 2E(Q) = M(Q)$, the lemma is proved.
\end{proof}

Thus by Theorem~\ref{T:embed}, for $n$ sufficiently large, the solution $v_n^1$ to \eqref{nlkg1st} with initial data $v_n^1(0) = \phi_n^1$
is global and satisfies $S_{\R}(v_n^1) \lesssim_{E_c} 1$.  We may now use \eqref{E:wn1 to 0} and Proposition~\ref{P:stability} to conclude that $S_{\R}(v_n) < \infty$, a contradiction.

\medskip

We must therefore have that $\lambda_n^1 \equiv 1$.  This implies that $\phi^1 \in H^1_x$,
$$
\phi_n^1 = T_{x_n^1}e^{it_n^1\jpn}\phi^1,
$$
and either $t_n^1 \to \pm \infty$ or $t_n^1 \equiv 0$.

\medskip

{\bf Case IB:}  Suppose that $\lambda_n^1\equiv 1$ and  $t_n^1 \to \pm\infty$.  We treat the case $t_n^1 \to -\infty$, the other case being similar.
By the Strichartz inequality, $e^{-it\jpn}\phi^1 \in L^4_{t,x}(\R \times \R^2)$, and so
$$
\|e^{-i(t-t_n^1)\jpn}  \phi^1\|_{L^4_{t,x}([0,\infty) \times \R^2)} \to 0 \quad \text{as $n\to\infty$.}
$$
Hence by Proposition~\ref{P:lwp}, if $v_n^1$ is the solution to \eqref{nlkg1st} with initial data $v_n^1(0) = \phi_n^1$,
then for $n$ sufficiently large, $S_{\geq 0}(v_n^1) < \infty$.  As in Case IA, we can now use Proposition~\ref{P:stability}
to conclude that for $n$ large, $S_{\geq 0}(v_n)<\infty$, a contradiction.

\medskip

{\bf Case IC:}  If $\lambda_n^1\equiv 1$ and $t_n^1 \equiv 0$, we have reduced the linear profile decomposition \eqref{E:lpd} to
$$
T_{-x_n^1}v_n(0) = \phi^1 + T_{-x_n^1}w_n^1.
$$
Combining this with \eqref{E:wn1 to 0}, we have proved the proposition when \eqref{E:one profile} holds.

\medskip

{\bf Case II:}  We will show that this is inconsistent with \eqref{E:infinte S} by using \eqref{E:ps lpd} to produce a nonlinear profile decomposition of the $v_n$ and then
applying the stability theory.  We begin by introducing nonlinear profiles $v_n^j$; their definition depends on the behavior of~$\lambda_n^j$.

First assume that $j$ is such that $\lambda_n^j \equiv 1$.  Then $\phi^j \in H^1_x$ and
$$
\phi_n^j = T_{x_n^j}e^{it_n^j\jpn}\phi^j.
$$
If, in addition, $t_n^j \equiv 0$, then we let $v^j$ be the maximal-lifespan solution to \eqref{nlkg1st} with $v^j(0) = \phi^j$.  If $t_n^j \to -\infty$ (respectively $t_n^j \to \infty$),
then we let $v^j$ be the maximal-lifespan solution to \eqref{nlkg1st} which scatters forward (respectively backward) in time to $e^{-it\jpn}\phi^j$.

\begin{lemma} \label{L:vj global}
In Case~II, if $\lambda_n^j \equiv 1$ for some $j$, then $v^j$ defined as above is global.
\end{lemma}

\begin{proof}
This follows from Corollary~\ref{Cor:gwp}.  In the focusing case however, we must first establish
$$
M(\Re \phi^j) < M(Q) \qtq{and} E(\phi^j) < E(Q).
$$
Since $M(\phi^j) = M(\phi_n^j)$, the first inequality is immediate from \eqref{E:bound single mass}.  We turn to the energy bound.  If $t_n^j \equiv 0$, then $E(\phi^j) \equiv E(\phi_n^j)$,
and we are done.  If $t_n^j \to \pm \infty$, then by using the dispersive estimate \eqref{E:Dispersive} and approximating $\phi^j$ in $H^1_x$ by Schwartz functions, we see that
$$
\lim_{n \to \infty} E(\phi_n^j) = \lim_{n \to \infty} \tfrac12\|\phi_n^j\|_{H^1_x}^2 = \tfrac12 \|\phi^j\|_{H^1_x}^2 \geq E(\phi^j),
$$
and hence by Lemma~\ref{L:setup 2 cases}, $E(\phi^j) < E_c < E(Q)$ in this case as well.
\end{proof}

Thus if $\lambda_n^j \equiv 1$, we may define nonlinear profiles by
$$
v_n^j(t,x) := v^j(t-t_n^j,x-x_n^j).
$$

Next, suppose that $\lim_{n \to \infty}\lambda_n^j = \infty$.  Then by Theorem~\ref{T:embed} (and Lemma~\ref{L:check embed hypothesis} in the focusing case), for $n$ sufficiently large
we may define $v_n^j$ to be the solution to \eqref{nlkg1st} with initial data $v_n^j(0) = \phi_n^j$.

\begin{lemma} \label{L:nl profiles nice}
In Case~II, for each $j$ $($regardless of the behavior of the $\lambda_n^j)$ we have
\begin{align}
\label{E:same energy}
\lim_{n \to \infty} E(v_n^j) &= \lim_{n \to \infty} E(\phi_n^j),  \\
\label{E:bound mass vnj}
\lim_{n \to \infty} M(v_n^j(0)) &< M(Q) \text{ in the focusing case, and }\\
\label{E:bounded ST norm}
\lim_{n \to \infty} S_{\R}(v_n^j) &\lesssim \lim_{n \to \infty} E(v_n^j)^2.
\end{align}
Furthermore, for each $j$ and $\eps > 0$, there exists $\psi=\psi_{\eps} \in C^{\infty}_c(\R \times \R^2)$ and $N_{j,\eps}$ such that if $\psi_n^j$ is defined as in
\eqref{E:nl profile approx} and $n > N_{j,\eps}$, then we have
\begin{equation} \label{E:psi approximates v}
\|\Re(\psi_n^j - v_n^j)\|_{L^4_{t,x}(\R \times \R^2)} < \eps.
\end{equation}
\end{lemma}

\begin{proof}
Equality \eqref{E:same energy} is a tautology if $\lambda_n^j \to \infty$ or $\lambda_n^j \equiv 1$ and $t_n^j \equiv 0$, since in these cases $v_n^j(0) = \phi_n^j$.
If $\lambda_n^j \equiv 1$ and $t_n^j \to \pm\infty$, then by the definition of $v_n^j$ and \eqref{E:scattering nrg}, we have
$$
E(v_n^j) = E(v^j) = \tfrac12 \|\phi^j\|_{H^1_x}^2 = \lim_{n \to \infty} E(\phi_n^j),
$$
where for the last equality, we have used the dispersive estimate as in the proof of Lemma~\ref{L:vj global}.

Inequality \eqref{E:bound mass vnj} follows easily from \eqref{E:bound single mass} and the definition of~$v_n^j$.

When $\lim_{n\to\infty} E(\phi_n^j)$ is below the small data threshold, \eqref{E:bounded ST norm} follows from Proposition~\ref{P:lwp}.  Note that in the focusing case,
\eqref{E:bound single mass} and \eqref{E:coercive E} imply that the energy controls the $H^1_x$ norm.  On the other hand, by \eqref{E:bound nrg sum} the limiting energy
can only exceed this threshold for finitely many values of $j$.  For these cases, we invoke \eqref{E:small profiles} and the definition of $E_c$.  As we are invoking the
contradiction hypothesis here, there is no hope of being explicit about the constant in \eqref{E:bounded ST norm} other than that it is independent of $j$.

As for \eqref{E:psi approximates v}, in the case $\lambda_n^j \equiv 1$, this follows from the fact that $v_n^j$ is just a translate of $v^j \in L^4_{t,x}(\R \times \R^2)$.
In the case $\lambda_n^j \to \infty$, this approximation follows from Theorem~\ref{T:embed}.
\end{proof}

For $1 \leq J < J_0$, we let
$$
V_n^J(t) := \sum_{j=1}^J v_n^j(t) + e^{-it\jpn}w_n^J,
$$
which is defined globally for $n$ sufficiently large (depending on $J$).  Our immediate goal is to use Proposition~\ref{P:stability} to show
that $V_n^J(t)$ is a good approximation to $v_n(t)$ when $n$ and $J$ are sufficiently large.

\begin{lemma} \label{L:nlpd stability setup}
We have the following spacetime bounds on $V_n^J$
\begin{equation} \label{E:uniform st bounds}
\limsup_{J \to \infty} \limsup_{n \to \infty}\bigl\{ \|\Re V_n^J\|_{L^4_{t,x}} + \|V_n^J\|_{L^{\infty}_t H^{1/2}_x} \bigr\} < \infty.
\end{equation}
The $V_n^J$ are approximate solutions to \eqref{nlkg1st} in the sense that
$$
(-i\partial_t +\jpn)V_n^J + \mu \jpn^{-1}(\Re V_n^J)^3 =  E_n^J,
$$
where
\begin{equation} \label{E:nearly nlkg}
\lim_{J \to \infty} \limsup_{n \to \infty} \|\jpn E_n^J\|_{L^{4/3}_{t,x}} = 0.
\end{equation}
Furthermore, for each $J$ we have
\begin{equation} \label{E:initial H1 ok}
\lim_{n \to \infty} \| v_n(0) - V_n^J(0)\|_{H^1_x(\R^2)} = 0.
\end{equation}
\end{lemma}

\begin{proof}

We begin with \eqref{E:initial H1 ok}.  By the triangle inequality and the definitions,
\begin{equation*}
\lim_{n\to\infty} \| v_n(0) - V_n^J(0)\|_{H^1_x(\R^2)} \leq \lim_{n\to\infty}  \sum_{j=1}^J\|v_n^j(0) - \phi_n^j\|_{H^1_x(\R^2)} =0.
\end{equation*}
To see that the limit vanishes, we note that each of the summands is identically zero, except in the case when $\lambda_n^j \equiv 1$ and $|t_n^j| \to \infty$.  However, even in this case,
the difference tends to zero in $H^1_x(\R^2)$ by construction.

As a preliminary to the main part of the proof, we note that combining \eqref{E:bounded ST norm} and \eqref{E:bound nrg sum} yields
\begin{equation}\label{E:the dude}
\limsup_{J \to \infty}\limsup_{n \to \infty} \sum_{j=1}^J \|\Re v_n^j\|_{L^4_{t,x}}^2 \lesssim \lim_{J \to \infty}\lim_{n \to \infty} \sum_{j=1}^J E(v_n^j) \leq E_c.
\end{equation}

We now bound the $L^4_{t,x}$ term in \eqref{E:uniform st bounds}.  By \eqref{E:psi approximates v} and Proposition~\ref{P:nonlinear decoup}, the nonlinear profiles decouple
in the sense that whenever $j \neq j'$, we have
\begin{equation} \label{E:vs decoup}
\lim_{n \to \infty} \|\Re v_n^j \Re v_n^{j'}\|_{L^2_{t,x}(\R \times \R^2)} = 0.
\end{equation}
Combining this with \eqref{E:lpd ST norms to zero} and then using \eqref{E:the dude} shows
\begin{align}\label{E:uniform A}
&\limsup_{J \to \infty}\limsup_{n \to \infty} \|\Re V_n^J\|_{L^4_{t,x}}^4
=\limsup_{J \to \infty}\limsup_{n \to \infty} \sum_{j=1}^J \|\Re v_n^j\|_{L^4_{t,x}}^4
    \lesssim E_c^2.
\end{align}

Next, we prove \eqref{E:nearly nlkg}.  A simple computation shows that
$$
E_n^J(t) = \mu\jpn^{-1}\biggl\{\Re \sum_{j=1}^J v_n^j(t) + \Re e^{-it\jpn}w_n^J\biggr\}^3 - \mu \jpn^{-1}\sum_{j=1}^J (\Re v_n^j)^3,
$$
and so, by the triangle inequality it suffices to show
\begin{gather}
\label{E:error from w}
\lim_{J \to \infty} \limsup_{n \to \infty} \ \biggl\|\biggl(\sum_{j=1}^J \Re v_n^j + e^{-it\jpn}w_n^J\biggr)^3
    - \biggl(\sum_{j=1}^J \Re v_n^j\biggr)^3\biggr\|_{L^{4/3}_{t,x}(\R \times \R^2)}   = 0 \\
\label{E:error from parentheses}
\text{and} \qquad \lim_{J \to \infty} \limsup_{n \to \infty} \ \biggl\|\biggl(\sum_{j=1}^J \Re v_n^j\biggr)^3 - \sum_{j=1}^J \bigl(\Re v_n^j\bigr)^3\biggr\|_{L^{4/3}_{t,x}(\R \times \R^2)} = 0.
\end{gather}
We observe that
$$
\biggl|\biggl(\sum_{j=1}^J \Re v_n^j + e^{-it\jpn}w_n^J\biggr)^3 - \biggl(\sum_{j=1}^J \Re v_n^j\biggr)^3\biggr|
    \lesssim \bigl|e^{-it\jpn}w_n^J\bigr|^3 + \bigl|e^{-it\jpn}w_n^J\bigr| \biggl|\sum_{j=1}^J \Re v_n^j\biggr|^2,
$$
and so \eqref{E:error from w} follows from H\"older's inequality, \eqref{E:lpd ST norms to zero}, \eqref{E:the dude}, and \eqref{E:vs decoup}.  As
$$
\biggl|\biggl(\sum_{j=1}^J \Re v_n^j\biggr)^3 - \sum_{j=1}^J \bigl(\Re v_n^j\bigr)^3\biggr| \lesssim \mathop{\sum_{1 \leq j_1,j_2,j_3 \leq J}}_{j_1 \neq j_3} \bigl|\Re v_n^{j_1} \Re v_n^{j_2} \Re v_n^{j_3}\bigr|,
$$
we can use H\"older's inequality together with \eqref{E:vs decoup} and \eqref{E:bounded ST norm} to see that \eqref{E:error from parentheses} is true, even without sending $J\to\infty$.

Finally, we complete the proof of \eqref{E:uniform st bounds} by bounding the $L^{\infty}_tH^{1/2}_x$ norm.  By the Strichartz inequality, \eqref{E:initial H1 ok},
and then \eqref{E:H1 vn}, \eqref{E:uniform A}, and \eqref{E:nearly nlkg},
\begin{align*}
\limsup_{J \to \infty} & \limsup_{n \to \infty} \|V_n^J\|_{L^{\infty}_tH^{1/2}_x} \\
    &\lesssim \limsup_{J \to \infty} \limsup_{n \to \infty} \Bigl\{\|v_n(0)\|_{H^{1}_x} + \|\Re V_n^J \|_{L^{4}_{t,x}}^3 + \|\jpn E_n^J\|_{L^{4/3}_{t,x}} \Bigr\} < \infty.
\end{align*}
This completes the proof of \eqref{E:uniform st bounds} and so also the lemma.
\end{proof}

By Lemma~\ref{L:nlpd stability setup}, we may apply Proposition~\ref{P:stability} to conclude that in Case II, $v_n$ is defined globally and $S_{\R}(v_n) \lesssim_{E_c} 1$ for
$n$ sufficiently large.  This contradicts \eqref{E:infinte S} and so Case~II cannot occur.  Tracing back, we see that the only possibility is Case~IC, and
so Proposition~\ref{P:palais smale} is proved.
\end{proof}

Now we prove the existence of a minimal-energy, almost periodic blowup solution to \eqref{nlkg}.

\begin{proof}[Proof of Theorem~\ref{T:reduct}]
By the definition of the critical energy and Corollary~\ref{Cor:gwp}, there exists a sequence $u_n:\R \times \R^2 \to \R$ of global solutions to \eqref{nlkg}
with $E(u_n) \leq E_c$ (and $M(u_n(0)) < M(Q)$ in the focusing case), $\lim_{n \to \infty} E(u_n) = E_c$, and $\lim_{n \to \infty} S_{\R}(u_n) = \infty$.
For each $n$, we choose $t_n$ so that $S_{\leq t_n}(u_n) = S_{\geq t_n}(u_n)$.  By time-translation invariance, we may assume that $t_n \equiv 0$.  We thus have
$$
\lim_{n \to \infty} S_{\leq 0}(u_n) = \lim_{n \to \infty} S_{\geq 0}(u_n) = \infty.
$$

By Proposition~\ref{P:palais smale}, after passing to a subsequence, there exist a sequence $\{x_n\} \subset \R^2$ and a pair of functions $(u_0,u_1)$ so that
\begin{equation} \label{E:un to u0u1}
(T_{x_n}u_n(0),T_{x_n}\partial_t u_n(0)) \to (u_0,u_1), \qtq{strongly in $H^1_x \times L^2_x$.}
\end{equation}
The limit then satisfies $E(u_0,u_1)  = E_c$ (and $M(u_0) \leq 2E_c< M(Q)$ in the focusing case).  By Corollary~\ref{Cor:gwp}, there exists a global solution $u:\R \times \R^2 \to \R$
to \eqref{nlkg} with initial data $u(0) = u_0$ and $\partial_t u(0) = u_1$, satisfying
\begin{equation} \label{E:H1 bounds}
\|u\|_{L^{\infty}_t H^1_x} \lesssim E(u_0,u_1).
\end{equation}

We will show that this solution $u$ satisfies the conclusions of the theorem; it remains to be seen that $u$ blows up forward and backward in time and is almost periodic modulo translations.
If $S_{\geq 0}(u) < \infty$, then by \eqref{E:un to u0u1} and \eqref{E:H1 bounds}, we may apply Proposition~\ref{P:stability} to conclude that
$$
\lim_{n \to \infty} S_{\geq 0} (u_n) < \infty,
$$
a contradiction.  Therefore $u$ must blow up forward in time, and by a similar argument, $u$ must blow up backward in time as well.

Finally, for almost periodicity modulo translations, we observe that if $\{t_n'\} \subset \R$ is any sequence, we have
$$
S_{\geq 0} (u(\cdot + t_n')) \equiv S_{\leq 0} (u(\cdot + t_n')) \equiv \infty
$$
and so by Proposition~\ref{P:palais smale}, a subsequence of $(u(t_n'),\partial_t u(t_n'))$  converges in $H^1_x \times L^2_x$ modulo translations.
Thus, the orbit $\{(u(t),\partial_t u(t)): t \in \R\}$ is precompact modulo translations.  By the Arzel\`a--Ascoli Theorem, this
is equivalent to $u$ being almost periodic modulo translations in the sense of Definition~\ref{D:apmt}.  This completes the proof of the theorem.
\end{proof}


\section{Death of a soliton}


In this section, we will preclude the soliton-like solution, thus concluding the proof of Theorem~\ref{T:ST bounds}.  More precisely, we will prove

\begin{theorem}[No soliton]  \label{T:kill}
There are no minimal-energy blowup solutions to \eqref{nlkg} that are soliton-like in the sense of Theorem~\ref{T:reduct}.
\end{theorem}

To prove this theorem, we will argue by contradiction.  Let $u:\R\times\R^2\to \R$ be a soliton-like solution, that is, a minimal-energy blowup solution that is almost periodic modulo
translations (and satisfies $M(u(0))<M(Q)$ in the focusing case).  Then, invoking \eqref{E:coercive E} in the focusing case,
\begin{align}\label{E controls H1}
\|u\|_{L_t^\infty H^1_x}^2 + \|u_t\|_{L_t^\infty L_x^2}^2 \leq 4E(u).
\end{align}
By Corollary~\ref{C:rest mass}, Remark~\ref{R:3:10}, and the minimality of $u$ as a blowup solution, we must have that the momentum of $u$ is zero:
\begin{equation} \label{E:P=0}
P(u) = 0.
\end{equation}

Our next step will be to use \eqref{E:P=0} to control the motion of $x(t)$, which we do in the manner of \cite{DHR,KenigMerle,Berbec}.

\begin{lemma}[Controlling $x(t)$] \label{L:x(t)}
The spatial center function of $u$ satisfies $|x(t)| = o(t)$ as $|t| \to \infty$.
\end{lemma}

\begin{proof}
By spatial-translation invariance, we may assume that $x(0) = 0$.  We argue by contradiction.  If the conclusion of the lemma did not hold, then there would exist $\delta > 0$
and a sequence $t_n \to \pm \infty$ such that
$$
|x(t_n)| > \delta|t_n|.
$$
Without loss of generality, we may assume that $t_n \to \infty$ and that
$$
 |x(t)| \leq |x(t_n)| \qtq{for all} 0\leq t \leq t_n.
$$

Now let $\eta>0$ be a small constant to be chosen later.  By Remark~\ref{R:conc E}, there exists $C(\eta)>0$ such that
\begin{align}\label{E:conc E}
\sup_{t\in \R} \int_{|x-x(t)|>C(\eta)} |u(t,x)|^2 + |\nabla u(t,x)|^2 + |u_t(t,x) |^2 + |u(t,x)|^4\, dx \leq \eta.
\end{align}
We define
$$
R_n := C(\eta) + |x(t_n)|.
$$
Finally, let $\phi$ be a smooth function with $\phi(r)=1$ for $r \leq 1$ and $\phi(r) = 0$ for $r \geq 2$ and define an approximation to $x(t)$ by
$$
X_{R_n}(t) := \int_{\R^2} x \phi(\tfrac{|x|}{R_n})e_u(t,x)\, dx,
$$
where $e_u$ denotes the energy density of $u$:
\begin{equation*}
e_u := \tfrac12 |u|^2 + \tfrac12|\nabla u|^2 + \tfrac12|u_t |^2 + \tfrac{\mu}4 |u|^4.
\end{equation*}

For each $n$, by the triangle inequality, \eqref{E controls H1}, and \eqref{E:conc E}
\begin{align*}
|X_{R_n}(0)| &\leq \int_{|x| \leq C(\eta)} |x||e_u(0,x)|\, dx + \int_{C(\eta)\leq |x|\leq 2R_n} |x| |e_u(0,x)|\, dx \\
&\lesssim C(\eta)E(u)+ \eta R_n.
\end{align*}
On the other hand, by the triangle inequality followed by \eqref{E:conc E} we also have
\begin{align*}
|X_{R_n}(t_n)|
&\geq |x(t_n)|E(u) - \int_{|x-x(t_n)| \leq C(\eta)} |x-x(t_n)| \phi(\tfrac{|x|}{R_n})|e_u(t_n)|\, dx\\
&\qquad - \int_{|x-x(t_n)| \geq C(\eta)} |x-x(t_n)| \phi(\tfrac{|x|}{R_n})|e_u(t_n)|\, dx\\
&\qquad -|x(t_n)| \int_{\R^2} \bigl[1-\phi(\tfrac{|x|}{R_n})\bigr]|e_u(t_n)|\, dx \\
&\geq |x(t_n)|\bigl[E(u) - 4\eta\bigr]- C(\eta)\bigl[2E(u) + 2\eta\bigr].
\end{align*}
Thus, taking $\eta$ sufficiently small compared to $E(u)$ we get
\begin{equation} \label{E:X(t)-X(0)}
\bigl|X_{R_n}(t_n) - X_{R_n}(0)\bigr| \gtrsim_{E(u)} |x(t_n)| - C(\eta).
\end{equation}

To derive a contradiction, we now seek an upper bound on the left-hand side of \eqref{E:X(t)-X(0)}.  A computation using \eqref{E:P=0} shows that
\begin{align*}
\partial_t X_{R_n}(t)&= \int_{\R^2} \Bigl[1-\phi(\tfrac{|x|}{R_n})\bigr] u_t \nabla u \, dx - \int_{\R^2} \tfrac{x}{|x|R_n}\phi'(\tfrac{|x|}{R_n}) u_t x\cdot\nabla u \, dx.
\end{align*}
Thus by \eqref{E:conc E},
\begin{equation} \label{E:X'(t)}
|\partial_tX_{R_n}(t)| \lesssim \eta.
\end{equation}

Combining \eqref{E:X(t)-X(0)} with \eqref{E:X'(t)} and the fundamental theorem of calculus, we obtain
$$
\eta t_n \gtrsim |X_{R_n}(t_n)-X_{R_n}(0)| \gtrsim_{E(u)} |x(t_n)| - C(\eta) \gtrsim_{E(u)} \delta t_n - C(\eta).
$$
Choosing $\eta$ sufficiently small (depending on $\delta$ and $E(u)$) and then choosing $n$ sufficiently large, we derive a contradiction.
\end{proof}

We are now in a position to complete the proof of Theorem~\ref{T:kill}.  We will use a virial-type argument.

Let $\eta_1>0$ and $\eta_2>0$ be small constants to be chosen later.  By Lemma~\ref{L:x(t)}, there exists $T_0=T_0(\eta_1)>0$ such that
\begin{align}\label{o(t)}
|x(t)|\leq \eta_1 t \qtq{for all} t\geq T_0.
\end{align}
By Remark~\ref{R:conc E}, there exist $C(\eta_1)>0$ and $C(\eta_2)>0$ such that
\begin{align}\label{E:conc E2}
\sup_{t\in \R} \int_{|x-x(t)|>C(\eta_1)} |u(t,x)|^2 + |\nabla u(t,x)|^2 + |u_t(t,x) |^2 + |u(t,x)|^4\, dx \leq \eta_1
\end{align}
and
\begin{align}\label{small freq small mass}
\sup_{t\in \R} \int_{|\xi|<1/C(\eta_2)} |\hat u(t,\xi)|^2 \, d\xi\leq \eta_2.
\end{align}
Using Plancherel and \eqref{small freq small mass}, we find
\begin{align}\label{FINITE mass}
\int_{\R^2} |u(t,x)|^2\, dx
&= \int_{|\xi|<1/C(\eta_2)} |\hat u(t,\xi)|^2 \, d\xi+ \int_{|\xi|\geq 1/C(\eta_2)} |\hat u(t,\xi)|^2\, d\xi \notag\\
&\leq \eta_2 + C(\eta_2)^2 \int_{\R^2} |\nabla u(t,x)|^2\, dx.
\end{align}

With $\phi$ as in the proof of Lemma~\ref{L:x(t)} and $0<\eps<1<R$ to be specified later, we define
$$
Z_{R}(t) = -\int_{\R^2}  \phi\bigl(\tfrac{|x|}{R}\bigr) u_t(t,x) x \cdot \nabla u(t,x) \, dx - (1-\eps) \int_{\R^2}u_t(t,x) u(t,x)\, dx.
$$
Note that by Cauchy--Schwarz and \eqref{E controls H1},
\begin{equation} \label{E:bound MR}
|Z_R(t)| \lesssim RE(u)\lesssim_u R.
\end{equation}

On the other hand, a computation establishes
\begin{align*}
\partial_t Z_R(t) &= \eps\bigl[ \|u(t)\|_{H^1_x}^2 +\|u_t\|_{L_x^2}^2 \bigr] + (1-2\eps) \int_{\R^2}|\nabla u(t)|^2+\tfrac{\mu}2|u(t)|^4 \, dx \\
&- 2\eps\int_{\R^2} |u(t)|^2\, dx-\int_{\R^2}\Bigl[1- \phi\bigl(\tfrac{|x|}{R}\bigr)\Bigr]\bigl[ |u_t(t)|^2- |u(t)|^2-\tfrac{\mu}2 |u(t)|^4 \bigr]\, dx\\
& +\int_{\R^2} \tfrac{|x|}{2R}\phi'\bigl(\tfrac{|x|}{R}\bigr)\bigl[ |u_t(t)|^2- |\nabla u(t)|^2-|u(t)|^2-\tfrac{\mu}2 |u(t)|^4 \bigr]\, dx\\
&+\int_{\R^2} \tfrac{1}{|x|R}\phi'\bigl(\tfrac{|x|}{R}\bigr) [x\cdot \nabla u(t)]^2\, dx.
\end{align*}
Invoking the sharp Gagliardo-Nirenberg inequality, \eqref{E:conc E2}, and \eqref{FINITE mass}, we find
\begin{align*}
|\partial_t Z_R(t)|&\geq \eps\bigl[ \|u(t)\|_{H^1_x}^2 +\|u_t\|_{L_x^2}^2 \bigr] - 2\eps \eta_2 - 10 \eta_1\\
&\quad + \Bigl\{ (1- 2\eps)\Bigl[1 + \mu \tfrac{M(u(t))}{M(Q)}\Bigr] -2\eps C(\eta_2)^2 \Bigr\} \int_{\R^2}|\nabla u(t)|^2\, dx
\end{align*}
for all $T_0\leq t\leq T_1$ and $R=C(\eta_1) + \sup_{t\in [T_0,T_1]}|x(t)|$.  Choosing $\eta_2$ small depending on $u$, then $\eps$ sufficiently small depending on $C(\eta_2)$ (and recalling
that in the focusing case we have $M(u(t))<M(Q)$), and finally $\eta_1$ small enough depending on $\eps$ and $u$, we derive
\begin{align}\label{lower bound}
|\partial_t Z_R(t)|&\gtrsim_u 1 \qtq{for} T_0\leq t\leq T_1 \text{ and } R=C(\eta_1) + \sup_{t\in [T_0,T_1]}|x(t)|.
\end{align}

Combining the fundamental theorem of calculus with \eqref{E:bound MR} and \eqref{lower bound}, and then invoking \eqref{o(t)}, we find
$$
 T_1-T_0 \lesssim_u C(\eta_1) + \eta_1 T_1 \qtq{for all} T_1>T_0.
$$
Choosing $\eta_1$ small depending on $u$ and $T_1$ sufficiently large, we derive a contradiction.

This completes the proof of Theorem~\ref{T:kill}.
\qed


\section{Finite time blowup}\label{S:blowup}


In this section we employ the method of Payne and Sattinger \cite{PayneSattinger} to prove Theorem~\ref{T:blowup}, whose statement we now repeat:

\begin{theorem}[Blowup]
Let $u$ be a maximal-lifespan solution to \eqref{nlkg} in the focusing case with initial data obeying
$$
E(u) < E(Q) \qtq{and} M(u(0)) > M(Q).
$$
Then the solution $u$ blows up in finite time in at least one time direction.
\end{theorem}

\begin{proof}
Let $M(t):=M(u(t))=\int |u(t,x)|^2\,dx$.  By part (ii) of Proposition~\ref{P:coercive}, we know that $M(t) > M(Q)$ (and so non-vanishing) and also that
\begin{equation*}
M''(t) > 6\int_{\R^2} |u_t(t,x)|^2\,dx.
\end{equation*}
Combining this with the Cauchy--Schwarz inequality we obtain
$$
\bigl[ M'(t) \bigr]^2 \leq 4 \biggl(\int_{\R^2} |u(t,x)|^2\,dx\biggr) \biggl(\int_{\R^2} |u_t(t,x)|^2\,dx\biggr) < \tfrac23 M(t) M''(t)
$$
and hence,
$$
\partial_{tt} \; M(t)^{-1/2}  = - \frac{2M''(t)M(t)-3M'(t)^2}{4M(t)^{5/2}} < 0.
$$
This says that $M(t)^{-1/2}$ is strictly concave, which is inconsistent with $M(t)^{-1/2}$ being a positive function defined on the whole real line.
In particular, if $M'(0)\geq 0$ then the solution must blow up in finite time in the future, while if $M'(0)\leq 0$, it must blow up in finite negative time.
\end{proof}


\end{document}